\documentclass[11pt,onecolumn]{article}
\usepackage[utf8]{inputenc}
\usepackage{latexsym}
\usepackage{framed,multirow,amsmath,amsthm,amssymb,amsfonts}
\usepackage[normalem]{ulem}
\usepackage{tikz}
\usetikzlibrary{arrows.meta,positioning}
\usepackage{algorithm}
\usepackage{algpseudocode}
\usepackage{subcaption}
\usepackage{booktabs}
\usepackage{newunicodechar}
\newunicodechar{，}{,}

\usepackage{booktabs} 
\usepackage{diagbox}  
\usepackage{array}    

\usepackage{tabularx}
\usepackage{array}
\usepackage{makecell}
\usepackage{empheq}
\usepackage{url}
\usepackage{listings}
\usepackage{float}
\usepackage[shortlabels]{enumitem}
\usepackage[dvipsnames]{xcolor}
\usepackage{threeparttable}

\usepackage{titlesec}
\titlespacing*{\section}{0pt}{1.6ex}{1ex}
\titlespacing*{\subsection}{0pt}{1.2ex}{0.8ex}
\titlespacing*{\subsubsection}{0pt}{1.0ex}{0.6ex}

\titleformat{\section}
{\normalfont\large\bfseries}
{\thesection}{1em}{}
\titlespacing*{\section}
{0pt}{2.5ex plus 0.5ex minus 0.2ex}{1.2ex plus 0.2ex}

\titleformat{\subsection}
{\normalfont\normalsize\bfseries}
{\thesubsection}{1em}{}
\titlespacing*{\subsection}
{0pt}{2ex plus 0.5ex minus 0.2ex}{0.8ex plus 0.2ex}

\titleformat{\subsubsection}
{\normalfont\normalsize\bfseries}
{\thesubsubsection}{1em}{}
\titlespacing*{\subsubsection}
{0pt}{1.5ex plus 0.5ex minus 0.2ex}{0.5ex plus 0.2ex}

\definecolor{sintefblue}{RGB/cmyk}{0,60,101/1,.57,0,.4}
\definecolor{sintefgrey}{RGB/cmyk}{235,235,230/0,0,0,.1}
\colorlet{sintefgray}{sintefgrey}
\definecolorset{RGB/cmyk}{sintef}{}{lightgreen, 205,250,225/.23, 0,.20, 0;%
                                    green,       20,185,120/.73, 0,.67, 0;%
                                    darkgreen,    0, 70, 40/.93,.43,.92,.52}
\definecolorset{RGB/cmyk}{sintef}{}{yellow, 200,155,20/20, 36,98, 8;%
                                    red,    190, 60,55/19, 86,77, 8;%
                                    lilla,  120,  0,80/48,100,27,31}
\definecolorset{HTML}{sintef}{}{cyan,      22A7E5;%
                                magenta,   EC008C;%
                                lightgrey, D8D0C7}
\colorlet{sinteflightgray}{sinteflightgrey}



\newtheorem{theorem}{Theorem}[section]

\newtheorem{definition}{Definition}[section]

\newtheorem{lemma}{Lemma}[section]

\newtheorem{remark}{Remark}[section]

\newtheorem{proposition}{Proposition}[section]

\newtheorem{corollary}{Corollary }[section]

\newtheorem{assumption}{Assumption }[section]

\usepackage{bm}

\newcommand\abs[1]{\left|#1\right|}
\newcommand\norm[1]{\left\lVert#1\right\rVert}

\newcommand{\T}{\mathrm{T}}

\newcommand{\mE}{\mathcal{E}}

\newcommand{\Tp}{{\bm{p}}}

\newcommand{\dt}{\mathrm{d}t}
\newcommand{\etaxv}{\eta_{\bm{x},\bm{v}}}
\newcommand{\xixv}{\xi_{\bm{x},\bm{v}}}
\newcommand{\detaxv}{\dot{\eta}_{\bm{x},\bm{v}}}
\newcommand{\dxixv}{\dot{\xi}_{\bm{x},\bm{v}}}
\newcommand{\ddetaxv}{\ddot{\eta}_{\bm{x},\bm{v}}}

\newcommand{\grad}{\mathrm{grad}\,}
\newcommand{\hess}{\mathrm{Hess\,}}

\makeatletter
\@addtoreset{equation}{section}
\@addtoreset{figure}{section}
\@addtoreset{table}{section}
\@addtoreset{lemma}{section}
\makeatother

\titleformat{\part}
{\large\bfseries}
{Part \thepart}
{1em}
{}

\definecolor{newcolor}{rgb}{.2,.3,.7}

\definecolor{newred}{rgb}{0.59, 0.0, 0.09}
\definecolor{newblue}{rgb}{.2,.3,.7}
\usepackage[
  colorlinks=true,
  linkcolor=newblue,
  citecolor=newred,
  urlcolor=newblue
]{hyperref}

\begin{document}

\hypersetup{
  colorlinks=true,
  linkcolor=newblue,
  citecolor=newred,
  urlcolor=newblue
}

\begin{center}
\renewcommand{\thefootnote}{\fnsymbol{footnote}}

{\Large \bf Optimization on Affine-Transversal Hilbert Submanifolds:\\ Part I -- Theoretical Foundations}\\

\bigskip
\medskip

{\bf Yongcun Song}\footnote{\parbox[t]{15.5cm}{
Division of Mathematical Sciences, School of Physical and Mathematical Sciences,
Nanyang Technological University, 21 Nanyang Link, 637371, Singapore. This author was supported by a Start-Up Grant from Nanyang Technological University. Email: {\color{newblue}yongcun.song@ntu.edu.sg}}\vspace{0.5em}}
\quad
{\bf Luhao Xue}\footnote{\parbox[t]{15.5cm}{
Department of Mathematics, The University of Hong Kong, Hong Kong, China. This author was supported by the Hong Kong PhD Fellowship Scheme. Email: {\color{newblue}luhxue214@connect.hku.hk}}\vspace{0.5em}}
\quad
{\bf Xiaoming Yuan}\footnote{\parbox[t]{15.5cm}{
Department of Mathematics, The University of Hong Kong, Hong Kong, China. This author was supported by the Croucher Senior Fellowship. Email: {\color{newblue}xmyuan@hku.hk}}\vspace{0.5em}}
\quad
{\bf Hangrui Yue}\footnote{\parbox[t]{15.5cm}{
School of Mathematical Sciences, Nankai University, Tianjin 300071, China. This author was supported by the National Natural Science Foundation of
China (No. 12301399) and the Natural Science Foundation of Tianjin (No. 25JCZDJC00970). Email: {\color{newblue}yuehangrui@gmail.com}}}

\bigskip
{\today}

\renewcommand{\thefootnote}{\arabic{footnote}}
\setcounter{footnote}{0}

\end{center}

\bigskip

\noindent\rule{\textwidth}{0.4pt}
\par\vspace{6pt}
\noindent\textbf{Abstract:} In this paper, we establish the theoretical foundations for the generic optimization problem in a Hilbert space whose feasible set is an affine-transversal Hilbert submanifold given by the intersection of a nonlinear manifold and an affine subspace. We develop the geometric and analytical toolkit, such as the tangent-space characterization, projection operators, and implicit retraction operators, to essentially ensure the feasibility of iterates for algorithmic design. We also derive weak-form expressions for the derivatives of lifted objective functionals. In particular, we introduce a projection-induced Riemannian metric whose induced norm is uniformly equivalent to the ambient norm, under which the projection-induced gradient becomes an exact Riemannian gradient with an explicit formula. This construction replaces the implicit tangent-space Riesz representation underlying classical Riemannian optimization with directly computable operator evaluations, yielding a practical variable-metric framework for algorithmic design while preserving the geometric structure required for convergence analysis. With these theoretical foundations, it becomes possible to apply standard techniques in Euclidean spaces to design algorithms with strictly feasible iterates for the optimization problem on an affine-transversal Hilbert submanifold. We also propose the theoretical frameworks for algorithmic design on affine-transversal Hilbert submanifolds by showcasing the Riemannian line-search and trust-region algorithms with rigorous convergence analysis.
\par\vspace{5pt}
\noindent\textbf{Keywords:}\ \ Optimization, PDEs and variational models， differential geometry， infinite-dimensional manifolds, affine-transversal submanifolds, projection-induced Riemannian metric
\par\vspace{5pt}
\noindent\textbf{AMS subject classifications:}\ \ 65K10, 58B20, 90C48, 49M37, 35Q93, 46T05, 58C20.
\par\vspace{4pt}
\noindent\rule{\textwidth}{0.4pt}
\par\vspace{12pt}










\setcounter{tocdepth}{2}

	
\section{Introduction}\label{sec:intro}

\subsection{Model}\label{subsec:model}

\noindent Let \(\mathcal{E}\) and \(\mathcal{Y}\) be Hilbert spaces endowed with the inner products \((\cdot,\cdot)_{\mathcal{E}}\) and \((\cdot,\cdot)_{\mathcal{Y}}\), and the induced norms \(\|\cdot\|_{\mathcal{E}}\) and \(\|\cdot\|_{\mathcal{Y}}\), respectively. We consider a generic optimization problem whose feasible set is given by the intersection of a nonlinear manifold in \(\mathcal{E}\) and an affine subspace induced by linear equality constraints
\begin{equation}\label{J}
    \begin{aligned}
        \min&\quad J(\bm{x})\\
        \mbox{s.t.}&\quad \bm{x}\in\mathcal{M}^B := \mathcal{M} \cap \{\bm{x}\in\mathcal{E}\mid B\bm{x}=c\}.
    \end{aligned}
\end{equation}

\noindent More specifically, \(\mathcal{M}\) is assumed to be a \(\mathcal{C}^2\) (twice continuously differentiable) embedded submanifold of \(\mathcal{E}\), \(B:\mathcal{E}\to\mathcal{Y}\) a surjective bounded linear operator,  and \(c\in\mathcal{Y}\) prescribed. Also, the objective functional \(J:\mathcal{E}\to\mathbb{R}\) is assumed to be \(\mathcal{C}^2\), bounded from below, and coercive (which implies that all sublevel sets of \(J\) are bounded in \(\mathcal{E}\)). We further assume that the gradient mapping \(\grad J:\mathcal{E}\to\mathcal{E}\) and the Hessian mapping \(\hess J:\mathcal{E}\to\mathcal{L}(\mathcal{E};\mathcal{E})\) are both bounded mappings, in the sense that each maps bounded subsets of its domain to bounded subsets of its codomain. To alleviate the notation, we denote by \(J^B(\bm{x}):= J|_{\mathcal{M}^B}(\bm{x})\) the restriction of \(J(\bm{x})\) intrinsically on \(\mathcal{M}^B\), and rewrite \eqref{J} as
\begin{equation}\label{J^B}
    \min_{\bm{x}\in \mathcal{M}^B} J^B(\bm{x}).
\end{equation}
Furthermore, the set \(\mathcal{M}^B\) is assumed to be nonempty, and \(B|_{\mathcal M}:\mathcal M\to\mathcal{Y}\), which denotes the restriction of \(B\) to \(\mathcal M\), is assumed to be a submersion at every point of \(\mathcal M^B\).

\subsection{Applications}\label{sec:applications}

\noindent The conceptual model \eqref{J^B} is general enough to capture a wide range of specific applications in areas such as physics, chemistry, materials science, and beyond. To mention a few, in continuum mechanics and computer graphics, the modeling of inextensible elastic rods \cite{bisshopp1945large,bourgat1980large,glowinski1989augmented}, the simulation of hair, cloth, and shell structures \cite{audoly2010elasticity,daviet2023interactive}, and nematic liquid crystals \cite{de1993physics,hardt1986existence,lin1989nonlinear,virga2018variational}, can all be modeled as specific cases of \eqref{J^B}. The model \eqref{J^B} also appears in modern large-scale and real-time simulation scenarios, such as digital twins and interactive environments, where physics-based constrained models play an important role \cite{rasheed2020digital,tao2018digital}. In quantum mechanics and materials science, the model \eqref{J^B} appears for the computation of Bose--Einstein condensates (BECs) \cite{danaila2017computation,griffin1996bose,lieb2006derivation,tian2020ground} and the micromagnetic simulations based on the Landau--Lifshitz equation \cite{alouges1992global,cimrak2007survey,guo2008landau,lakshmanan2011fascinating,visintin1985landau}. It is also relevant to first-order Hamilton--Jacobi type models, such as Eikonal equations arising in geometrical optics, front propagation, shape-from-shading, and seismic traveltime computation \cite{crandall1983viscosity,rawlinson2004wave,rouy1992viscosity,sethian1996fast,tsitsiklis1995efficient,vidale1988finite}. The model \eqref{J^B} is also used to study computational biology problems such as the topological analysis of DNA supercoiling \cite{goriely2017mathematics}.

\noindent For such concrete applications, the nonlinear manifold \(\mathcal{M}\) in \eqref{J^B} typically encodes intrinsic geometric constraints arising from physical invariants, such as pointwise unit-norm, orthogonality, or isometry constraints, while the linear equality constraints encode essential additional requirements, such as boundary conditions, interface conditions, or global conservation laws \cite{audoly2010elasticity,de1993physics,goriely2017mathematics,hardt1986existence,lin1989nonlinear,virga2018variational,zhao2023systematic}. To see how these concrete applications fit into \eqref{J^B}, we illustrate two representative examples.

\noindent The first example is the simulation of inextensible elastic rods in \cite{bourgat1980large}. Let \(\bm{y}(s)\in H^2(0,L;\mathbb{R}^3)\) denote the rod centerline, parameterized by arc length \(s\). Introducing \(\bm{x}=\bm{y}'\), one may take the state space to be \(\mathcal{E}=H^1(0,L;\mathbb{R}^3)\). The functional \(J\) represents the total potential energy, consisting of the elastic bending energy and the work of external forces, and is given by
\begin{equation}\label{eq:rod_model}
\frac{1}{2} \int_0^L EI |\bm{x}'(s)|^2 \, \mathrm{d}s
- \int_0^L \bm{f}(s)\cdot \left(\bm{y}(0)+\int_0^s \bm{x}(t)\,\mathrm{d}t\right)\,\mathrm{d}s.
\end{equation}
The geometric inextensibility condition gives rise to the nonlinear manifold
\[
\mathcal{M}
=
\left\{
\bm{x}\in \mathcal{E}
\;\middle|\;
|\bm{x}(s)|=1 \text{ for } s\in[0,L],\;
\bm{x}(0)=\bm{e}_0,\;
\bm{x}(L)=\bm{e}_L
\right\}.
\]
Here, \(EI > 0\) denotes the flexural rigidity of the rod, and \(\bm{f}\) is the linear density of the external loads. The vectors \(\bm{e}_0, \bm{e}_L \in \mathbb{R}^3\) are the prescribed unit tangent vectors to the centerline at the endpoints \(P_0\) and \(P_L\), respectively, with \(|\bm{e}_0| = |\bm{e}_L| = 1\). The linear constraint \(B\bm{x} = c\), defined by \(\int_0^L \bm{x}(s)\,\mathrm{d}s = c\), dictates the relative position of the clamped endpoints. Specifically, the integral operator \(B\) computes the total displacement along the rod, while \(c\) prescribes the fixed spatial vector pointing from \(P_0\) to \(P_L\).

\noindent The second example concerns the Eikonal-constrained traveltime reconstruction. Let \(\Omega\subset\mathbb{R}^d\) be bounded, let \(\bm{u}:\Omega\to\mathbb{R}\) denote the phase or first-arrival traveltime, and let \(\bm{\sigma}:\Omega\to\mathbb{R}\) denote the slowness field. To obtain a smooth Sobolev-space representative of Eikonal-constrained inverse problems within the abstract framework \eqref{J^B}, consider
\[
\mathcal{E}
=
H^s(\Omega)\times H^{s-1}(\Omega),
\qquad
s>\frac{d}{2}+1,
\]
and assume that the slowness field satisfies $\bm{\sigma}(x)> \bm{\sigma}_{\min}>0$ for all $x\in\Omega$. Motivated by variational formulations of Eikonal inverse problems and traveltime tomography
\cite{deckelnick2011numerical,glowinski2015penalization,leung2006adjoint}, we consider the Tikhonov-type functional
\begin{equation}\label{eq:eikonal_model}
J(\bm{u},\bm{\sigma})
=
\frac12
\|\mathcal{T}\bm{u}-d\|_{\mathcal Y}^2
+
\frac{\alpha}{2}
\|\bm{\sigma}-\bm{\sigma}_{\rm ref}\|_{H^{s-1}(\Omega)}^2,
\end{equation}
where
\(
\mathcal T:H^s(\Omega)\to\mathcal Y
\)
is a bounded linear observation operator, \(d\in\mathcal Y\) denotes the measured arrival-time data, \(\bm{\sigma}_{\rm ref}\) is a reference slowness field, and \(\alpha>0\) is a regularization parameter. The Eikonal relation \(|\nabla \bm{u}|=\bm{\sigma}\) gives the nonlinear equality constraint
\[
\mathcal{M}=\left\{(\bm{u},\bm{\sigma})\in\mathcal{E}\;\middle|\;|\nabla \bm{u}(x)|=\bm{\sigma}(x) \text{ in } \Omega,\quad \bm{\sigma}> \bm{\sigma}_{\mathrm{min}}>0\right\}.
\]
Source or boundary normalization can be imposed by the linear constraint \(B(\bm{u},\bm{\sigma})=c\); for example, \(B(\bm{u},\bm{\sigma})=\gamma_{\Gamma_s}\bm{u}\) and \(c=0\) enforce \(\bm{u}=0\) on a prescribed source boundary \(\Gamma_s\subset\partial\Omega\). If \(\bm{\sigma}\) is fixed, the variable reduces to \(\bm{u}\) and the constraint becomes the fixed-slowness Eikonal equation \(|\nabla \bm{u}|=\bm{\sigma}\). \eqref{eq:eikonal_model} provides a smooth representative instance of \eqref{J^B}; nonsmooth global first-arrival solutions are usually treated in the viscosity-solution sense \cite{crandall1983viscosity,rouy1992viscosity}.

\subsection{Related Works}\label{subsec:related work}

\noindent Despite its broad applicability, the model \eqref{J^B} remains very challenging from both theoretical and numerical perspectives; see, e.g., \cite{alouges1997new,bartels2013approximation,bartels2020numerical,bourgat1980large,glowinski1989augmented,kovnatsky2016madmm,liu2020simple,saad2010numerical,simo1986three}. When discretized counterparts are considered for various applications of \eqref{J^B} arising in PDE and physical domains, a common difficulty is the high stiffness of the underlying energy functionals. Such examples include the bending energy of inextensible elastic rods \cite{audoly2010elasticity,bergou2008discrete}, the Oseen--Frank energy of liquid crystals \cite{de1993physics}, and the Gross--Pitaevskii energy of Bose--Einstein condensates \cite{danaila2017computation}. Hence, solving ill-conditioned Euclidean systems repeatedly becomes inevitable, and the convergence behavior of algorithms designed for the discretized problems may deteriorate as the mesh is refined; see, e.g., \cite{bank1989conditioning,borzi2011computational,graham2006anisotropic}. Moreover, a discretized counterpart of \eqref{J^B} may be largely and structurally simplified, yet it does not necessarily reflect some important geometrical and analytical properties of \eqref{J^B} in the function space, see, e.g., \cite{holst2012geometric,strang1973analysis}\footnote{We also refer to \cite{song2026optimization2} for the inextensible elastic rod problem, in which it is shown that a projection may fail to be orthogonal with respect to the induced inner product of a Hilbert space, even though its discrete counterpart may be orthogonal in Euclidean space.}.

\noindent These facts have inspired some literature working directly on \eqref{J^B} in the function space for algorithmic design, rather than on its discretized counterparts in Euclidean space. In \cite{adler2016constrained}, the classical penalty and Lagrange multiplier methods were proposed for liquid-crystal equilibrium problems; in \cite{bourgat1980large,glowinski1989augmented}, the augmented Lagrangian method was applied to inextensible elastic rod model problems. In \cite{danaila2017computation} and \cite{zhao2023systematic}, the Riemannian gradient and conjugate-gradient methods were applied to the Gross--Pitaevskii ground-state problem and the extreme 3D Euler flows problem, respectively. For the problems considered in \cite{danaila2017computation,zhao2023systematic}, \(\mathcal{M}^B\) can be readily identified as a Hilbert submanifold of \(\mathcal{E}\), and thus the associated geometric operators, such as the retractions and orthogonal projections, admit closed-form expressions through modifications of their counterparts on \(\mathcal{M}\). For the generic model \eqref{J^B}, a Riemannian sequential quadratic programming (SQP) approach, motivated by classical SQP methods, was proposed in~\cite{schiela2021sqp} in the context of the manifold \(\mathcal{M}\), while the affine constraint was relaxed and incorporated into the objective functional.

\subsection{Motivation}\label{subsec:numerical challenges}

\noindent Recall that a strong motivation to consider \eqref{J^B} is because of its broad applications in the PDE and physical domains. For such applications, the feasible set \(\mathcal{M}^B\) usually encodes strict invariants or conservation laws, so any constraint violation may produce physically inconsistent states. For instance, in micromagnetics \cite{lakshmanan2011fascinating} and liquid crystal dynamics \cite{lin1989nonlinear}, the manifold imposes a pointwise unit-norm condition representing a fundamental material property; even infinitesimal violations can lead to nonphysical configurations. Therefore, it becomes crucial to design such an algorithm for \eqref{J^B} that ensures the feasibility of all iterates. Meanwhile, the mandated feasibility requirement excludes the straightforward applications of optimization algorithms in either the standard or manifold contexts. The mentioned relaxation approaches, which commonly treat the manifold \(\mathcal{M}\) and the linear constraints \(B\bm{x}=c\) separately, are not applicable either. Instead, it requires the identification of the feasible directions of each point in \(\mathcal{M}^B\) analytically and geometrically --- which essentially urges the analysis of the retraction and projection operators on \(\mathcal{M}^B\).

\noindent Our target is to establish the theoretical foundations for \eqref{J^B} in the function space, and then provide a comprehensive analytic framework for generalizing the algorithmic design techniques from classical optimization literature to the more complex context of the form \eqref{J^B}.

\subsection{Roadmap}\label{subsec:roadmap}

\noindent Our roadmap for the establishment of theoretical foundations of \eqref{J^B} is to characterize \(\mathcal{M}^B\) as an affine-transversal submanifold of \(\mathcal{E}\) under mild assumptions, construct its intrinsic geometric objects (including tangent spaces, projections, and retractions), derive the derivatives of lifted objective functionals, establish uniform analytic estimates (including retraction radii, Lipschitz continuity properties of lifted objective functionals, and derivative bounds) that hold uniformly on bounded subsets of \(\mathcal{M}^B\), and develop a projection-induced variable-metric framework that facilitates algorithmic design. Ultimately, we lay down the theoretical foundations for algorithmic design at the operator level for \eqref{J^B}.

\noindent First, under some mild conditions (see Assumption~\ref{asmp:Lx invertible}), we show that \(\mathcal{M}^B\) is a Hilbert submanifold, and then derive its tangent space in terms of that of \(\mathcal{M}\) and the operator \(B\). Then, we construct the projection \(\bm{P}_{\bm{x}}^B\) and an implicit \(\mathcal{C}^2\) retraction \(\bm{R}_{\bm{x}}^B\), and derive weak-form expressions for the gradient and Hessian of the pullback objective \(J^B \circ \bm{R}_{\bm{x}}^B\). The projection and retraction operators ensure the feasibility throughout the iteration, while the derivative formulas provide the analytic foundation for algorithmic design.

\noindent Second, under a uniform coercivity assumption (see Assumption~\ref{asmp:Lx bounded below}), we establish the uniform analytic estimates for convergence analysis. These estimates include uniform retraction radii, radial Lipschitz continuity of \(J^B \circ \bm{R}_{\bm{x}}^B\), and uniform bounds on the Riemannian gradient of \(J^B\) and the Hessian of \(J^B \circ \bm{R}_{\bm{x}}^B\). We strengthen the geometric constructions into uniform statements on bounded subsets of \(\mathcal{M}^B\), thereby providing the quantitative foundation for rigorous optimization analysis on the affine-transversal manifold.

\noindent Then, a central component of the framework is the projection-induced variable-metric interpretation developed in Section~\ref{sec:projection-induced gradient}. For the generally non-orthogonal projection \(\bm{P}_{\bm{x}}^B\), a direct projection of the ambient gradient does not, in general, yield the Riemannian gradient under the ambient-induced metric. Indeed, the Riemannian gradient associated with the ambient-induced metric is characterized by a tangent-space Riesz representation problem, whose solution may be computationally expensive. This motivates the introduction of a projection-induced Riemannian metric on \(\T_{\bm{x}}\mathcal{M}^B\). We show that, with respect to this metric, the computable vector $\bm{P}_{\bm{x}}^B(\bm{P}_{\bm{x}}^B)^*\grad J(\bm{x})$ is the exact Riemannian gradient of \(J^B\). In this way, the variable-metric construction provides a fully explicit representation of the Riemannian gradient in terms of $\bm{P_x}^B$, avoiding implicit tangent-space Riesz representation solves. Moreover, by establishing that the projection-induced metric is uniformly equivalent to the ambient-induced metric in the sense of induced norms on bounded subsets, the resulting framework remains compatible with the standard analytical tools used in convergence analysis of first- and second-order methods in Hilbert spaces under the ambient metric. This structure plays an essential role in the subsequent numerical treatment of \eqref{J^B}, motivating, among other developments, the projection-induced CG--Steihaug method.

\noindent Last, with the established geometric and analytical foundations, we focus on the algorithmic design for \eqref{J^B} by leveraging standard optimization techniques in finite-dimensional spaces. We showcase how to design the benchmark Riemannian line-search and trust-region methods (in which a novel projection-induced CG--Steihaug solver is embedded). We reiterate that a distinctive feature of our approach is that the nonlinear manifold constraint \(\mathcal{M}\) and the linear equality constraint \(\{\bm{x}\in\mathcal{E}\mid B\bm{x}=c\}\) are treated as a single coupled geometric object at the operator level, so that the feasibility is guaranteed throughout. At the same time, our framework does not require \(\mathcal{M}^B\) to admit the kind of simple structure for which the projection and retraction are readily available in closed form, as in the special cases discussed in Section~\ref{subsec:related work}; the geometric operators constructed here remain valid even when \(\mathcal{M}^B\) has a more general geometry and the projection \(\bm{P}_{\bm{x}}^B\) is not orthogonal.

\subsection{Organization}

\noindent The rest of this paper is organized as follows. In Section~\ref{sec:geometric-prelim}, some preliminaries are summarized for the Hilbert manifold \(\mathcal{M}\). In Section~\ref{sec:geometry of MB}, some geometric properties of the affine-transversal submanifold \(\mathcal{M}^B\) are derived, and the associated retraction and projection operators are constructed.  In Section~\ref{sec:derivatives computations}, we derive the derivatives of the objective functional, and in Section~\ref{sec:properties}, we establish the uniform analytic estimates required for the convergence analysis of optimization algorithms on \(\mathcal{M}^B\). In Section~\ref{sec:projection-induced gradient}, we introduce the projection-induced Riemannian metric and the corresponding projection-induced gradient. Finally, in Sections~\ref{sec:algorithmLS} and~\ref{sec:algorithmTR}, we showcase how to design the Riemannian line-search and trust-region methods, respectively, together with rigorous convergence analysis. 

\subsection{Remark}

\noindent To implement the theoretical foundations to design specific algorithms for different applications of \eqref{J^B}, meticulous techniques for analysis and computation are required. The present article constitutes the theoretical part of a two-part study. We leave the application part, including the verification and practical implications of the assumptions introduced herein, the specification of implementable algorithms, as well as the numerical results, to the companion article \cite{song2026optimization2} in which the simulation of large displacements in inextensible elastic rods is discussed in detail.

\section{Preliminaries}\label{sec:geometric-prelim}

\noindent In this section, we collect the geometric notation, operators, and standing assumptions on the manifold \(\mathcal{M}\) that will be used throughout the paper. These assumptions are incorporated into the analytical framework and will be invoked in subsequent development without further mention.

\medskip
\noindent\textbf{Tangent spaces and Riemannian metric.}
Following standard texts~\cite{absil2008optimization,lang2012fundamentals}, the tangent space of \(\mathcal{M}\) at a point \(\bm{x} \in \mathcal{M}\) is given by
\[
\mathrm{T}_{\bm{x}}\mathcal{M} := \left\{ \frac{d}{d\alpha} \bm{\gamma}(\alpha) \Big|_{\alpha=0} \,\bigg|\, \bm{\gamma} \colon I \to \mathcal{M} \text{ is \(\mathcal{C}^1\) and } \bm{\gamma}(0) = \bm{x} \right\},
\]
where \(I \subseteq \mathbb{R}\) is an open interval containing \(0\). The tangent space $\mathrm{T}_{\bm{x}}\mathcal{M}$ is endowed with a Riemannian metric $(\cdot, \cdot)_{\mathcal{M}, \bm{x}}$ inherited from the ambient space $\mathcal{E}$ by restriction:
\[
(\bm{u}, \bm{v})_{\mathcal{M}, \bm{x}} = (\bm{u}, \bm{v})_{\mathcal{E}} \quad \text{for all } \bm{u}, \bm{v} \in \mathrm{T}_{\bm{x}}\mathcal{M}.
\]

\noindent Since \(\mathcal{M}\) is a submanifold of \(\mathcal{E}\) equipped with the induced metric, we naturally identify the tangent space \(\T_{\bm{x}}\mathcal{M}\) at any \(\bm{x} \in \mathcal{M}\) with a closed linear subspace of \(\mathcal{E}\). Under this identification, the zero tangent vector \(\bm{0_x} \in \T_{\bm{x}}\mathcal{M}\) coincides with the zero vector \(\bm{0}_{\mathcal{E}}\) of \(\mathcal{E}\). Consequently, we will hereafter use \(\bm{0}_{\mathcal{E}}\) to denote the zero vector of \(\T_{\bm{x}}\mathcal{M}\).

\medskip
\noindent\textbf{Projection operators.}
Throughout our analysis, a linear operator $\bm{F}$ from $\mathcal{E}$ to its closed subspace $\bar{\mathcal{E}}$ is termed a \textit{projection} if it is surjective onto $\bar{\mathcal{E}}$ and satisfies the idempotence property $\bm{F}^2 = \bm{F}$~\cite{rudin1991functional}. An immediate consequence is that $\bm{F}$ acts as the identity map on its range: for any $\bm{v} \in \bar{\mathcal{E}}$, we have $\bm{F}\bm{v} = \bm{v}$.

\noindent We assume the existence of a bounded linear projection operator $\bm{P}_{\bm{x}} \colon \mathcal{E} \to \T_{\bm{x}}\mathcal{M}$ onto the tangent space at each $\bm{x} \in \mathcal{M}$. Note that $\bm{P}_{\bm{x}}$ is not necessarily orthogonal or self-adjoint with respect to $(\cdot, \cdot)_{\mathcal{E}}$. 

\medskip
\noindent\textbf{Retraction mappings.}
For each point \(\bm{x}\in\mathcal{M}\), let \(V_{\bm{x}}\subset \mathrm{T}_{\bm{x}}\mathcal{M}\) be a neighborhood of \(\bm{0}_{\mathcal{E}}\). A \textit{(local) retraction} at \(\bm{x}\) is a mapping \(\bm{R}_{\bm{x}}:V_{\bm{x}}\to \mathcal{M}\) satisfying
\begin{equation}\label{retraction property 1}
\bm{R}_{\bm{x}}(\bm{0}_{\mathcal{E}}) = \bm{x} \quad \hbox{and} \quad D\bm{R}_{\bm{x}}(\bm{0}_{\mathcal{E}})= \mathrm{Id}_{\mathrm{T}_{\bm{x}}\mathcal{M}},
\end{equation}
where \(\mathrm{Id}_{\mathrm{T}_{\bm{x}}\mathcal{M}}\) is the identity operator on \(\mathrm{T}_{\bm{x}}\mathcal{M}\), and \(D\bm{R}_{\bm{x}}\) denotes the Fr\'echet derivative of \(\bm{R}_{\bm{x}}\). When \(V_{\bm{x}}=\mathrm{T}_{\bm{x}}\mathcal{M}\), the retraction is said to be \textit{globally defined}. Throughout this work, we assume that for each $\bm{x}\in \mathcal{M}$, \(\bm{R}_{\bm{x}}\) is a \(\mathcal{C}^{2}\) global retraction on \(\mathcal{M}\) at $\bm{x}$.

\noindent For any given point \(\bm{x} \in \mathcal{M}\), tangent vector \(\bm{v} \in \mathrm{T}_{\bm{x}}\mathcal{M}\), and parameter \(\alpha \in [0, +\infty)\), consider the curve defined by the retraction
\begin{equation}\label{eqn:directional retraction original}
    \bm{c}(\alpha) := \bm{R}_{\bm{x}}(\alpha \bm{v}).
\end{equation}
Since \(\bm{R}_{\bm{x}}\) is of class \(\mathcal{C}^2\), the second-order derivative of \(\bm{c}\) at \(\alpha=0\) is determined by the second-order derivative of \(\bm{R}_{\bm{x}}\) at \(\bm{0}_{\mathcal{E}}\). By Schwarz's theorem, this second-order derivative is represented by a symmetric bilinear mapping. Consequently, we can express the curve's second-order derivative through a mapping \(\ell: \bm{x}\in\mathcal{M} \mapsto \ell_{\bm{x}}\in\mathcal{L}^2_{\mathrm{Sym}}(\mathrm{T}_{\bm{x}}\mathcal{M}; \mathcal{E})\),  satisfying
\begin{equation}\label{second-order derivative Rx}
    \left.\frac{d^2}{d\alpha^2} \bm{c}(\alpha)\right|_{\alpha=0} = \ell_{\bm{x}}(\bm{v}, \bm{v}).
\end{equation}
Here, \(\mathcal{L}^2_{\mathrm{Sym}}(\mathrm{T}_{\bm{x}}\mathcal{M}; \mathcal{E})\) denotes the space of symmetric bilinear mappings from \(\mathrm{T}_{\bm{x}}\mathcal{M} \times \mathrm{T}_{\bm{x}}\mathcal{M}\) to \(\mathcal{E}\). We further assume that $\ell$ is a bounded mapping in the sense that, for every bounded subset
$X\subset\mathcal{M}$, it holds that
\[
    \sup_{\bm{x}\in X}
    \|\ell_{\bm{x}}\|
    <\infty.
\]

\section{Geometry of \texorpdfstring{$\mathcal{M}^B$}{M\^{}B}}\label{sec:geometry of MB}

\noindent This section develops the geometric structures used for optimization on \(\mathcal{M}^B\). We first prove that \(\mathcal{M}^B\) is a Hilbert submanifold of \(\mathcal{M}\). To this end, we formulate the transversality condition for the intersection of \(\mathcal{M}\) with the affine constraint set and show that it is equivalent to the submersion property of \(B|_{\mathcal{M}}\). Under Assumption~\ref{asmp:Lx invertible}, this submersion property follows constructively from the invertibility of \(L_{\bm{x}}:=B\bm{P}_{\bm{x}}B^*\), which yields a bounded right inverse of \(\T_{\bm{x}}(B|_{\mathcal{M}})=B|_{\T_{\bm{x}}\mathcal{M}}\). Consequently, \(\mathcal{M}^B\) is an affine-transversal Hilbert submanifold of \(\mathcal{E}\), with tangent space
\[
\T_{\bm{x}}\mathcal{M}^B = \T_{\bm{x}}\mathcal{M}\cap\ker(B), \qquad \bm{x}\in\mathcal{M}^B.
\]

\noindent We then construct an implicit local retraction \(\bm{R}_{\bm{x}}^B\) by applying the implicit function theorem to the ambient retraction \(\bm{R}_{\bm{x}}\) and the operator \(B\), and derive its second-order directional derivatives. Finally, we construct an explicit projection operator from \(\mathcal{E}\) onto \(\T_{\bm{x}}\mathcal{M}^B\). The retraction and projection operators are the basic geometric ingredients for Riemannian optimization algorithms on \(\mathcal{M}^B\). The projection provides tangent search directions and later induces the projection-induced Riemannian metric, while the retraction maps tangent updates back to \(\mathcal{M}^B\), thereby ensuring that all iterates remain feasible.

\subsection{Manifold Structure of \texorpdfstring{$\mathcal{M}^B$}{M\^{}B}}

\noindent To establish that $\mathcal{M}^B$ is a submanifold of $\mathcal{M}$, we appeal to standard results in infinite-dimensional differential geometry. A sufficient condition is that $\mathcal{M}$ and the affine constraint set intersect \emph{transversally} \cite[Theorem~3.5.12]{abraham2012manifolds}; equivalently, as we show below, the map $B|_{\mathcal{M}}$ is a \emph{submersion} at every point of $\mathcal{M}^B$, which then yields the submanifold structure directly via the submersion theorem \cite[Theorem~3.5.4]{abraham2012manifolds}. Both conditions follow constructively from Assumption~\ref{asmp:Lx invertible}.

\noindent We now make the transversality condition precise. Let \(\mathcal A_c:=\{\bm{x}\in\mathcal E\mid B\bm{x}=c\}\). Since \(B:\mathcal E\to\mathcal Y\) is bounded and surjective, \(\mathcal A_c\) is a closed affine subspace of \(\mathcal E\), and \(\T_{\bm{x}}\mathcal A_c=\ker(B)\) for all \(\bm{x}\in\mathcal A_c\). Let \(B^*:\mathcal Y\to\mathcal E\) denote the adjoint operator of \(B\), so that
\begin{equation}\label{norm of B*}
    \|B^*\|=\|B\|.
\end{equation}

\noindent In the Hilbert manifold setting, the transversality of \(\mathcal M\) and \(\mathcal A_c\) at \(\bm{x}\in\mathcal M^B\) means that
\begin{equation}\label{cond:transversality}
\T_{\bm{x}}\mathcal M+\T_{\bm{x}}\mathcal A_c=\mathcal E,
\qquad
\T_{\bm{x}}\mathcal M\cap\T_{\bm{x}}\mathcal A_c
\text{ splits in }
\T_{\bm{x}}\mathcal M.
\end{equation}
Condition \eqref{cond:transversality} ensures that \(\mathcal M^B\) is a submanifold of \(\mathcal M\), by \cite[Theorem 3.5.12]{abraham2012manifolds}. Since \(\T_{\bm{x}}\mathcal A_c=\ker(B)\), \eqref{cond:transversality} is equivalent to
\[
    \T_{\bm{x}}\mathcal M+\ker(B)=\mathcal E,
    \qquad
    \T_{\bm{x}}\mathcal M\cap\ker(B)
    \ \text{is split in }\
    \T_{\bm{x}}\mathcal M.
\]

\noindent On the other hand, the tangent map of \(B|_{\mathcal M}:\mathcal M\to
\mathcal Y\) at \(\bm{x}\) is
\begin{equation}\label{eqn:Sx}
    S_{\bm{x}}
    :=\T_{\bm{x}}(B|_{\mathcal M})
    =B|_{\T_{\bm{x}}\mathcal M}
    :\T_{\bm{x}}\mathcal M\to\mathcal Y.
\end{equation}
Recall that \(B|_{\mathcal M}\) is a submersion at \(\bm{x}\) if
\begin{equation}\label{cond:submersion}
S_{\bm{x}}\ \text{is surjective}
\qquad\text{and}\qquad
\ker S_{\bm{x}}
\ \text{splits in }\
\T_{\bm{x}}\mathcal M.
\end{equation}

\noindent Because \(B\) is surjective, the condition \(\T_{\bm{x}}\mathcal M+\ker(B)=\mathcal E\) is equivalent to the surjectivity of \(S_{\bm{x}}\). Indeed, the former gives \(B(\T_{\bm{x}}\mathcal M)=B(\mathcal E)=\mathcal Y\); conversely, if \(S_{\bm{x}}\) is surjective, then for any \(\bm e\in\mathcal E\) one can find \(\bm u\in\T_{\bm{x}}\mathcal M\) such that \(B\bm u=B\bm e\), hence \(\bm e-\bm u\in\ker(B)\). Moreover, $\ker S_{\bm{x}}=\T_{\bm{x}}\mathcal M\cap\ker(B)$. Therefore, the transversality condition \eqref{cond:transversality} is equivalent to the submersion condition \eqref{cond:submersion}.

\noindent We now give a convenient sufficient condition, formulated in terms of the projection, under which the submersion condition \eqref{cond:submersion} holds.

\begin{assumption}[Invertibility of \(L_{\bm{x}}\)]
\label{asmp:Lx invertible}
For any \(\bm{x}\in\mathcal M^B\), the linear operator
\[
L_{\bm{x}}:=B\bm P_{\bm{x}}B^*:\mathcal Y\to\mathcal Y
\]
is invertible.
\end{assumption}

\begin{remark}[Physical interpretation]
Assumption~\ref{asmp:Lx invertible} excludes degenerate configurations in many applications. For instance, in the inextensible rod model with fixed endpoints, \(L_{\bm{x}}\) fails to be invertible when the endpoint distance equals the rod length, in which case the rod is forced to be the straight segment between the endpoints and no deformation degrees of freedom remain. Such configurations are typically trivial and are excluded from the non-degenerate regime of computational interest.
\end{remark}

\begin{proposition}[Submersion induced by \(L_{\bm{x}}\)]
\label{prop:submersion-induced-by-Lx}
Under Assumption~\ref{asmp:Lx invertible}, the map
\(B|_{\mathcal M}:\mathcal M\to\mathcal Y\) is a submersion at every
\(\bm{x}\in\mathcal M^B\).
\end{proposition}

\begin{proof}
Fix \(\bm{x}\in\mathcal M^B\) and define
\[
\mathcal R_{\bm{x}}
:=
\bm P_{\bm{x}}B^*L_{\bm{x}}^{-1}
:
\mathcal Y\to\T_{\bm{x}}\mathcal M.
\]
Since \(L_{\bm{x}}\) is a bounded bijective operator on the Hilbert space \(\mathcal Y\), its inverse \(L_{\bm{x}}^{-1}\) is bounded by the Bounded Inverse Theorem. Consequently, \(\mathcal R_{\bm{x}}\), being a composition of the bounded operators \(\bm P_{\bm{x}}\), \(B^*\), and \(L_{\bm{x}}^{-1}\), is itself bounded. For any \(\eta\in\mathcal Y\),
\[
S_{\bm{x}}\mathcal R_{\bm{x}}\eta
=
B\bm P_{\bm{x}}B^*L_{\bm{x}}^{-1}\eta
=
L_{\bm{x}}L_{\bm{x}}^{-1}\eta
=
\eta.
\]
Hence \(\mathcal R_{\bm{x}}\) is a bounded right inverse of \(S_{\bm{x}}\), and therefore \(S_{\bm{x}}\) is surjective. For every \(\bm v\in\T_{\bm{x}}\mathcal M\), it holds that
\[
\bm v=\bigl(\bm v-\mathcal R_{\bm{x}}S_{\bm{x}}\bm v\bigr)+\mathcal R_{\bm{x}}S_{\bm{x}}\bm v.
\]
The first term belongs to \(\ker S_{\bm{x}}\), while the second belongs to
\(\operatorname{Ran}(\mathcal R_{\bm{x}})\). Moreover, if
\(\bm v\in\ker S_{\bm{x}}\cap\operatorname{Ran}(\mathcal R_{\bm{x}})\),
then \(\bm v=\mathcal R_{\bm{x}}\eta\) for some \(\eta\in\mathcal Y\), and
\[
0_{\mathcal Y}=S_{\bm{x}}\bm v=S_{\bm{x}}\mathcal R_{\bm{x}}\eta=\eta,
\]
which implies \(\bm v=0_{\mathcal E}\). Therefore, we have
\[
\T_{\bm{x}}\mathcal M=\ker S_{\bm{x}}\oplus\operatorname{Ran}(\mathcal R_{\bm{x}}).
\]
That is, \(\ker S_{\bm{x}}\) is split in \(\T_{\bm{x}}\mathcal M\). Thus, the condition \eqref{cond:submersion} holds, and therefore
\(B|_{\mathcal M}\) is a submersion at \(\bm{x}\).
\end{proof}

\begin{corollary}[Hilbert Manifold Structure of
\texorpdfstring{$\mathcal{M}^B$}{M\^{}B}]\label{cor:MB manifold}
Under Assumption~\ref{asmp:Lx invertible}, the feasible set
\(\mathcal M^B=(B|_{\mathcal M})^{-1}(c)=\mathcal M\cap\mathcal A_c\) is a \(\mathcal C^2\) Hilbert submanifold of \(\mathcal M\), and hence an embedded Hilbert submanifold of \(\mathcal E\). Moreover, for every
\(\bm{x}\in\mathcal M^B\),
\begin{equation}\label{eqn:TxMB}
    \T_{\bm{x}}\mathcal M^B=\ker\bigl(\T_{\bm{x}}(B|_{\mathcal M})\bigr)=\T_{\bm{x}}\mathcal M\cap\ker(B)=\T_{\bm{x}}\mathcal M\cap\{\bm v\in\mathcal E\mid B\bm v=0_{\mathcal Y}\}.
\end{equation}
\end{corollary}

\begin{proof}
By Proposition~\ref{prop:submersion-induced-by-Lx}, the restriction \(B|_{\mathcal M}:\mathcal M\to\mathcal Y\) is a submersion at every point of \(\mathcal M^B=(B|_{\mathcal M})^{-1}(c)\). Hence \(c\) is a regular value of
\(B|_{\mathcal M}\). The submersion theorem for Hilbert manifolds
\cite[Theorem 3.5.4]{abraham2012manifolds} implies that \(\mathcal M^B\) is a \(\mathcal C^2\) Hilbert submanifold of \(\mathcal M\). Since \(\mathcal M\) is embedded in \(\mathcal E\), \(\mathcal M^B\) is also an embedded Hilbert submanifold of \(\mathcal E\).

\noindent The tangent space of a regular level set is the kernel of the corresponding
tangent map. Hence
\[
    \T_{\bm{x}}\mathcal M^B=\{\bm v\in\T_{\bm{x}}\mathcal M\mid B\bm v=0_{\mathcal Y}\}=\T_{\bm{x}}\mathcal M\cap\ker(B),
\]
which proves \eqref{eqn:TxMB}.
\end{proof}

\noindent We endow each tangent space \(\T_{\bm{x}}\mathcal M^B\) with the Riemannian
metric inherited from the ambient Hilbert space \(\mathcal E\), namely
\begin{equation}\label{eqn:ambient riemannian metri}
    (\bm u,\bm v)_{\bm{x}}=(\bm u,\bm v)_{\mathcal E},\quad\bm u,\bm v\in\T_{\bm{x}}\mathcal M^B.
\end{equation}
Note that \eqref{eqn:ambient riemannian metri} equips \(\mathcal M^B\) with a natural Riemannian manifold structure and provides the geometric setting for the optimization problem \eqref{J^B}.

\begin{remark}[Relation with transversality and the orthogonal case]\label{rem:equivalence of transversal and invertible}
The preceding proof is written in the submersion language because this is the form directly used by the submersion theorem. Equivalently, it proves the transversality of \(\mathcal M\) and the affine constraint set
\(\mathcal A_c\) at every point of \(\mathcal M^B\).

\noindent The role of Assumption~\ref{asmp:Lx invertible} is constructive and
projection-dependent. In general, \(\bm P_{\bm{x}}\) is not assumed to be
orthogonal, and therefore \(\bm P_{\bm{x}}B^*\) is not necessarily the
adjoint of \(S_{\bm{x}}=B|_{\T_{\bm{x}}\mathcal M}\). Consequently,
\(L_{\bm{x}}=B\bm P_{\bm{x}}B^*\) should not be interpreted as
\(S_{\bm{x}}S_{\bm{x}}^*\) in the non-orthogonal case.

\noindent If \(\bm P_{\bm{x}}\) is the orthogonal projection onto
\(\T_{\bm{x}}\mathcal M\), then \(S_{\bm{x}}^*=\bm P_{\bm{x}}B^*\) and
\(L_{\bm{x}}=S_{\bm{x}}S_{\bm{x}}^*\). In this special case, 
\begin{enumerate}[(i)]
    \item the invertibility of \(L_{\bm{x}}\)
    \item the submersion property of \(B|_{\mathcal M}\) or the transversality of \(\mathcal M\)
    \item the coercivity condition
\[
    (L_{\bm{x}}\delta,\delta)_{\mathcal Y}\ge C_{\bm{x}}\|\delta\|_{\mathcal Y}^2,\qquad\forall\,\delta\in\mathcal Y,
\]
\end{enumerate}
are equivalent. Indeed,
\((L_{\bm{x}}\delta,\delta)_{\mathcal Y}
=\|S_{\bm{x}}^*\delta\|_{\mathcal E}^2\), and the lower bound for
\(S_{\bm{x}}^*\) is equivalent to the surjectivity of \(S_{\bm{x}}\) by
\cite[Theorem 2.20]{brezis2011functional}.
\end{remark}
\noindent In the remainder of this paper, we work under Assumption~\ref{asmp:Lx invertible}.

\subsection{Retraction on \texorpdfstring{$\mathcal{M}^B$}{}}
\noindent Under Assumption~\ref{asmp:Lx invertible}, it is shown that \(\mathcal{M}^B\) is a Hilbert submanifold. Now,  we construct a local retraction on \(\mathcal{M}^B\). For any $\bm{x} \in \mathcal{M}^B$, we define
\begin{equation}\label{eqn:VxB}
V_{\bm{x}}^B(r):=\left\{\bm{v}\in \mathrm{T}_{\bm{x}}\mathcal{M}^B \mid \norm{\bm{v}}_{\bm{x}}< r\right\}.
\end{equation}
\begin{lemma}\label{lem:construction of retraction}
For any point $\bm{x} \in \mathcal{M}^B$, there exist a radius $r_{\bm{x}} > 0$ and a $\mathcal{C}^2$ mapping $\delta_{\bm{x}} : V_{\bm{x}}^B(r_{\bm{x}}) \to \mathcal{Y}$ such that, for all $\bm{v} \in V_{\bm{x}}^B(r_{\bm{x}})$, the following condition is satisfied:
\begin{equation}\label{eqn:r_x}
\bm{R}_{\bm{x}}\left( \bm{v} - \bm{P}_{\bm{x}} B^* \delta_{\bm{x}}(\bm{v}) \right) \in \mathcal{M}^B.
\end{equation}
\end{lemma}
\begin{proof}
Fix an arbitrary $\bm{x} \in \mathcal{M}^B\subset\mathcal{M}$. Note that for any $\bm{v} \in \mathrm{T}_{\bm{x}} \mathcal{M}^B\subset\mathrm{T}_{\bm{x}} \mathcal{M}$ and $\delta \in \mathcal{Y}$, we have $$
\bm{v} - \bm{P}_{\bm{x}} B^* \delta \in \mathrm{T}_{\bm{x}} \mathcal{M}.
$$
Define the mapping $F_{\bm{x}}: \mathrm{T}_{\bm{x}} \mathcal{M}^B \times \mathcal{Y} \rightarrow \mathcal{Y}$ by
\begin{equation}\label{eqn:Fx}
	F_{\bm{x}}(\bm{v}, \delta) = B \bm{R}_{\bm{x}}\left( \bm{v} - \bm{P}_{\bm{x}} B^* \delta \right).
\end{equation}
Since $F_{\bm{x}}$ is the composition of the bounded linear map $(\bm{v},\delta) \mapsto \bm{v} - \bm{P}_{\bm{x}} B^* \delta$, the globally defined $\mathcal{C}^2$ retraction $\bm{R}_{\bm{x}}$, and the bounded linear operator $B$, it is of class $\mathcal{C}^2$. In particular, the partial derivative
\begin{equation}\label{eqn:partial Fx delta}
\frac{\partial F_{\bm{x}}}{\partial \delta}(\bm{v},\delta) = -B\circ {D}\bm{R}_{\bm{x}}\left( \bm{v} - \bm{P}_{\bm{x}} B^* \delta \right)\circ \bm{P}_{\bm{x}} B^* \in \mathcal{L}(\mathcal{Y}; \mathcal{Y})
\end{equation}
exists at all points $(\bm{v},\delta) \in \mathrm{T}_{\bm{x}} \mathcal{M}^B \times \mathcal{Y}$ and depends continuously on $(\bm{v},\delta)$.

\noindent By \eqref{retraction property 1}, we have
$
F_{\bm{x}}(\bm{0}_{\mathcal{E}}, 0_{\mathcal{Y}}) = B\bm{R_x}(\bm{0}_{\mathcal{E}})=B\bm{x}=c,
$
and 
\begin{equation}\label{eqn:partial F =-Lx}
    \frac{\partial F_{\bm{x}}}{\partial \delta}\bigg|_{(\bm{0}_{\mathcal{E}}, 0_{\mathcal{Y}})} =-B\circ {D}\bm{R}_{\bm{x}}\left( \bm{0}_{\mathcal{E}} \right)\circ \bm{P}_{\bm{x}} B^*=-B\bm{P}_{\bm{x}} B^* =-L_{\bm{x}}.
\end{equation}
Since $L_{\bm{x}}$ is invertible, the Implicit Function Theorem \cite[Theorem 7.13-1 on p.~548]{ciarlet2013linear} implies that there exist a radius $r_{\bm{x}} > 0$ and a $\mathcal{C}^2$ map $\delta_{\bm{x}}:  V_{\bm{x}}^B(r_{\bm{x}}) \to \mathcal{Y}$, such that for any $\bm{v} \in  V_{\bm{x}}^B(r_{\bm{x}})$ we have
\begin{equation}\label{s1}
B \bm{R}_{\bm{x}}\left( \bm{v} - \bm{P}_{\bm{x}} B^* \delta_{\bm{x}} (\bm{v}) \right) = c,
\end{equation}
with $\delta_{\bm{x}} (\bm{0}_{\mathcal{E}}) = 0_{\mathcal{Y}}$. We thus complete the proof.
\end{proof}

\noindent We simplify the notation for the neighborhood of radius $r_{\bm{x}}$ (from Lemma~\ref{lem:construction of retraction}) by writing
\begin{equation}\label{eqn:Vx}
    V_{\bm{x}}^B:=\left\{\bm{v}\in \mathrm{T}_{\bm{x}}\mathcal{M}^B \mid \|\bm{v}\|_{\bm{x}}< r_{\bm{x}}\right\}.
\end{equation}
Accordingly, we define the mapping \(\bm{R}_{\bm{x}}^B:V_{\bm{x}}^B\to \mathcal{M}^B\) by
\begin{equation}\label{retraction RxB}
\bm{R}_{\bm{x}}^B(\bm{v})
:= \bm{R}_{\bm{x}}\left(\bm{v}-\bm{P}_{\bm{x}} B^{*}\,\delta_{\bm{x}}(\bm{v})\right),
\end{equation}
where the operator \(\delta_{\bm{x}}\) is constructed as in Lemma~\ref{lem:construction of retraction}.

\begin{theorem}\label{thm:retraction second-order derivative}
	For any point $\bm{x}\in \mathcal{M}^B$, let $V_{\bm{x}}^B$ be defined by \eqref{eqn:Vx}. For any $\bm{v} \in V_{\bm{x}}^B$, we have $\bm{R}_{\bm{x}}^B(\bm{v}) \in \mathcal{M}^B$. Moreover, for any $\bm{v} \in \T_{\bm{x}}\mathcal{M}^B$, it holds that
	\begin{equation}\label{eqn:retraction RxB value at 0 and first-order derivative}
	    	\bm{c}^B(0) = \bm{x}, \qquad \left. \frac{d}{d\alpha} \bm{c}^B(\alpha) \right|_{\alpha=0} = \bm{v},
	\end{equation}
	where $\bm{c}^B(\alpha) := \bm{R}_{\bm{x}}^B(\alpha \bm{v})$ is defined for 
    all $\alpha \geq 0$ such that $\alpha\|\bm{v}\|_{\bm{x}} < r_{\bm{x}}$. Furthermore, we have
	\begin{equation}\label{second-order derivative RxB}
		\left. \frac{d^2}{d\alpha^2} \bm{c}^B(\alpha) \right|_{\alpha=0} = \Lambda_{\bm{x}} \left(\left. \frac{d^2}{d\alpha^2} \bm{c}(\alpha) \right|_{\alpha=0} \right)= \Lambda_{\bm{x}} \left( \ell_{\bm{x}}(\bm{v}, \bm{v}) \right),
	\end{equation}
	where for any $\bm{x} \in \mathcal{M}^B$, the operator $\Lambda_{\bm{x}}: \mathcal{E} \to \mathcal{E}$ is defined by
	\begin{equation}\label{Lambda}
		\Lambda_{\bm{x}} := \mathrm{Id}_{\mathcal{E}} - \bm{P_x}B^*L_{\bm{x}}^{-1}B.
	\end{equation}
\end{theorem}

\begin{proof}
Recall from Lemma~\ref{lem:construction of retraction} that $\delta_{\bm{x}}(\bm{0}_{\mathcal{E}}) = 0_{\mathcal{Y}}$ and $\delta_{\bm{x}}$ is $\mathcal{C}^2$. 
We then obtain the expansion
\begin{equation*}
\alpha \bm{v} - \bm{P}_{\bm{x}} B^* \delta_{\bm{x}}(\alpha\bm{v})
= \left[ \bm{v} - \bm{P}_{\bm{x}} B^* g_0 \right] \alpha
- \frac{1}{2} \bm{P}_{\bm{x}} B^* h_0 \alpha^2 + o(\alpha^2),
\end{equation*}
where 
\[
g_0 = \left. \frac{d}{d\alpha} \delta_{\bm{x}}(\alpha\bm{v}) \right|_{\alpha=0}
\quad \text{and} \quad 
h_0 = \left. \frac{d^2}{d\alpha^2} \delta_{\bm{x}}(\alpha\bm{v}) \right|_{\alpha=0}.
\]

\noindent The retraction $\bm{R}_{\bm{x}}$ admits the local expansion
\[
\bm{R}_{\bm{x}}(\bm{v}) = \bm{x} + \bm{v} + \frac{1}{2} \ell_{\bm{x}}(\bm{v}, \bm{v}) + o\left( \|\bm{v}\|_{\bm{x}}^2 \right).
\]
Combining these results, we compute
\begin{equation}\label{eqn:taytor expansion of retraction RB original}
    \begin{aligned}
	&\bm{R}_{\bm{x}}\left( \alpha \bm{v} - \bm{P}_{\bm{x}} B^* \delta_{\bm{x}}(\alpha\bm{v}) \right)\\
	=& \bm{x} + \left( \alpha \bm{v} - \bm{P}_{\bm{x}} B^* \delta_{\bm{x}}(\alpha\bm{v}) \right)
	+ \frac{1}{2} \ell_{\bm{x}}\left( \alpha \bm{v} - \bm{P}_{\bm{x}} B^* \delta_{\bm{x}}(\alpha\bm{v}), \alpha \bm{v} - \bm{P}_{\bm{x}} B^* \delta_{\bm{x}}(\alpha\bm{v}) \right)
	+ o(\alpha^2) \\
	=& \bm{x} + \left[ \bm{v} - \bm{P}_{\bm{x}} B^* g_0 \right] \alpha
	+ \frac{\alpha^2}{2} \left( -\bm{P}_{\bm{x}} B^* h_0 + \ell_{\bm{x}}\left( \bm{v} - \bm{P}_{\bm{x}} B^* g_0, \bm{v} - \bm{P}_{\bm{x}} B^* g_0 \right) \right)
	+ o(\alpha^2).
\end{aligned}
\end{equation}
By Lemma~\ref{lem:construction of retraction}, for any 
\(\bm{v}\in \mathrm{T}_{\bm{x}}\mathcal{M}^B\) and any \(\alpha \geq 0\) 
satisfying \(\alpha\|\bm{v}\|_{\bm{x}} < r_{\bm{x}}\) 
(so that \(\alpha\bm{v}\in V_{\bm{x}}^B\)), we have
\[
B\bm{R}_{\bm{x}}\left(\alpha\bm{v}-\bm{P}_{\bm{x}}B^*\,\delta_{\bm{x}}(\alpha\bm{v})\right)
= c= B\bm{x}.
\]
\noindent Taking the limit, we have
\[
0_{\mathcal{Y}}=\lim_{\alpha \to 0} \frac{B \left[ \bm{R}_{\bm{x}}\left( \alpha \bm{v} - \bm{P}_{\bm{x}} B^* \delta_{\bm{x}}(\alpha\bm{v}) \right) - \bm{x} \right] }{\alpha}
= B \left[ \bm{v} - \bm{P}_{\bm{x}} B^* g_0 \right]
= -B \bm{P}_{\bm{x}} B^*g_0 =-L_{\bm{x}}g_0,
\]
which implies $g_0 = 0_{\mathcal{Y}}$. Substituting $g_0 = 0_{\mathcal{Y}}$ in \eqref{eqn:taytor expansion of retraction RB original} yields the simplified expansion
\begin{equation}\label{eq:thm31_1}
    \bm{R}_{\bm{x}}\left( \alpha \bm{v} - \bm{P}_{\bm{x}} B^* \delta_{\bm{x}}(\alpha\bm{v}) \right)
= \bm{x} + \alpha \bm{v} + \frac{\alpha^2}{2} \left( -\bm{P}_{\bm{x}} B^* h_0 + \ell_{\bm{x}}(\bm{v}, \bm{v}) \right) + o(\alpha^2).
\end{equation}
Since $\bm{v} \in\mathrm{T}_{\bm{x}}\mathcal{M}^B$, we have $B\bm{v} = {0}_{\mathcal{Y}}$. By the constraint condition, we have
\[
B \bm{R}_{\bm{x}}\left( \alpha \bm{v} - \bm{P}_{\bm{x}} B^* \delta_{\bm{x}}(\alpha\bm{v}) \right) = c = B\bm{x} + \alpha B\bm{v},
\]
which implies
\[
B \left[ \bm{R}_{\bm{x}}\left( \alpha \bm{v} - \bm{P}_{\bm{x}} B^* \delta_{\bm{x}}(\alpha\bm{v}) \right) - \bm{x} - \alpha \bm{v} \right] = {0}_{\mathcal{Y}}.
\]
\noindent Taking the limit, we obtain
\[
\lim_{\alpha \to 0} \frac{ B \left[ \bm{R}_{\bm{x}}\left( \alpha \bm{v} - \bm{P}_{\bm{x}} B^* \delta_{\bm{x}}(\alpha\bm{v}) \right) - \bm{x} - \alpha \bm{v} \right] }{ \alpha^2 }
= \frac{1}{2} B \left[ -\bm{P}_{\bm{x}} B^* h_0 + \ell_{\bm{x}}(\bm{v}, \bm{v}) \right] = {0}_{\mathcal{Y}},
\]
which yields
\begin{equation}\label{eq:thm31_2}
h_0 = L_{\bm{x}}^{-1} B \ell_{\bm{x}}(\bm{v}, \bm{v}).
\end{equation}
From \eqref{retraction RxB}, \eqref{eq:thm31_1}, and \eqref{eq:thm31_2}, we have that
\[
\bm{c}^B(\alpha) = \bm{x} + \alpha \bm{v} + \frac{\alpha^2}{2} \left( \mathrm{Id}_{\mathcal{E}} - \bm{P}_{\bm{x}} B^* L_{\bm{x}}^{-1} B \right) \left( \ell_{\bm{x}}(\bm{v}, \bm{v}) \right) + o(\alpha^2).
\]
Then, the desired results follow from some direct calculations.
\end{proof}
\noindent For any \(\bm{x}\in\mathcal{M}^B\) and \(\bm{v}\in \mathrm{T}_{\bm{x}}\mathcal{M}^B\), \eqref{eqn:retraction RxB value at 0 and first-order derivative} yields
\[
\bm{R}_{\bm{x}}^B(\bm{0}_{\mathcal{E}})= \bm{c}^B(0)= \bm{x},\qquad D\bm{R}_{\bm{x}}^B(\bm{0}_{\mathcal{E}})[\bm{v}]= \left.\frac{d}{d\alpha}\bm{c}^B(\alpha)\right|_{\alpha=0}= \bm{v}.
\]
Consequently,
\[
\bm{R}_{\bm{x}}^B(\bm{0}_{\mathcal{E}})=\bm{x},\qquad D\bm{R}_{\bm{x}}^B(\bm{0}_{\mathcal{E}})=\mathrm{Id}_{\mathrm{T}_{\bm{x}}\mathcal{M}^B}.
\]
Replacing \(\mathcal{M}\) with \(\mathcal{M}^B\) and \( V_{\bm{x}}\) with \( V_{\bm{x}}^B\) in~\eqref{retraction property 1}, we conclude that \(\bm{R}_{\bm{x}}^B\) defines a $\mathcal{C}^2$ local retraction of \(\mathcal{M}^B\) at \(\bm{x}\).

\subsection{Projection on \texorpdfstring{$\mathcal{M}^B$}{}}

\noindent Inspired by Theorem~\ref{thm:retraction second-order derivative}, we construct a projection operator from $\mathcal{E}$ onto the tangent space $\mathrm{T}_{\bm{x}} \mathcal{M}^B$, which acts as the identity on $\mathrm{T}_{\bm{x}} \mathcal{M}^B$, as stated in the following proposition:

\begin{proposition}\label{prop:projection to TxMb}
    For any $\bm{x} \in \mathcal{M}^B$, the linear operator 
    \begin{equation}\label{eqn:PxB def}
        \bm{P}_{\bm{x}}^B := \bm{P_x} - \bm{P_x} B^* L_{\bm{x}}^{-1} B \bm{P_x}
    \end{equation} 
    defined on $\mathcal{E}$ is a projection operator onto $\mathrm{T}_{\bm{x}} \mathcal{M}^B$. Specifically, it satisfies
    \begin{equation}\label{eqn:PxBv in TxMB}
        \bm{P}_{\bm{x}}^B\bm{u} \in \mathrm{T}_{\bm{x}} \mathcal{M}^B, \quad \forall \bm{u} \in \mathcal{E},
    \end{equation}
    and acts as the identity map on $\mathrm{T}_{\bm{x}} \mathcal{M}^B$
    \begin{equation}\label{eqn:PxBv=v}
        \bm{P}_{\bm{x}}^B\bm{v} = \bm{v}, \quad \forall \bm{v} \in \mathrm{T}_{\bm{x}} \mathcal{M}^B.
    \end{equation}
    Furthermore, if $\bm{P_x}$ is an orthogonal projection, then $\bm{P}_{\bm{x}}^B$ is also an orthogonal projection.
\end{proposition}

\begin{proof}
    First, we verify that the range of \( \bm{P}_{\bm{x}}^B \) lies in \( \mathrm{T}_{\bm{x}} \mathcal{M}^B \). For any \( \bm{u} \in \mathcal{E} \), the definition of \( \bm{P}_{\bm{x}}^B \) gives
    \[
        \bm{P}_{\bm{x}}^B\bm{u} = \bm{P_x} \left( \bm{u} - B^* L_{\bm{x}}^{-1} B \bm{P_x} \bm{u} \right) \in \mathrm{T}_{\bm{x}} \mathcal{M}.
    \]
    Furthermore, applying the linear operator \( B \) yields
    \[
        B \bm{P}_{\bm{x}}^B \bm{u} = B \bm{P_x} \bm{u} - B\bm{P_x}B^* L_{\bm{x}}^{-1} B \bm{P_x} \bm{u} = B \bm{P_x} \bm{u} - L_{\bm{x}} L_{\bm{x}}^{-1} B \bm{P_x} \bm{u} = \bm{0}_{\mathcal{Y}}.
    \]
    Together, these imply that \( \bm{P}_{\bm{x}}^B\bm{u} \in \mathrm{T}_{\bm{x}} \mathcal{M}^B \), which proves \eqref{eqn:PxBv in TxMB}.

  \noindent  Next, we show that \( \bm{P}_{\bm{x}}^B \) acts as the identity mapping on \( \mathrm{T}_{\bm{x}} \mathcal{M}^B \). Let \( \bm{v} \in \mathrm{T}_{\bm{x}} \mathcal{M}^B \subset \mathrm{T}_{\bm{x}} \mathcal{M} \). By the definition, we have \( \bm{P_x} \bm{v} = \bm{v} \) and \( B \bm{v} = \bm{0}_{\mathcal{Y}} \). Thus,
    \[
        \bm{P}_{\bm{x}}^B\bm{v} = \left( \bm{P_x} - \bm{P_x} B^* L_{\bm{x}}^{-1} B \bm{P_x} \right) \bm{v} = \bm{v} - \bm{P_x} B^* L_{\bm{x}}^{-1} B \bm{v} = \bm{v} - \bm{P_x} B^* L_{\bm{x}}^{-1} \bm{0}_{\mathcal{Y}} = \bm{v},
    \]
    which proves \eqref{eqn:PxBv=v}.

 \noindent   With \eqref{eqn:PxBv in TxMB} and \eqref{eqn:PxBv=v} established, the idempotence of \( \bm{P}_{\bm{x}}^B \) follows immediately. For any \( \bm{u} \in \mathcal{E} \), letting \( \bm{v} = \bm{P}_{\bm{x}}^B \bm{u} \in \mathrm{T}_{\bm{x}} \mathcal{M}^B \), we obtain
    \[
        (\bm{P}_{\bm{x}}^B)^2 \bm{u} = \bm{P}_{\bm{x}}^B (\bm{P}_{\bm{x}}^B \bm{u}) = \bm{P}_{\bm{x}}^B \bm{v} = \bm{v} = \bm{P}_{\bm{x}}^B \bm{u}.
    \]
    This implies \( (\bm{P}_{\bm{x}}^B)^2 = \bm{P}_{\bm{x}}^B \), confirming that \( \bm{P}_{\bm{x}}^B \) is a projection operator.

 \noindent   Finally, we establish that \( \bm{P}_{\bm{x}}^B \) is an orthogonal projection by proving it is self-adjoint. Since \( \bm{P_x} \) is an orthogonal projection, it is self-adjoint, i.e., \( \bm{P_x}^* = \bm{P_x} \). Consequently, the operator \( L_{\bm{x}} = B \bm{P_x} B^* \) is also self-adjoint because
    \[
        L_{\bm{x}}^* = (B \bm{P_x} B^*)^* = (B^*)^* \bm{P_x}^* B^* = B \bm{P_x} B^* = L_{\bm{x}}.
    \]
    This implies that its inverse \( L_{\bm{x}}^{-1} \) is self-adjoint as well. Taking the adjoint of \( \bm{P}_{\bm{x}}^B \), we obtain
    \begin{equation*}
        (\bm{P}_{\bm{x}}^B)^* = \bm{P_x}^* - (\bm{P_x} B^* L_{\bm{x}}^{-1} B \bm{P_x})^* 
        = \bm{P_x}^* - \bm{P_x}^* B^* (L_{\bm{x}}^{-1})^* (B^*)^* \bm{P_x}^* 
        = \bm{P_x} - \bm{P_x} B^* L_{\bm{x}}^{-1} B \bm{P_x} 
        = \bm{P}_{\bm{x}}^B.
    \end{equation*}
    Since \( \bm{P}_{\bm{x}}^B \) is both idempotent and self-adjoint, we conclude that it is an orthogonal projection.
\end{proof}

\begin{remark}
The surjective projection \(\bm{P}_{\bm{x}}^B\) defines a mapping from the ambient space \(\mathcal{E}\) onto the tangent space \(\mathrm{T}_{\bm{x}}\mathcal{M}^B\). It plays a crucial role in transferring elements of \(\mathcal{E}\) to the tangent space, which enables us to represent and characterize the tangent space through ambient-space quantities. This projection further underlies the construction of the projection-induced gradient and, consequently, our algorithmic design.
\end{remark}

\section{Computing the Derivatives}\label{sec:derivatives computations}

\noindent We first fix notation for the derivatives of $J$ on $\mathcal{E}$. By the Riesz representation theorem, we identify \(\mathcal{E}^*\) with \(\mathcal{E}\). Under this identification, the first- and second-order Fr\'echet derivatives of \(J\) at \(\bm{x}\in\mathcal{E}\) are represented by the gradient \(\grad J(\bm{x})\in\mathcal{E}\) and the Hessian \(\hess J(\bm{x})\in\mathcal{L}(\mathcal{E};\mathcal{E})\). Throughout, a mapping is called \emph{bounded} if it maps bounded sets to bounded sets; recall from Section~\ref{subsec:model} that
\[
\grad J:\mathcal{E}\to\mathcal{E}
\quad\text{and}\quad
\hess J:\mathcal{E}\to\mathcal{L}(\mathcal{E};\mathcal{E})
\]
are assumed to be bounded in this sense.

\noindent A standard approach in manifold optimization is to transfer the objective functional locally from \(\mathcal{M}^B\) to the tangent space \(\mathrm{T}_{\bm{x}}\mathcal{M}^B\) at the current iterate \(\bm{x}\), via the retraction \(\bm{R}_{\bm{x}}^B\) defined in \eqref{retraction RxB}. The resulting \emph{pullback} of \(J^B\), also called the \emph{lifted objective functional}, is the mapping \(J_{\bm{x}}^B:V_{\bm{x}}^B\to\mathbb{R}\) defined by
\begin{equation}\label{eqn:JxB}
    J_{\bm{x}}^B(\bm{\xi})
    :=
    J^B\!\left(\bm{R}_{\bm{x}}^B(\bm{\xi})\right).
\end{equation}
Since \(\mathrm{T}_{\bm{x}}\mathcal{M}^B\) is a linear subspace of \(\mathcal{E}\), this formulation permits the use of standard unconstrained vector-space optimization techniques. In Theorem~\ref{thm: taylor}, we derive explicit expressions for the gradient and Hessian of \(J_{\bm{x}}^B\) near the origin, which yield the linear and quadratic models underlying line-search and trust-region methods on \(\mathcal{M}^B\).

\noindent We next introduce the derivative operators that will be used throughout the analysis. Let \(\grad J_{\bm{x}}^B(\bm{0}_{\mathcal{E}})\in\mathrm{T}_{\bm{x}}\mathcal{M}^B\) denote the gradient of \(J_{\bm{x}}^B\) at \(\bm{0}_{\mathcal{E}}\), and let \(\grad J^B(\bm{x})\in\mathrm{T}_{\bm{x}}\mathcal{M}^B\) denote the Riemannian gradient of \(J^B\) at \(\bm{x}\). Then, by the proof of Proposition~3.59 in \cite{boumal2023introduction}, for every \(\bm{x}\in\mathcal{M}^B\), we have
\[
\grad J^B(\bm{x})=\grad J_{\bm{x}}^B(\bm{0}_{\mathcal{E}}).
\]

\noindent Let \(\hess J_{\bm{x}}^B(\bm{0}_{\mathcal{E}}):\mathrm{T}_{\bm{x}}\mathcal{M}^B\to\mathrm{T}_{\bm{x}}\mathcal{M}^B\) denote the Hessian of \(J_{\bm{x}}^B\) at \(\bm{0}_{\mathcal{E}}\). For each \(\bm{x}\in\mathcal{M}^B\), define the bilinear form \(b_{\bm{x}}:\mathcal{E}\times\mathcal{E}\to\mathbb{R}\)
by
\[
b_{\bm{x}}(\bm{u}_1,\bm{u}_2):=\left(\grad J(\bm{x}),\Lambda_{\bm{x}}
    \bigl(
        \ell_{\bm{x}}
        (\bm{P}_{\bm{x}}\bm{u}_1,\bm{P}_{\bm{x}}\bm{u}_2)
    \bigr)
\right)_{\mathcal{E}},
\qquad
\forall\,\bm{u}_1,\bm{u}_2\in\mathcal{E}.
\]

\noindent Since \(\grad J(\bm{x})\), \(B\), \(L_{\bm{x}}^{-1}\), \(\ell_{\bm{x}}\), and \(\bm{P}_{\bm{x}}\) are bounded, the bilinear form \(b_{\bm{x}}\) is bounded; moreover, \(b_{\bm{x}}\) is symmetric since \(\ell_{\bm{x}}\) is. Therefore, by the Riesz representation theorem, there exists a unique bounded self-adjoint operator \(\bm{\pi}_{\bm{x}}:\mathcal{E}\to\mathcal{E}\) such that
\[
\left(
    \bm{\pi}_{\bm{x}}[\bm{u}_1],
    \bm{u}_2
\right)_{\mathcal{E}}
=
b_{\bm{x}}(\bm{u}_1,\bm{u}_2),
\qquad
\forall\,\bm{u}_1,\bm{u}_2\in\mathcal{E}.
\]

\noindent In particular, for any \(\bm{u},\bm{v}\in\mathrm{T}_{\bm{x}}\mathcal{M}^B\subset\mathrm{T}_{\bm{x}}\mathcal{M},\) we have
\[
\left(
    \bm{\pi}_{\bm{x}}[\bm{u}],
    \bm{v}
\right)_{\mathcal{E}}
=
\left(
    \grad J(\bm{x}),
    \Lambda_{\bm{x}}
    \bigl(
        \ell_{\bm{x}}(\bm{u},\bm{v})
    \bigr)
\right)_{\mathcal{E}}.
\]

\noindent With the above notation in place, the following theorem characterizes the derivatives and Taylor expansion of ${J}_{\bm{x}}^B$.

\begin{theorem}\label{thm: taylor}
	For any $\bm{x} \in \mathcal{M}^B$ and $\bm{v} \in \mathrm{T}_{\bm{x}} \mathcal{M}^B$, we have:
	\begin{enumerate}
		\item [(1).] The gradient of $J^B$ at $\bm{x}$ and of ${J}^B_{\bm{x}}$ at $\bm{0}_{\mathcal{E}}$ in the direction $\bm{v}$ satisfy
		\begin{equation}\label{eqn:first-order derivative of JB and JxB}
		\left(\grad J^B(\bm{x}), \bm{v}\right)_{\bm{x}} = \left(\grad J^B_{\bm{x}}(\bm{0}_{\mathcal{E}}), \bm{v}\right)_{\bm{x}} = \left( \grad J(\bm{x}), \bm{v} \right)_{\mathcal{E}}.
		\end{equation}
		
		\item [(2).] The Hessian of \(J^B_{\bm{x}}\) at \(\bm{0}_{\mathcal{E}}\) is self-adjoint with respect to \((\cdot,\cdot)_{\bm{x}}\). Moreover, for any \(\bm{v},\bm{u}\in \T_{\bm{x}}\mathcal{M}^B\), it is given by
		\begin{equation}\label{eqn:second-order direvative of JxB}
		\left(\hess J^B_{\bm{x}}(\bm{0}_{\mathcal{E}})[\bm{v}], \bm{u}\right)_{\bm{x}} = \left( \left[ \hess J(\bm{x}) + \bm{\pi}_{\bm{x}} \right] \bm{v}, \bm{u} \right)_{\mathcal{E}}.
		\end{equation}
		
		\item [(3).] The Taylor expansion of ${J}^B_{\bm{x}}$ at $\bm{0}_{\mathcal{E}}$ along the direction $\bm{v}$ can be written as
		$$
		\begin{aligned}
			J^B\left(\bm{R}^B_{\bm{x}}(\bm{v})\right) = {J}^B_{\bm{x}}(\bm{v}) &= J^B(\bm{x}) + \left(\grad J^B(\bm{x}), \bm{v}\right)_{\bm{x}} + \frac{1}{2} \left(\hess J^B_{\bm{x}}(\bm{0}_{\mathcal{E}})[\bm{v}], \bm{v}\right)_{\bm{x}} + o\left(\|\bm{v}\|^2_{\bm{x}}\right)\\
			&= J(\bm{x}) + \left( \grad J(\bm{x}), \bm{v} \right)_{\mathcal{E}} + \frac{1}{2} \left( \left[ \hess J(\bm{x}) + \bm{\pi}_{\bm{x}} \right] \bm{v}, \bm{v} \right)_{\mathcal{E}} + o\left(\norm{\bm{v}}^2_{\mathcal{E}}\right).
		\end{aligned}
		$$
	\end{enumerate}
\end{theorem}

\begin{proof}
    The case \(\bm v=\bm 0_{\mathcal E}\) is immediate. Let \(\bm v\in\T_{\bm x}\mathcal M^B\setminus\{\bm 0_{\mathcal E}\}\) be arbitrary. Recall from Theorem~\ref{thm:retraction second-order derivative} that we define
    $\bm{c}^B(\alpha) := \bm{R}_{\bm{x}}^B(\alpha \bm{v}),
    \alpha \in \Big[0,\frac{r_{\bm{x}}}{\|\bm{v}\|_{\bm{x}}}\Big).$ Consider the function $g(\alpha) = {J}^B_{\bm{x}}(\alpha\bm{v}) = J^B\left(\bm{c}^B(\alpha)\right) = J\left(\bm{c}^B(\alpha)\right)$. Then, the first- and second-order derivatives of $g$ with respect to $\alpha$ are given, respectively, by
	$$
	g'(\alpha) = \left( \grad J\left(\bm{c}^B(\alpha)\right), \frac{d}{d\alpha}\bm{c}^B(\alpha) \right)_{\mathcal{E}},
	$$
	and
	\begin{equation}\label{eqn:second-order directional derivative of JxB}
	    g''(\alpha) = \left( \hess J\left(\bm{c}^B(\alpha)\right)\left[\frac{d}{d\alpha}\bm{c}^B(\alpha)\right], \frac{d}{d\alpha}\bm{c}^B(\alpha) \right)_{\mathcal{E}} + \left( \grad J\left(\bm{c}^B(\alpha)\right), \frac{d^2}{d\alpha^2}\bm{c}^B(\alpha) \right)_{\mathcal{E}}.
	\end{equation}
	By substituting $\alpha=0$ and using $\bm{c}^B(0) = \bm{x}$, $\frac{d}{d\alpha}\bm{c}^B(\alpha)\mid_{\alpha=0} = \bm{v}$, and $\frac{d^2}{d\alpha^2}\bm{c}^B(\alpha)\mid_{\alpha=0} = \Lambda_{\bm{x}}\left(\ell_{\bm{x}}(\bm{v}, \bm{v})\right)$, we obtain
	$$
	\left(\grad J^B_{\bm{x}}(\bm{0}_{\mathcal{E}}), \bm{v}\right)_{\bm{x}} = g'(0) = \left( \grad J(\bm{x}), \bm{v} \right)_{\mathcal{E}},
	$$
	and
	\begin{multline*}
	    	\left(\hess J^B_{\bm{x}}(\bm{0}_{\mathcal{E}})[\bm{v}], \bm{v}\right)_{\bm{x}} = g''(0) = \left( \hess J(\bm{x})[\bm{v}], \bm{v} \right)_{\mathcal{E}} + \left( \grad J(\bm{x}), \Lambda_{\bm{x}}\left(\ell_{\bm{x}}(\bm{v}, \bm{v})\right) \right)_{\mathcal{E}}\\
	= \left( \left[ \hess J(\bm{x}) + \bm{\pi}_{\bm{x}} \right] \bm{v}, \bm{v} \right)_{\mathcal{E}}.
	\end{multline*}
    Since \(\hess J^B_{\bm{x}}(\bm{0}_{\mathcal{E}})\) is the Hessian of a scalar-valued $\mathcal{C}^2$ function, it is self-adjoint with respect to \((\cdot,\cdot)_{\bm{x}}\). Moreover, \(\hess J(\bm{x})\) is self-adjoint and, by the symmetry of \(\ell_{\bm{x}}\), so is \(\bm{\pi}_{\bm{x}}\). Define
    \[
    H_{\bm{x}}(\bm{v},\bm{u}):=\left(\hess J^B_{\bm{x}}(\bm{0}_{\mathcal{E}})[\bm{v}],\bm{u}\right)_{\bm{x}},\qquad Q_{\bm{x}}(\bm{v},\bm{u}):=\left(\left[\hess J(\bm{x})+\bm{\pi}_{\bm{x}}\right]\bm{v},\bm{u}\right)_{\mathcal{E}}.
    \]
    Both forms are symmetric. Applying the preceding identity to \(\bm{v}+\bm{u}\), \(\bm{v}\), and \(\bm{u}\), polarization yields
    \begin{multline*}
    2H_{\bm{x}}(\bm{v},\bm{u})=H_{\bm{x}}(\bm{v}+\bm{u},\bm{v}+\bm{u})-H_{\bm{x}}(\bm{v},\bm{v})-H_{\bm{x}}(\bm{u},\bm{u})\\
    =Q_{\bm{x}}(\bm{v}+\bm{u},\bm{v}+\bm{u})-Q_{\bm{x}}(\bm{v},\bm{v})-Q_{\bm{x}}(\bm{u},\bm{u})=2Q_{\bm{x}}(\bm{v},\bm{u}).
    \end{multline*}
    This proves \eqref{eqn:second-order direvative of JxB}. With the first- and second-order identities established, applying the Taylor expansion gives
	$$
	\begin{aligned}
		{J}^B_{\bm{x}}(\bm{v}) &= {J}^B_{\bm{x}}(\bm{0}_{\mathcal{E}}) + \left(\grad J^B_{\bm{x}}(\bm{0}_{\mathcal{E}}), \bm{v}\right)_{\bm{x}} + \frac{1}{2}\left(\hess J^B_{\bm{x}}(\bm{0}_{\mathcal{E}})[\bm{v}], \bm{v}\right)_{\bm{x}} + o\left(\norm{\bm{v}}_{\bm{x}}^2\right) \\
        &= J^B(\bm{x}) + \left(\grad J^B(\bm{x}), \bm{v}\right)_{\bm{x}} + \frac{1}{2} \left(\hess J^B_{\bm{x}}(\bm{0}_{\mathcal{E}})[\bm{v}], \bm{v}\right)_{\bm{x}} + o\left(\|\bm{v}\|^2_{\bm{x}}\right)\\
		&= J(\bm{x}) + \left( \grad J(\bm{x}), \bm{v} \right)_{\mathcal{E}} + \frac{1}{2} \left( \left[ \hess J(\bm{x}) + \bm{\pi}_{\bm{x}} \right] \bm{v}, \bm{v} \right)_{\mathcal{E}} + o\left(\norm{\bm{v}}^2_{\mathcal{E}}\right).
	\end{aligned}
	$$
	This completes the proof.
\end{proof}

\section{Foundations of Optimization Theory}\label{sec:properties}

\noindent The preceding sections establish fundamental analytic and geometric properties of \(\mathcal{M}^B\), including its tangent spaces, projection operators, retractions, and derivative formulas. Now, we develop uniform estimates on an arbitrary bounded subset \(X\subset\mathcal{M}^B\). Such uniform estimates are essential for establishing the foundation of optimization theory intrinsically on \(\mathcal{M}^B\): it is precisely what enables standard techniques in Euclidean spaces to be formulated on \(\mathcal{M}^B\) with rigorous convergence guarantees.


\noindent First, we introduce a uniform coercivity condition for \(L_{\bm{x}}=B\bm{P}_{\bm{x}}B^*\), which provides uniform control of \(L_{\bm{x}}^{-1}\) on bounded subsets. We also clarify its relation to the usual uniform submersion estimate, showing their equivalence in the orthogonal projection case. Second, under mild boundedness assumptions on the projection field \(\bm{P}_{\bm{x}}\) and the base retraction \(\bm{R}_{\bm{x}}\), we prove the existence of a uniform radius for the implicit correction \(\delta_{\bm{x}}\) and the retraction \(\bm{R}_{\bm{x}}^B\), together with the required \(\mathcal{C}^2\) regularity on bounded sets. Third, we establish uniform bounds for the second-order directional derivatives of \(\bm{R}_{\bm{x}}^B\). Finally, we derive uniform regularity properties of the lifted objective \(J_{\bm{x}}^B=J^B\circ\bm{R}_{\bm{x}}^B\), including radial Lipschitz \(C^1\) continuity, a uniform quadratic upper estimate, and uniform bounds for \(\grad J^B(\bm{x})\) and \(\hess J_{\bm{x}}^B(\bm{0}_{\mathcal{E}})\). Together, these results complete the analytic toolkit needed to treat optimization on \(\mathcal{M}^B\) at the Hilbert space level.

\subsection{Uniform Coercivity and Uniform Submersion}
\label{subsec:uniform-coercivity-submersion}

\noindent Pointwise invertibility of \(L_{\bm x}\), as required in Assumption~\ref{asmp:Lx invertible}, is sufficient to establish the  Hilbert-submanifold structure of \(\mathcal M^B\). However, for the subsequent algorithmic construction and convergence analysis, pointwise invertibility is not sufficient. We need uniform control of \(L_{\bm x}^{-1}\) on bounded subsets of $\mathcal{M}^B$. This motivates the following strengthened assumption.

\begin{assumption}[Uniform coercivity]
\label{asmp:Lx bounded below}
For every bounded set \(S\subset\mathcal M^B\), there exists a constant
\(C_S>0\) such that
\begin{equation}\label{eqn:Lx bounded below}
    (L_{\bm x}\delta,\delta)_{\mathcal Y}
    \ge
    C_S\|\delta\|_{\mathcal Y}^2,
    \qquad
    \forall\,\bm x\in S,\ \delta\in\mathcal Y .
\end{equation}
\end{assumption}

\noindent Since \(\bm P_{\bm x}\) is not assumed to be orthogonal, the operator \(L_{\bm x}=B\bm P_{\bm x}B^*\) is generally not self-adjoint. Nevertheless, \eqref{eqn:Lx bounded below} is a coercivity condition for the bilinear form \((\delta,\eta)\mapsto (L_{\bm x}\delta,\eta)_{\mathcal Y}\). Hence, by the Lax--Milgram theorem, \(L_{\bm x}\) is boundedly invertible and
\begin{equation}\label{eqn:Lx-1 bounded above}
    \|L_{\bm x}^{-1}\|
    \le
    \frac{1}{C_S},
    \qquad
    \forall\,\bm x\in S .
\end{equation}

\begin{remark}[Algorithmic Significance of Assumption~\ref{asmp:Lx bounded below}]
Assumption~\ref{asmp:Lx bounded below} is fundamental for the subsequent analysis. In particular, the uniform bound \eqref{eqn:Lx-1 bounded above} on \(\norm{L_{\bm{x}}^{-1}}\) yields uniform control of the projection operators and of the derivatives of the retractions constructed on \(\mathcal{M}^B\). Such bounds are among the key ingredients in establishing the convergence analysis of Riemannian optimization algorithms on \(\mathcal{M}^B\). Thus, Assumption~\ref{asmp:Lx bounded below} provides not only an analytic nondegeneracy condition but also the quantitative estimates needed for rigorous algorithmic analysis.
\end{remark}

\noindent We next give a uniform submersion estimate. This is a different condition in general, because it involves the adjoint of \(S_{\bm x}= B|_{\T_{\bm{x}}\mathcal{M}}\), rather than the extrinsic operator \(\bm P_{\bm x}\).

\begin{definition}[Uniform submersion]
\label{def:uniform-submersion}
Let \(S\subset\mathcal M^B\) be bounded. We say that \(B|_{\mathcal M}\) is a
uniform submersion on \(S\) if there exists \(\gamma_S>0\) such that
\begin{equation}\label{eqn:uniform submersion estimate}
    \|S_{\bm x}^*\delta\|_{\mathcal E}
    \ge
    \gamma_S\|\delta\|_{\mathcal Y},
    \qquad
    \forall\,\bm x\in S,\ \delta\in\mathcal Y .
\end{equation}
Equivalently,
\begin{equation}\label{eqn:tangent space angle}
    \inf_{\delta\in\mathcal Y,\ \delta\ne 0}
    \sup_{\bm v\in\T_{\bm x}\mathcal M,\ \bm v\ne 0}
    \frac{|(\delta,B\bm v)_{\mathcal Y}|}
    {\|\delta\|_{\mathcal Y}\|\bm v\|_{\mathcal E}}
    \ge
    \gamma_S,
    \qquad
    \forall\,\bm x\in S .
\end{equation}
\end{definition}

\begin{remark}[Geometric interpretation of uniform submersion]
\label{rem:angle-interpretation-uniform-submersion}
Fix \(\bm x\in\mathcal M^B\). The sine of the transverse angle between
\(\T_{\bm x}\mathcal M\) and \(\ker(B)\) is given by
\[
    \sin\theta(\T_{\bm x}\mathcal M,\ker(B))
    :=
    \inf_{\bm z\in (\ker(B))^\perp,\ \bm z\ne 0}
    \sup_{\bm v\in \T_{\bm x}\mathcal M,\ \bm v\ne 0}
    \frac{|(\bm z,\bm v)_{\mathcal E}|}
    {\|\bm z\|_{\mathcal E}\|\bm v\|_{\mathcal E}} .
\]
Since \(B\) is surjective, \(\operatorname{Ran}(B^*)=(\ker(B))^\perp\).
Hence, writing \(\bm z=B^*\delta\), we obtain
\[
    \sin\theta(\T_{\bm x}\mathcal M,\ker(B))
    =
    \inf_{\delta\in\mathcal Y,\ \delta\ne 0}
    \sup_{\bm v\in \T_{\bm x}\mathcal M,\ \bm v\ne 0}
    \frac{|(B^*\delta,\bm v)_{\mathcal E}|}
    {\|B^*\delta\|_{\mathcal E}\|\bm v\|_{\mathcal E}}  
    =
    \inf_{\delta\in\mathcal Y,\ \delta\ne 0}
    \sup_{\bm v\in \T_{\bm x}\mathcal M,\ \bm v\ne 0}
    \frac{|(\delta,B\bm v)_{\mathcal Y}|}
    {\|B^*\delta\|_{\mathcal E}\|\bm v\|_{\mathcal E}} .
\]
Since \(B\) is surjective, \(B^*\) is bounded from below; that is, there exists a constant \(C_B>0\) such that
\[
    \|B^*\delta\|_{\mathcal E}\ge C_B\|\delta\|_{\mathcal Y},
    \qquad \forall\,\delta\in\mathcal Y .
\]
Moreover, \(\|B^*\delta\|_{\mathcal E}\le \|B\|\|\delta\|_{\mathcal Y}\).
Therefore, it holds that
\[
    C_B\sin\theta(\T_{\bm x}\mathcal M,\ker(B))
    \le
    \inf_{\delta\in\mathcal Y,\ \delta\ne 0}
    \sup_{\bm v\in\T_{\bm x}\mathcal M,\ \bm v\ne 0}
    \frac{|(\delta,B\bm v)_{\mathcal Y}|}
    {\|\delta\|_{\mathcal Y}\|\bm v\|_{\mathcal E}}  
    \le
    \|B\|\sin\theta(\T_{\bm x}\mathcal M,\ker(B)).
\]
Thus, the submersion inf--sup constant is equivalent, up to constants depending only on \(B\), to the sine of the transverse angle. Consequently, uniform submersion on a bounded set \(S\subset\mathcal M^B\) is equivalent to the angle \eqref{eqn:tangent space angle} being uniformly bounded away from zero on \(S\).
\end{remark}

\noindent The uniform coercivity assumption and the standard uniform submersion estimate are generally distinct. They coincide, however, when the projection $\bm{P_x}$ is orthogonal.

\begin{proposition}[Orthogonal projection case]
\label{prop:equivalence-uniform-coercivity-submersion}
Suppose that \(\bm P_{\bm x}\) is the orthogonal projection onto
\(\T_{\bm x}\mathcal M\) for every \(\bm x\in\mathcal M^B\). Then
Assumption~\ref{asmp:Lx bounded below} is equivalent to the uniform submersion estimate \eqref{eqn:uniform submersion estimate}. More precisely, on each bounded set \(S\subset\mathcal M^B\), the constants satisfy \(\gamma_S=\sqrt{C_S}\) and \(C_S=\gamma_S^2\).
\end{proposition}

\begin{proof}
In the orthogonal case, \(\bm P_{\bm x}=\bm P_{\bm x}^*\), and therefore
\(S_{\bm x}^*=\bm P_{\bm x}B^*\). Hence
\[
    L_{\bm x}
    =
    B\bm P_{\bm x}B^*
    =
    S_{\bm x}S_{\bm x}^* .
\]
It follows that, for every \(\delta\in\mathcal Y\),
\[
    (L_{\bm x}\delta,\delta)_{\mathcal Y}
    =
    \|S_{\bm x}^*\delta\|_{\mathcal E}^2 .
\]
Thus, the lower bound
\((L_{\bm x}\delta,\delta)_{\mathcal Y}\ge
C_S\|\delta\|_{\mathcal Y}^2\) is equivalent to
\(\|S_{\bm x}^*\delta\|_{\mathcal E}\ge
\sqrt{C_S}\|\delta\|_{\mathcal Y}\).
\end{proof}

\noindent In the remainder of this paper, we work under Assumption~\ref{asmp:Lx bounded below}. Consequently, Assumption~\ref{asmp:Lx invertible} holds automatically, the feasible set \(\mathcal M^B\) is a Hilbert submanifold by Corollary~\ref{cor:MB manifold}, and the uniform inverse estimate \eqref{eqn:Lx-1 bounded above} is available on every bounded subset relevant to the subsequent analysis.

\subsection{Additional Assumptions on \texorpdfstring{$\bm{P_x}$}{} and \texorpdfstring{$\bm{R_x}$}{}}

\noindent With the uniform coercivity framework in place, we next record two mild assumptions on the ambient projection \(\bm{P}_{\bm{x}}\) and the base retraction \(\bm{R}_{\bm{x}}\). These assumptions will allow us to derive uniform operator bounds on bounded subsets and will be invoked repeatedly in both the algorithm design and the convergence analysis.

\begin{assumption}\label{asmp:boundedness of P}
Let \(\bm{P_x}:\mathcal{E}\to \T_{\bm{x}}\mathcal{M}\subset \mathcal{E}\) denote the projection operator introduced above. Throughout this section, each \(\bm{P_x}\) is regarded as an element of \(\mathcal{L}(\mathcal{E},\mathcal{E})\). We assume that the mapping
\begin{equation}\label{eqn:boundedness of P}
    \bm{P}:\mathcal{M}\to \mathcal{L}(\mathcal{E},\mathcal{E}),\qquad \bm{x}\mapsto \bm{P_x},
\end{equation}
is bounded and continuous.
\end{assumption}

\noindent For \(\bm{x}\in\mathcal{M}\) and \(\bm{u}\in\T_{\bm{x}}\mathcal{M}\), let \(D\bm{R}_{\bm{x}}(\bm{u}):\T_{\bm{x}}\mathcal{M}\to\mathcal{E}\) and \(D^2\bm{R}_{\bm{x}}(\bm{u}):\T_{\bm{x}}\mathcal{M}\times\T_{\bm{x}}\mathcal{M}\to\mathcal{E}\) denote the first- and second-order Fréchet derivatives of \(\bm{R}_{\bm{x}}\) at \(\bm{u}\), respectively. Norms of $D\bm{R}_{\bm{x}}(\bm{u})$ and $D^2\bm{R}_{\bm{x}}(\bm{u})$ are the corresponding operator norms, with \(\T_{\bm{x}}\mathcal{M}\) endowed with \((\cdot,\cdot)_{\mathcal{M},\bm{x}}\) and \(\mathcal{E}\) with \((\cdot,\cdot)_{\mathcal{E}}\). Since \((\cdot,\cdot)_{\mathcal{M},\bm{x}}\) is the restriction of \((\cdot,\cdot)_{\mathcal{E}}\) to \(\T_{\bm{x}}\mathcal{M}\), these coincide with the ambient operator norms.

\begin{assumption}\label{asmp:uniform bound of first second-order derivative of retraction}
	Let \(S \subset \mathcal{M}\) be any bounded set. Then there exist constants
\(\sigma>0\), \(K_1>0\), and \(K_2>0\), depending only on \(S\), such that for every
\(\bm{x}\in S\) and every \(\bm{u}\in \T_{\bm{x}}\mathcal{M}\) with \(\|\bm{u}\|_{\mathcal{E}}\le \sigma\), the following
bounds hold:
\begin{equation}\label{eqn:boundedness of DR and D2R}
     \|D\bm{R}_{\bm{x}}(\bm{u})\| \le K_1 \quad \hbox{and} \quad
        \|D^2\bm{R}_{\bm{x}}(\bm{u})\| \le K_2.
    \end{equation}
\end{assumption}

\noindent From this point onward, let 
\begin{equation}\label{eqn:X}
    X \subset \mathcal{M}^B
\end{equation}
be an arbitrary bounded subset. All constants appearing below are understood to depend only on \(X\), unless otherwise stated. Since \(X\) is bounded, there exists a constant \(C_1>0\) such that
\begin{equation}\label{eqn:boundedness of X}
\norm{\bm{x}}_{\mathcal{E}}\le C_1 \quad \forall \bm{x}\in X.
\end{equation}

\noindent By Assumption~\ref{asmp:boundedness of P}, there exists a constant \(K_P>0\) such that
\begin{equation}\label{eqn:Px bounded above X}
	\norm{\bm{P_x}}\le K_P,\quad \forall \bm{x}\in X.
\end{equation}

\noindent Moreover, by \eqref{eqn:boundedness of DR and D2R}, there exist constants \(\sigma, K_1, K_2>0\), depending only on \(X\), such that for every
\(\bm{x}\in X\) and every \(\bm{u}\in \mathrm{T}_{\bm{x}}\mathcal{M}\) with \(\|\bm{u}\|_{\mathcal{E}}\le \sigma\), we have
\begin{equation}\label{eqn:DR D2R bounded above X}
    \|D \bm{R_x}(\bm{u})\| \le K_1 \quad \hbox{and} \quad \|D^2 \bm{R_x}(\bm{u})\| \le K_2.
\end{equation}

\noindent Finally, by \eqref{eqn:Lx-1 bounded above}, there exists a constant \(K_L>0\) such that
\begin{equation}\label{eqn:Lx-1 bounded above X}
	\norm{L_{\bm{x}}^{-1}}\le K_L,\quad \forall \bm{x}\in X.
\end{equation}

\subsection{Uniform Radius for \texorpdfstring{$\bm{R_x}^B$}{}}
\noindent The previous subsection provides the main ingredients for algorithm design. To carry out the convergence analysis, however, we also need uniform local control of the implicit correction \(\delta_{\bm{x}}\) that underlies the retraction \(\bm{R}_{\bm{x}}^B\). We therefore begin by establishing a uniform existence radius, independent of the base point \(\bm{x}\in X\). 
We first show that there exists a uniform radius $\Delta_s>0$ such that for any $\bm{x}\in X$, the implicit function $\delta_{\bm{x}}$ exists on $V_{\bm{x}}^B(\Delta_s)$.
\begin{theorem}\label{thm:uniform radius for solution}
	For any point $\bm{x}\in X$, recall the definition of $F_{\bm{x}}$ in \eqref{eqn:Fx} and the definition of $V_{\bm{x}}^B(\cdot)$ in \eqref{eqn:VxB}. Define
    \begin{equation}\label{eqn:Sr definition}
        S_r:=\{\delta \in \mathcal{Y} \mid \norm{\delta}_{\mathcal{Y}}\le r\}.
    \end{equation}
    There exist $\Delta_s>0$ and $r>0$, independent of $\bm{x}$, and a unique continuous map $\delta_{\bm{x}}:V_{\bm{x}}^B(\Delta_s)\to S_r$, such that for $\bm{v}\in V_{\bm{x}}^B(\Delta_s)$, we have
	\begin{equation}\label{eqn:existence of solution}
		F_{\bm{x}}(\bm{v},\delta_{\bm{x}}(\bm{v}))=B\bm{R}_{\bm{x}}\left(\bm{v}-\bm{P}_{\bm{x}}B^* \delta_{\bm{x}}(\bm{v})\right)=c,
	\end{equation}
    and 
    \begin{equation}\label{eqn:delta0=0}
    \delta_{\bm{x}}\bigl(\bm{0}_{\mathcal{E}}\bigr)=0_{\mathcal{Y}}.
    \end{equation}
\end{theorem}

\begin{proof}
	Define the constants
	\begin{equation}\label{eqn:def of r}
		r:=\min\left\{\frac{\sigma}{2K_P\norm{B}},\frac{1}{4K_2K_LK_P^2\norm{B}^3}\right\},
	\end{equation}
	\begin{equation}\label{eqn:def of barDelta 1}
		\Delta_s:=\min\left\{\frac{\sigma}{2},\frac{1}{4K_2K_LK_P\norm{B}^2},\frac{r}{2K_1K_L\norm{B}}\right\}.
	\end{equation}
	Fix an arbitrary $\bm{x}\in X$. Subsequently, choose an arbitrary $\bm{v}\in V^B_{\bm{x}}(\Delta_s)$. Define the mapping $\Psi_{\bm{x},\bm{v}} : \mathcal{Y} \to \mathcal{Y}$ by
	\[
	\Psi_{\bm{x},\bm{v}}(\delta)=L^{-1}_{\bm{x}}\left(	B\bm{R}_{\bm{x}}\left(\bm{v}-\bm{P}_{\bm{x}}B^* \delta\right)-c+L_{\bm{x}}\delta\right).
	\]
	The existence of a solution to equation \eqref{eqn:existence of solution} is equivalent to $\Psi_{\bm{x},\bm{v}}$ possessing a fixed point. For any \(\delta_1,\delta_2 \in S_r\), we compute
	\begin{equation}\label{eqn:psi1-psi2}
	    \Psi_{\bm{x},\bm{v}}(\delta_1)-\Psi_{\bm{x},\bm{v}}(\delta_2)=L^{-1}_{\bm{x}}\bigg(	B\bm{R}_{\bm{x}}\left(\bm{v}-\bm{P}_{\bm{x}}B^* \delta_1\right)-B\bm{R}_{\bm{x}}\left(\bm{v}-\bm{P}_{\bm{x}}B^* \delta_2\right)+L_{\bm{x}}(\delta_1-\delta_2)\bigg).
	\end{equation}
	Recalling \eqref{retraction property 1}, we have
    \(D\bm{R}_{\bm{x}}(\bm{0}_{\mathcal{E}})=\mathrm{Id}_{\mathrm{T}_{\bm{x}}\mathcal{M}}\).
    Hence, by the definition of \(L_{\bm{x}}\) in Assumption~\ref{asmp:Lx invertible},
    \begin{equation}\label{eqn:Lx=BDR0PB}
        B\circ D\bm{R}_{\bm{x}}(\bm{0}_{\mathcal{E}})\circ \bm{P}_{\bm{x}}B^*
    = B\bm{P}_{\bm{x}}B^*
    = L_{\bm{x}}.
    \end{equation}
    \noindent To estimate the term in parentheses in \eqref{eqn:psi1-psi2}, let
    \[
        \bm w_\theta:=\bm v-\bm P_{\bm x}B^*\bigl(\theta\delta_1+(1-\theta)\delta_2\bigr),\qquad \theta\in(0,1).
    \]
    Using \eqref{eqn:Lx=BDR0PB} and \cite[Theorem 7.2-2]{ciarlet2013linear}, we obtain
    \begin{multline}\label{eqn:Mean Value Theorem used for BDRx}
    \norm{
    B\bm R_{\bm x}\left(\bm v-\bm P_{\bm x}B^*\delta_1\right)-B\bm R_{\bm x}\left(\bm v-\bm P_{\bm x}B^*\delta_2\right)+L_{\bm x}(\delta_1-\delta_2)
    }_{\mathcal Y}  \\
    \quad \le\sup_{\theta\in(0,1)}\norm{
    B\circ \left(D\bm R_{\bm x}(\bm w_\theta)-D\bm R_{\bm x}(\bm 0_{\mathcal E})\right)\circ \bm P_{\bm x}B^*
    }
    \norm{\delta_1-\delta_2}_{\mathcal Y}. 
    \end{multline}
    Moreover, it follows from \eqref{eqn:Px bounded above X}, \eqref{eqn:def of r}, and \eqref{eqn:def of barDelta 1} that
    \begin{equation}\label{eqn:wtheta estimate}
    \norm{\bm w_\theta}_{\mathcal{E}}\le\norm{\bm v}_{\mathcal{E}}+\norm{\bm P_{\bm x}}\norm{B}\bigl\|\theta\delta_1+(1-\theta)\delta_2\bigr\|_{\mathcal Y}\le\Delta_s+K_P\norm{B}r\le \sigma,
    \end{equation}
    Then, with \eqref{eqn:Px bounded above X}, \eqref{eqn:DR D2R bounded above X}, and \eqref{eqn:wtheta estimate}, for every \(\theta\in(0,1)\), it holds that
    \begin{multline}
    \norm{B\circ \left(D\bm R_{\bm x}(\bm w_\theta)-D\bm R_{\bm x}(\bm 0_{\mathcal E})\right)\circ \bm P_{\bm x}B^*} \le\norm{B}^2 K_P\norm{D\bm R_{\bm x}(\bm w_\theta)-D\bm R_{\bm x}(\bm 0_{\mathcal E})} \\
    \le\norm{B}^2 K_P\sup_{s\in(0,1)}\norm{D^2\bm R_{\bm x}(s\bm w_\theta)}\norm{\bm w_\theta}_{\mathcal{E}} \le\norm{B}^2 K_P K_2 \norm{\bm w_\theta}_{\mathcal{E}}.
    \label{eqn:DR-difference-bound}
    \end{multline}
    Also, it follows from \eqref{eqn:def of r}, \eqref{eqn:def of barDelta 1}, and  \eqref{eqn:wtheta estimate} that 
    \begin{equation}\label{eqn:B2KPK2 wtheta estimate}
    \norm{B}^2 K_P K_2 \norm{\bm w_\theta}_{\mathcal{E}}\le\norm{B}^2K_PK_2(\Delta_s+K_P\norm{B}r)
    \le \frac{1}{2K_L}.
    \end{equation}
    Combining \eqref{eqn:Mean Value Theorem used for BDRx}, \eqref{eqn:DR-difference-bound}, and \eqref{eqn:B2KPK2 wtheta estimate} yields
    \begin{equation}\label{eqn:contraction for BRx}
    \norm{
    B\bm R_{\bm x}\left(\bm v-\bm P_{\bm x}B^*\delta_1\right)
    -
    B\bm R_{\bm x}\left(\bm v-\bm P_{\bm x}B^*\delta_2\right)
    +
    L_{\bm x}(\delta_1-\delta_2)
    }_{\mathcal Y}
    \le
    \frac{1}{2K_L}\norm{\delta_1-\delta_2}_{\mathcal Y}.
    \end{equation}
    
	\noindent Then, by \eqref{eqn:Lx-1 bounded above X}, \eqref{eqn:psi1-psi2}, and \eqref{eqn:contraction for BRx}, we have
	\begin{equation}\label{eqn:contration of Psi}
		\norm{\Psi_{\bm{x},\bm{v}}(\delta_1)-\Psi_{\bm{x},\bm{v}}(\delta_2)}_{\mathcal{Y}}\le K_L \left( \frac{1}{2K_L}\norm{\delta_1-\delta_2}_{\mathcal{Y}} \right) = \frac{1}{2}\norm{\delta_1-\delta_2}_{\mathcal{Y}},
	\end{equation}
	which implies that $\Psi_{\bm{x},\bm{v}}$ is a contraction mapping on $S_r$. Recalling \eqref{retraction property 1}, we have $B\bm{R}_{\bm{x}}(\bm{0}_{\mathcal{E}})=B\bm{x}=c$. Therefore, by \eqref{eqn:DR D2R bounded above X}, \eqref{eqn:Lx-1 bounded above X}, \eqref{eqn:def of barDelta 1}, and the Mean Value Theorem, we have
	\begin{equation}\label{eqn:Psi 0 upper bound}
    \begin{aligned}
    \norm{\Psi_{\bm{x},\bm{v}}(0_{\mathcal{Y}})}_{\mathcal{Y}}&=\norm{L^{-1}_{\bm{x}}\left(B\bm{R}_{\bm{x}}\left(\bm{v}\right)-B\bm{R}_{\bm{x}}(\bm{0}_{\mathcal{E}})\right)}_{\mathcal{Y}}\\
    &\le \norm{L^{-1}_{\bm{x}}}\norm{B}\sup_{\theta\in(0,1)}\norm{D\bm{R}_{\bm{x}}\left(\theta\bm{v}\right)}\norm{\bm{v}}_{\bm{x}}\\
    &\le K_L\norm{B}K_1\norm{\bm{v}}_{\bm{x}}\le K_1K_L\norm{B}\Delta_s\le \frac{r}{2}.
    \end{aligned}
	\end{equation}
	Combining \eqref{eqn:contration of Psi} and \eqref{eqn:Psi 0 upper bound}, for any $\delta\in S_r$, we have
	\[
	\norm{\Psi_{\bm{x},\bm{v}}(\delta)}_{\mathcal{Y}} \le \norm{\Psi_{\bm{x},\bm{v}}(\delta) - \Psi_{\bm{x},\bm{v}}(0_{\mathcal{Y}})}_{\mathcal{Y}} + \norm{\Psi_{\bm{x},\bm{v}}(0_{\mathcal{Y}})}_{\mathcal{Y}} \le \frac{1}{2}\norm{\delta}_{\mathcal{Y}} + \frac{r}{2} \le \frac{r}{2} + \frac{r}{2} = r,
	\]
	which implies that $\Psi_{\bm{x},\bm{v}}$ maps $S_r$ into itself (i.e., $\Psi_{\bm{x},\bm{v}}:S_r\to S_r$). Thus, by Banach's Fixed Point Theorem, there exists a unique $\delta \in S_r$ such that $\Psi_{\bm{x},\bm{v}}(\delta)=\delta$.
	
	\noindent This establishes the uniqueness and existence of the implicit function $\bm{v}\in V_{\bm{x}}^B(\Delta_s)\mapsto \delta_{\bm{x}}(\bm{v})\in S_r$ satisfying \eqref{eqn:existence of solution} and thus
    \begin{equation}\label{eqn:Psixv deltaxv = deltaxv}
        \Psi_{\bm{x},\bm{v}}(\delta_{\bm{x}}(\bm{v}))=\delta_{\bm{x}}(\bm{v}),\, \forall \bm{v}\in V_{\bm{x}}^B(\Delta_s).
    \end{equation}
    Notice that $\Psi_{\bm{x},\bm{0}_{\mathcal{E}}}(0_{\mathcal{Y}})=L^{-1}_{\bm{x}}\left(	B\bm{R}_{\bm{x}}\left(\bm{0}_{\mathcal{E}}\right)-c\right)=0_{\mathcal{Y}}$, hence $\delta_{\bm{x}}\bigl(\bm{0}_{\mathcal{E}}\bigr)=0_{\mathcal{Y}}$. For continuity, let $\bm{v},\bm{v}_0\in V_{\bm{x}}^B(\Delta_s)$. By \eqref{eqn:contration of Psi}, and \eqref{eqn:Psixv deltaxv = deltaxv}, we have
    \[
    \begin{aligned}
        \norm{\delta_{\bm{x}}(\bm{v})-\delta_{\bm{x}}(\bm{v}_0)}_{\mathcal{Y}}&\le\norm{\Psi_{\bm{x},\bm{v}}(\delta_{\bm{x}}(\bm{v}))-\Psi_{\bm{x},\bm{v}}(\delta_{\bm{x}}(\bm{v}_0))}_{\mathcal{Y}}+\norm{\Psi_{\bm{x},\bm{v}}(\delta_{\bm{x}}(\bm{v}_0))-\Psi_{\bm{x},\bm{v}_0}(\delta_{\bm{x}}(\bm{v}_0))}_{\mathcal{Y}}\\
        &\le \frac{1}{2}\norm{\delta_{\bm{x}}(\bm{v})-\delta_{\bm{x}}(\bm{v}_0)}_{\mathcal{Y}}+\norm{\Psi_{\bm{x},\bm{v}}(\delta_{\bm{x}}(\bm{v}_0))-\Psi_{\bm{x},\bm{v}_0}(\delta_{\bm{x}}(\bm{v}_0))}_{\mathcal{Y}},
    \end{aligned}
    \]
    which implies
    \[
    \begin{aligned}
        \norm{\delta_{\bm{x}}(\bm{v})-\delta_{\bm{x}}(\bm{v}_0)}_{\mathcal{Y}}&\le 2\norm{\Psi_{\bm{x},\bm{v}}(\delta_{\bm{x}}(\bm{v}_0))-\Psi_{\bm{x},\bm{v}_0}(\delta_{\bm{x}}(\bm{v}_0))}_{\mathcal{Y}}\\
        &= 2\norm{L^{-1}_{\bm{x}}\left[	B\bm{R}_{\bm{x}}\left(\bm{v}-\bm{P}_{\bm{x}}B^* \delta_{\bm{x}}(\bm{v}_0)\right)-B\bm{R}_{\bm{x}}\left(\bm{v}_0-\bm{P}_{\bm{x}}B^* \delta_{\bm{x}}(\bm{v}_0)\right)\right]}_{\mathcal{Y}}.
    \end{aligned}
    \]
    Since $\bm{R}_{\bm{x}}$ is continuous and both $L^{-1}_{\bm{x}}$ and $B$ are bounded linear operators, we have
    \[
    \lim_{\bm{v}\to\bm{v}_0}\norm{L^{-1}_{\bm{x}}\left[	B\bm{R}_{\bm{x}}\left(\bm{v}-\bm{P}_{\bm{x}}B^* \delta_{\bm{x}}(\bm{v}_0)\right)-B\bm{R}_{\bm{x}}\left(\bm{v}_0-\bm{P}_{\bm{x}}B^* \delta_{\bm{x}}(\bm{v}_0)\right)\right]}_{\mathcal{Y}}=0,
    \]
    which shows that $\delta_{\bm{x}}\in \mathcal{C}(V_{\bm{x}}^B(\Delta_s),S_r)$. This completes the proof.
\end{proof} 

\noindent We next show that the implicit function $\delta_{\bm{x}}$ is $\mathcal{C}^2$ on $V_{\bm{x}}^B(\Delta_s)$.
\begin{theorem}\label{thm:uniform radius for C2}
    For any $\bm{x}\in X$, $\delta_{\bm{x}}$ is $\mathcal{C}^2$ on $V^B_{\bm{x}}(\Delta_s)$.
\end{theorem}
\begin{proof}
    Recall that for any $\bm{x}\in X$, $\bm{R_x}$ is a $C^2$ map. Then by the definition of $F_{\bm{x}}$ in \eqref{eqn:Fx}, $F_{\bm{x}}$ is a $\mathcal{C}^2$ map. Let
    \begin{equation}\label{eqn:epsilon}
    \epsilon:=\min\left\{\sigma,\frac{1}{2K_L\norm{B}^2K_PK_2}\right\}.
    \end{equation}
    
    \noindent For any $(\bm{v},\delta)\in V_{\bm{x}}^B(\Delta_s)\times S_{r}$, by \eqref{eqn:Px bounded above X}, \eqref{eqn:def of r}, and \eqref{eqn:def of barDelta 1}, we have
    \begin{equation}\label{eqn:v-PxB*delta norm upper}
    \norm{\bm{v}-\bm{P_x}B^*\delta}_{\mathcal{E}}\le\norm{\bm{v}}_{\mathcal{E}}+\norm{\bm{P_x}}\norm{B}\norm{\delta}_{\mathcal{Y}}\le \Delta_s+K_P\norm{B}r\le\epsilon.
    \end{equation}
    By \eqref{eqn:partial Fx delta} and \eqref{eqn:partial F =-Lx}, it holds that
    \begin{equation}\label{eqn:difference partial Fx}
    \norm{ \frac{\partial F_{\bm{x}}}{\partial\delta}(\bm{v},\delta) - \frac{\partial F_{\bm{x}}}{\partial\delta}(\bm{0}_{\mathcal E},\bm{0}_{\mathcal Y}) }= \norm{ B\circ\left( D\bm{R}_{\bm{x}}\left(\bm{v}-\bm{P}_{\bm{x}}B^*\delta\right) - D\bm{R}_{\bm{x}}(\bm{0}_{\mathcal E}) \right) \circ\bm{P}_{\bm{x}}B^*}.
    \end{equation}
    Moreover, it follows from \eqref{eqn:DR D2R bounded above X}, \eqref{eqn:v-PxB*delta norm upper}, and the Mean Value Theorem that
    \begin{equation}\label{eqn:estimate difference DR2}
            \norm{ D\bm{R}_{\bm{x}}\left(\bm{v}-\bm{P}_{\bm{x}}B^*\delta\right) - D\bm{R}_{\bm{x}}(\bm{0}_{\mathcal E}) } \le \sup_{\theta\in[0,1]} \norm{ D^2\bm{R}_{\bm{x}}\left( \theta\left(\bm{v}-\bm{P}_{\bm{x}}B^*\delta\right) \right) } \norm{\bm{v}-\bm{P}_{\bm{x}}B^*\delta}_{\mathcal E} 
            \le K_2\epsilon.
    \end{equation}
    Consequently, by \eqref{eqn:partial F =-Lx}, \eqref{eqn:Px bounded above X}, \eqref{eqn:Lx-1 bounded above X}, \eqref{eqn:difference partial Fx} and \eqref{eqn:estimate difference DR2}, we have
    \begin{equation}\label{eqn:bound of difference of partial F} 
            \norm{ \frac{\partial F_{\bm{x}}}{\partial\delta}(\bm{v},\delta) - \frac{\partial F_{\bm{x}}}{\partial\delta}(\bm{0}_{\mathcal E},\bm{0}_{\mathcal Y}) }\le \norm{B}^2K_PK_2\epsilon \le \frac{1}{2K_L} \le \frac{1}{2\norm{L_{\bm{x}}^{-1}}} = \frac{1}{ 2\norm{ \left( \frac{\partial F_{\bm{x}}}{\partial\delta}(\bm{0}_{\mathcal E},\bm{0}_{\mathcal Y}) \right)^{-1} } }.
    \end{equation}
    Since $\frac{\partial F_{\bm{x}}}{\partial \delta}(\bm{0}_{\mathcal{E}},0_{\mathcal{Y}})=-L_{\bm{x}}$ is invertible, by \eqref{eqn:bound of difference of partial F} and Theorem 3.6-3 in \cite{ciarlet2013linear}, $\frac{\partial F_{\bm{x}}}{\partial \delta}(\bm{v},\delta)$ is invertible for all $(\bm{v},\delta)\in V_{\bm{x}}^B(\Delta_s)\times S_{r}$. Moreover, using \eqref{eqn:Lx-1 bounded above X}, \eqref{eqn:bound of difference of partial F}, and Theorem 3.6-3 in \cite{ciarlet2013linear} again, we obtain
    \begin{multline*}
    \norm{\left(\frac{\partial F_{\bm{x}}}{\partial \delta}(\bm{v},\delta)\right)^{-1}}\le \frac{\norm{\left(\frac{\partial F_{\bm{x}}}{\partial \delta}(\bm{0}_{\mathcal{E}},0_{\mathcal{Y}})\right)^{-1}}}{1-\norm{\left(\frac{\partial F_{\bm{x}}}{\partial \delta}(\bm{0}_{\mathcal{E}},0_{\mathcal{Y}})\right)^{-1}}\norm{\left(\frac{\partial F_{\bm{x}}}{\partial \delta}(\bm{0}_{\mathcal{E}},0_{\mathcal{Y}})-\frac{\partial F_{\bm{x}}}{\partial \delta}(\bm{v},\delta)\right)}}\\
    \le 2\norm{\left(\frac{\partial F_{\bm{x}}}{\partial \delta}(\bm{0}_{\mathcal{E}},0_{\mathcal{Y}})\right)^{-1}}=2\norm{L_{\bm{x}}^{-1}}\le 2K_L.
    \end{multline*}

    \noindent Observing that the partial derivative $\frac{\partial F_{\bm{x}}}{\partial \delta}(\bm{v},\delta)$ is invertible and $\left(\frac{\partial F_{\bm{x}}}{\partial \delta}(\bm{v},\delta)\right)^{-1}$ is bounded linear for all $(\bm{v},\delta)\in V_{\bm{x}}^B(\Delta_s)\times S_{r}$, we follow steps (iv)--(vii) of the proof of Theorem 7.13-1 in \cite{ciarlet2013linear} to deduce that $\delta_{\bm{x}}$ is a $\mathcal{C}^2$ mapping on $V_{\bm{x}}^B(\Delta_s)$. 
\end{proof}
\begin{remark}
    To facilitate the subsequent analysis, we summarize the key properties established in Theorem \ref{thm:uniform radius for C2}. 

   \noindent Let $\bm{x}\in X$. For any $\bm{v}\in V_{\bm{x}}^B(\Delta_s)$, we have
    \begin{equation}\label{eqn:v-PxB*deltax v upper bound}
        \norm{\bm{v}-\bm{P_x}B^*\delta_{\bm{x}}(\bm{v})}_{\mathcal{E}}\le\epsilon,
    \end{equation}
    and the partial derivative $\frac{\partial F_{\bm{x}}}{\partial \delta}(\bm{v},\delta_{\bm{x}}(\bm{v}))$ is invertible with
    \begin{equation}\label{eqn:partialx v deltaxv inverse and upper bound}
        \norm{\left(\frac{\partial F_{\bm{x}}}{\partial \delta}(\bm{v},\delta_{\bm{x}}(\bm{v}))\right)^{-1}}\le 2K_L.
    \end{equation}
\end{remark}

\noindent By Theorems~\ref{thm:uniform radius for solution} and \ref{thm:uniform radius for C2}, the local radius \(r_{\bm{x}}\) appearing in Lemma~\ref{lem:construction of retraction} and Theorem~\ref{thm:retraction second-order derivative} can be replaced, on the bounded set \(X\), by a uniform radius \(\Delta_s\) independent of \(\bm{x}\). Consequently, for each \(\bm{x}\in X\), there exists a \(\mathcal{C}^2\) mapping \(\delta_{\bm{x}}: V_{\bm{x}}^B(\Delta_s)\to \mathcal{Y}\) such that, for any \(\bm{v}\in \mathrm{T}_{\bm{x}}\mathcal{M}^B\) with \(\|\bm{v}\|_{\mathcal{E}}<\Delta_s\),
\begin{equation}\label{eqn:existence of implicit function}
B \bm{R}_{\bm{x}}\Bigl(\bm{v}-\bm{P}_{\bm{x}} B^* \delta_{\bm{x}}(\bm{v})\Bigr)=c.
\end{equation}
This uniform radius will be used in the next subsection to derive uniform estimates for the retraction \(\bm{R}_{\bm{x}}^B\) on \(X\). Moreover, since \(\bm{R}_{\bm{x}}\) and \(\delta_{\bm{x}}\) are \(\mathcal{C}^2\), the mapping
\[
\bm{R}_{\bm{x}}^B(\bm{v}) = \bm{R}_{\bm{x}}\!\left( \bm{v}-\bm{P}_{\bm{x}}B^*\delta_{\bm{x}}(\bm{v}) \right), \qquad \bm{v}\in V_{\bm{x}}^B(\Delta_s),
\]
defines a \(\mathcal{C}^2\) local retraction on \(\mathcal{M}^B\) at every \(\bm{x}\in X\).

\subsection{Uniform Boundedness of the Directional Derivatives of \texorpdfstring{$\bm{R_x}^B$}{}}

\noindent We next establish a uniform bound for the second-order directional derivatives of the retraction \(\bm{R}_{\bm{x}}^B\) on \(X\). This estimate will be used in the subsequent regularity analysis of the lifted objective functional.

\begin{theorem}\label{thm:uniform boundedness of directional second-order retraction}
	There exist constants $\beta_D>0$ and $0<\Delta_D< \Delta_s$ such that for any $\bm{x}\in X$ and any $\bm{v}\in \mathrm{T}_{\bm{x}}\mathcal{M}^B$ satisfying
    $\|\bm{v}\|_{\bm{x}}=1$, we have
    \[
    \left\|\frac{\mathrm{d}^2}{\mathrm{d}t^2}\,\bm{R}^B_{\bm{x}}(t\bm{v})\right\|_{\mathcal{E}}
    \le \beta_D,
    \qquad \forall\, t\in[0,\Delta_D].
    \]
\end{theorem}
\begin{proof}
	Fix an arbitrary $\bm{x}\in X$. Subsequently, choose an arbitrary $\bm{v}\in\mathrm{T}_{\bm{x}}\mathcal{M}^B$ with $\norm{\bm{v}}_{\bm{x}}=1$. Denote
	\begin{equation}\label{eqn:eta definition}
		\eta_{\bm{x},\bm{v}}(t):=\delta_{\bm{x}}(t\bm{v}),
	\end{equation}
	and
	\begin{equation}\label{eqn:xi definition}
		\xi_{\bm{x},\bm{v}}(t):=t\bm{v}-\bm{P_x}B^*\eta_{\bm{x},\bm{v}}(t),
	\end{equation}
	where $t\in [0,\Delta_s)$. To simplify the notation, we denote the first-order derivatives as 
    $$\detaxv(t) = \frac{\mathrm{d}}{\dt}\etaxv(t),\quad\dxixv(t) = \frac{\mathrm{d}}{\dt}\xixv(t).$$ 
    We apply the same convention to the second-order derivatives, denoted by $\ddot{\eta}_{\bm{x},\bm{v}}(t)$ and $\ddot{\xi}_{\bm{x},\bm{v}}(t)$. Substituting \eqref{eqn:eta definition} and \eqref{eqn:xi definition} into \eqref{eqn:existence of implicit function} and differentiating \eqref{eqn:existence of implicit function} with respect to $t$, we obtain
	\begin{equation}\label{eqn:first derivative of retraction equation origin}
	    B\circ D\bm{R_x}(\xi_{\bm{x},\bm{v}}(t))[\dxixv(t)]=0_{\mathcal{Y}},
	\end{equation}
	which implies
	\begin{equation}\label{eqn:first derivative of retraction equation}
		B\circ D\bm{R_x}(\xi_{\bm{x},\bm{v}}(t))[\bm{v}]=\mathcal{B}(t)[\detaxv(t)],
	\end{equation}
	where
    \begin{equation}\label{eqn:mathcal B definition}
        \mathcal{B}(t):=B\circ D\bm{R_x}(\xi_{\bm{x},\bm{v}}(t))\bm{P_x}B^*.
    \end{equation}
    By \eqref{eqn:delta0=0} and \eqref{eqn:Lx=BDR0PB}, we have
    \begin{equation}\label{eqn:mathcal B(0)}
        \mathcal{B}(0)=B\circ D\bm{R_x}(\xi_{\bm{x},\bm{v}}(0))\circ\bm{P_x}B^*=B\circ D\bm{R_x}(\bm{0}_{\mathcal{E}})\circ \bm{P_x}B^*=B\bm{P_x}B^*=L_{\bm{x}}.
    \end{equation}
    Differentiating \eqref{eqn:first derivative of retraction equation origin} with respect to $t$, we get
	\begin{equation}\label{eqn:second derivative of retraction equation}
		\mathcal{B}(t)\ddetaxv(t)=B\circ D^2\bm{R_x}(\xi_{\bm{x},\bm{v}}(t))[\dxixv(t),\dxixv(t)].
	\end{equation}
	Recall the constant \(\epsilon\) defined in \eqref{eqn:epsilon}.  Notice that $t\bm{v}\in V_{\bm{x}}^B(\Delta_s)$ for any $t\in[0,\Delta_s)$. Then, by \eqref{eqn:v-PxB*deltax v upper bound}, \eqref{eqn:eta definition}, \eqref{eqn:xi definition}, we have
    \begin{equation}\label{eqn:xi upper bound}
        \norm{\xi_{\bm{x},\bm{v}}(t)}_{\mathcal{E}}=\norm{t\bm{v}-\bm{P_x}B^*\eta_{\bm{x},\bm{v}}(t)}_{\mathcal{E}}=\norm{t\bm{v}-\bm{P_x}B^*\delta_{\bm{x}}(t\bm{v})}_{\mathcal{E}}\le \epsilon\leq \sigma.
    \end{equation}
    Meanwhile, it follows from  \eqref{eqn:partial Fx delta}, \eqref{eqn:v-PxB*deltax v upper bound}, \eqref{eqn:eta definition} and \eqref{eqn:xi definition} that
    \begin{equation*}
        \mathcal{B}(t)=B\circ D\bm{R_x}(\xi_{\bm{x},\bm{v}}(t))\circ \bm{P_x}B^*=B\circ D\bm{R_x}(t\bm{v}-\bm{P_x}B^*\delta_{\bm{x}}(t\bm{v}))\circ \bm{P_x}B^*=-\frac{\partial F_{\bm{x}}}{\partial \delta}(t\bm{v},\delta_{\bm{x}}(t\bm{v})).
    \end{equation*}
    Then, because of \eqref{eqn:partialx v deltaxv inverse and upper bound}, we know that $\mathcal{B}(t)$ is invertible and 
    \begin{equation}\label{eqn:Bt^-1 bound}
		\norm{\mathcal{B}(t)^{-1}}=\norm{\left(-\frac{\partial F_{\bm{x}}}{\partial \delta}(t\bm{v},\delta_{\bm{x}}(t\bm{v}))\right)^{-1}}\le 2K_L.
	\end{equation}
    Here, we take
    $$
    \Delta_D:=\frac{\Delta_s}{2}.
    $$ 
    Then, for any $t\in [0,\Delta_D]$, by \eqref{eqn:DR D2R bounded above X}, \eqref{eqn:first derivative of retraction equation}, \eqref{eqn:xi upper bound} and \eqref{eqn:Bt^-1 bound}, we have
	\begin{equation}\label{eqn:deta upper bound}
		\norm{\detaxv(t)}_{\mathcal{Y}}\le \norm{\mathcal{B}(t)^{-1}}\norm{B}\norm{D\bm{R_x}(\xixv(t))}\norm{\bm{v}}_{\bm{x}}\le 2K_L\norm{B}K_1:=C_{\eta}.
	\end{equation}
	Meanwhile, by differentiating \eqref{eqn:xi definition} and using \eqref{eqn:Px bounded above X}, \eqref{eqn:deta upper bound}, we have
	\begin{equation}\label{eqn:dxi upper bound}
		\norm{\dxixv(t)}_{\mathcal{E}}\le \norm{\bm{v}}_{    \mathcal{E}}+\norm{\bm{P_x}}\norm{B^*}\norm{\detaxv}_{\mathcal{Y}}\le 1+K_P\norm{B}C_{\eta}:=C_{\xi}.
	\end{equation}

\noindent Moreover, by \eqref{eqn:DR D2R bounded above X}, \eqref{eqn:second derivative of retraction equation}, \eqref{eqn:Bt^-1 bound}, and \eqref{eqn:dxi upper bound}, for any $t\in[0,\Delta_D]$ we obtain
	\begin{equation}\label{eqn:ddeta upper bound}
		\norm{\ddetaxv(t)}_{\mathcal{Y}}\le \norm{\mathcal{B}(t)^{-1}}\norm{B}\norm{D^2\bm{R_x}(\xixv(t))}\norm{\dxixv(t)}^2_{\mathcal{E}}\le 2K_L\norm{B}K_2(C_{\xi})^2:=C_{\ddot{\eta}}.
	\end{equation}
	By some direct calculations, we have
	\[
	\frac{\mathrm{d}^2}{\mathrm{d}t^2} \bm{R}^B_{\bm{x}}(t\bm{v})=-D\bm{R_x}(\xi_{\bm{x},\bm{v}}(t))\circ \bm{P_x}B^*\ddetaxv(t)+D^2\bm{R_x}(\xi_{\bm{x},\bm{v}}(t))[\dxixv(t),\dxixv(t)].
	\]
	Then, for any $t\in [0,\Delta_D]$, by \eqref{eqn:Px bounded above X}, \eqref{eqn:DR D2R bounded above X}, \eqref{eqn:xi upper bound}, \eqref{eqn:dxi upper bound} and \eqref{eqn:ddeta upper bound}, we have 
	\[
	\begin{aligned}
		\norm{\frac{\mathrm{d}^2}{\mathrm{d}t^2} \bm{R}^B_{\bm{x}}(t\bm{v})}_{\mathcal{E}}=&\norm{D\bm{R_x}(\xi_{\bm{x},\bm{v}}(t))[\ddot\xi_{\bm{x},\bm{v}}(t)]+D^2\bm{R_x}(\xi_{\bm{x},\bm{v}}(t))[\dxixv(t),\dxixv(t)]}_{\mathcal{E}}\\
        \le &\norm{D\bm{R_x}(\xi_{\bm{x},\bm{v}}(t))[\bm{P_x}B^*\ddetaxv(t)]}_{\mathcal{E}}+\norm{D^2\bm{R_x}(\xi_{\bm{x},\bm{v}}(t))[\dxixv(t),\dxixv(t)]}_{\mathcal{E}}\\
        \le&\norm{D\bm{R_x}(\xi_{\bm{x},\bm{v}}(t))}\norm{\bm{P_x}}\norm{B}\norm{\ddetaxv(t)}_{\mathcal{Y}}+\norm{D^2\bm{R_x}(\xi_{\bm{x},\bm{v}}(t))}\norm{\dxixv(t)}_{\mathcal{E}}^2\\
		\le &K_1K_P\norm{B}C_{\ddot{\eta}}+K_2(C_{\xi})^2:=\beta_D.
	\end{aligned}
	\]
	We see that $\Delta_D$ and $\beta_D$ depend only on the uniform constants $K_1$, $K_2$, $K_P$, $K_L$, $\sigma$, $\norm{B}$ and are independent of $\bm{x}$ or $\bm{v}$. 
\end{proof}

\subsection{Regularity Properties of \texorpdfstring{$J^B_{\bm{x}}$}{}}

\noindent The uniform estimates for the retraction obtained above can now be transferred to the lifted objective functional. In particular, we first show that \(J^B_{\bm{x}}\) is radially Lipschitz continuously differentiable on \(X\), in the sense of the following definition.

\begin{definition}[radially Lipschitz-$\mathcal{C}^1$ lifted objective functionals on $X$]
Let $X\subset \mathcal{M}^B$. We say that the lifted objective functionals
$J_{\bm{x}}^B$, $\bm{x}\in X$, are radially Lipschitz continuously differentiable on $X$ with uniform constants if there exist constants $\beta_{RL}>0$ and $\Delta_{RL}>0$ such that, for every $\bm{x}\in X$, every $\bm{\xi}\in \T_{\bm{x}}\mathcal{M}^B$ with $\norm{\bm{\xi}}_{\mathcal{E}}=1$, and every $t\in[0,\Delta_{RL}]$, it holds that
\begin{equation*}
\left|
\left.\frac{\mathrm{d}}{\mathrm{d}s} J_{\bm{x}}^B(s\bm{\xi})\right|_{s=t}
-
\left.\frac{\mathrm{d}}{\mathrm{d}s} J_{\bm{x}}^B(s\bm{\xi})\right|_{s=0}
\right|
\le \beta_{RL} t,
\end{equation*}
where $J_{\bm{x}}^B$ is defined in \eqref{eqn:JxB}.
\end{definition}

\begin{theorem}\label{thm:radially lipschitz}
	The radially Lipschitz-$\mathcal{C}^1$ regularity for ${J}^B_{\bm{x}}$ holds on $X$ with $\Delta_{RL} = \Delta_D$ and some $\beta_{RL}>0$.
\end{theorem}
\begin{proof}
    Fix \( \bm{x}\in X \), and let \( \bm{c}^B(t)=\bm{R}_{\bm{x}}^B(t\bm{v}) \) with \( \bm{v}\in \mathrm{T}_{\bm{x}}\mathcal{M}^B \) and \( \|\bm{v}\|_{\bm{x}}=1 \) as in Theorem~\ref{thm:retraction second-order derivative}. For any \( t\in[0,\Delta_D] \), applying the Mean Value Theorem and Theorem~\ref{thm:uniform boundedness of directional second-order retraction} yields
    \[
    \left\|\left.\frac{\mathrm{d}}{\mathrm{d}\tau}\bm{c}^B(\tau)\right|_{\tau=t}
          -\left.\frac{\mathrm{d}}{\mathrm{d}\tau}\bm{c}^B(\tau)\right|_{\tau=0}\right\|_{\mathcal{E}}
    \le \sup_{\tau \in (0, t)} \left\|\frac{\mathrm{d}^2}{\mathrm{d}\tau^2}\bm{c}^B(\tau)\right\|_{\mathcal{E}} t 
    \le \beta_D t .
    \]
    By \eqref{eqn:retraction RxB value at 0 and first-order derivative},
    $\left.\frac{\mathrm{d}}{\mathrm{d}\tau}\bm{c}^B(\tau)\right|_{\tau=0}=\bm{v}$; hence, for all $t\in[0,\Delta_D]$,
    \[
    \left\|\left.\frac{\mathrm{d}}{\mathrm{d}\tau}\bm{c}^B(\tau)\right|_{\tau=t}\right\|_{\mathcal{E}}
    \le \|\bm{v}\|_{\mathcal{E}} + t\beta_D
    =\|\bm{v}\|_{\bm{x}} + t\beta_D
    \le 1+\Delta_D\beta_D
    =: \beta_1 .
    \]
    
 \noindent  Moreover, since \( \bm{c}^B(0)=\bm{x} \) by \eqref{eqn:retraction RxB value at 0 and first-order derivative}, applying the Mean Value Theorem again yields, for all \( t\in[0,\Delta_D] \), we have
    \[
    \|\bm{c}^B(t)-\bm{c}^B(0)\|_{\mathcal{E}}
    \le \sup_{\tau \in (0, t)} \left\|\frac{\mathrm{d}}{\mathrm{d}\tau}\bm{c}^B(\tau)\right\|_{\mathcal{E}} t
    \le \beta_1 t .
    \]
    Consequently, since $\bm{c}^B(0)=\bm{x}$, by \eqref{eqn:boundedness of X}, we have for all \( t\in[0,\Delta_D] \),
    \[
    \|\bm{c}^B(t)\|_{\mathcal{E}}
    \le \|\bm{x}\|_{\mathcal{E}} + t\beta_1
    \le C_1 + \Delta_D\beta_1
    =: \beta_2 .
    \]
    
\noindent  Next, by \eqref{eqn:second-order directional derivative of JxB}, we have
    \[
    \frac{\mathrm{d}^2}{\mathrm{d}\tau^2} J^B_{\bm{x}}(\tau\bm{v})
    =
    \left( \hess J(\bm{c}^B(\tau))\left[\frac{\mathrm{d}}{\mathrm{d}\tau}\bm{c}^B(\tau)\right],
           \frac{\mathrm{d}}{\mathrm{d}\tau}\bm{c}^B(\tau) \right)_{\mathcal{E}}
    +
    \left( \grad J(\bm{c}^B(\tau)), \frac{\mathrm{d}^2}{\mathrm{d}\tau^2}\bm{c}^B(\tau) \right)_{\mathcal{E}} .
    \]
    Since $\grad J$ and $\hess J$ are bounded mappings, there exists a constant
    $\beta_3>0$ such that, for all $\bm{y}\in\mathcal{M}^B$ with $\|\bm{y}\|_{\mathcal{E}}\le \beta_2$, it holds that
    \[
    \|\grad J(\bm{y})\|_{\mathcal{E}}\le \beta_3,
    \qquad
    \|\hess J(\bm{y})\|\le \beta_3 .
    \]
    Therefore, for any $\tau\in[0,\Delta_D]$, using the Cauchy--Schwarz inequality together with
    $\left\|\frac{\mathrm{d}}{\mathrm{d}\tau}\bm{c}^B(\tau)\right\|_{\mathcal{E}}\le \beta_1$ and
    $\left\|\frac{\mathrm{d}^2}{\mathrm{d}\tau^2}\bm{c}^B(\tau)\right\|_{\mathcal{E}}\le \beta_D$, we obtain
    \[
    \left|\frac{\mathrm{d}^2}{\mathrm{d}\tau^2} J^B_{\bm{x}}(\tau\bm{v})\right|
    \le \|\hess J(\bm{c}^B(\tau))\|\left\|\frac{\mathrm{d}}{\mathrm{d}\tau}\bm{c}^B(\tau)\right\|_{\mathcal{E}}^{2}
       + \|\grad J(\bm{c}^B(\tau))\|_{\mathcal{E}}\left\|\frac{\mathrm{d}^2}{\mathrm{d}\tau^2}\bm{c}^B(\tau)\right\|_{\mathcal{E}}
    \le \beta_3(\beta_1^2+\beta_D)
    =: \beta_{RL}.
    \]
    Consequently, for any $t\in[0,\Delta_D]$, we have
    \begin{align*}
    \left|
    \left.\frac{\mathrm{d}}{\mathrm{d}\tau}J^B_{\bm{x}}(\tau\bm{v})\right|_{\tau=t}
    -
    \left.\frac{\mathrm{d}}{\mathrm{d}\tau}J^B_{\bm{x}}(\tau\bm{v})\right|_{\tau=0}
    \right|
    &\le \int_{0}^{t}\left|\frac{\mathrm{d}^2}{\mathrm{d}\tau^2}J^B_{\bm{x}}(\tau\bm{v})\right|\,\mathrm{d}\tau
    \le \beta_{RL}\,t .
    \end{align*}
    This verifies the radially Lipschitz-$\mathcal{C}^1$ regularity with constants $\beta_{RL}$ and $\Delta_{RL}:=\Delta_D$.
\end{proof}

\noindent As a direct consequence of the radially Lipschitz-$\mathcal{C}^1$ regularity, the pullback objective satisfies the following quantitative estimate on \(X\).
\begin{proposition}\label{prop:LC1 second order property}
    Let $\bm{x}\in X$ and $\bm{v}\in \T_{\bm{x}}\mathcal{M}^B$ such that $\norm{\bm{v}}_{\bm{x}}\le \Delta_{RL}$. Then, we have
    \begin{equation}\label{eqn:LC1 second order property}
        J_{\bm{x}}^B(\bm{v})=J^B(\bm{R_x}^B(\bm{v}))\leq J^B(\bm{x})+\left(\grad J^B(\bm{x}),\bm{v}\right)_{\bm{x}}+\frac{\beta_{RL}}{2}\norm{\bm{v}}_{\bm{x}}^2.
    \end{equation}
\end{proposition}
\begin{proof}
Applying Theorem \ref{thm:radially lipschitz}, we obtain
\begin{equation}\label{eqn:LC1 second order property part}
    \begin{aligned}
		\abs{{J}^B_{\bm{x}}(\bm{v})-{J}^B_{\bm{x}}(\bm{0}_{\mathcal{E}})-\left(\grad J^B_{\bm{x}}(\bm{0}_{\mathcal{E}}),\bm{v}\right)_{\bm{x}}}&=\abs{{J}^B_{\bm{x}}(\bm{v})-{J}^B_{\bm{x}}(\bm{0}_{\mathcal{E}})-\norm{\bm{v}}_{\bm{x}}\frac{\mathrm{d}}{\mathrm{d}\tau}{J}^B_{\bm{x}}\left(\tau\frac{\bm{v}}{\norm{\bm{v}}}_{\bm{x}}\right)\mid_{\tau=0}}\\
		&=\abs{\int_0^{\norm{\bm{v}}_{\bm{x}}}\left(\frac{\mathrm{d}}{\mathrm{d}\tau}{J}^B_{\bm{x}}\left(\tau\frac{\bm{v}}{\norm{\bm{v}}}_{\bm{x}}\right)-\frac{\mathrm{d}}{\mathrm{d}\tau}{J}^B_{\bm{x}}\left(\tau\frac{\bm{v}}{\norm{\bm{v}}}_{\bm{x}}\right)\mid_{\tau=0}\right)\,\mathrm{d}\tau}\\
		&\le \int_0^{\norm{\bm{v}}_{\bm{x}}}\beta_{RL}\tau\,\mathrm{d}\tau=\frac{1}{2}\beta_{RL}\norm{\bm{v}}^2_{\bm{x}}.
	\end{aligned}
\end{equation}
Recall from the definition of $J_{\bm{x}}^B$ in \eqref{eqn:JxB} that 
\begin{equation}\label{eqn:JxB0=Jx,JxBv=JBRxv}
    {J}^B_{\bm{x}}(\bm{0}_{\mathcal{E}})=J^B(\bm{x})\quad\text{and}\quad J_{\bm{x}}^B(\bm{v})=J^B(\bm{R}_{\bm{x}}^B(\bm{v})).
\end{equation}
Combining \eqref{eqn:first-order derivative of JB and JxB}, \eqref{eqn:LC1 second order property part}, and \eqref{eqn:JxB0=Jx,JxBv=JBRxv}, we obtain the desired inequality \eqref{eqn:LC1 second order property}.
\end{proof}

\noindent In addition to radial Lipschitz-$\mathcal{C}^1$ regularity, we also need uniform boundedness of first- and second-order quantities associated with the lifted objective functionals. The following theorem shows that the Riemannian gradient of \(J^B\) and the second derivative of the pullback \(J_{\bm{x}}^B\) at the origin are uniformly bounded on \(X\).

\begin{theorem}\label{thm:boundedness of DJ and D^2J}
	There exist $\beta_g>0,\beta_H>0$ such that, for any $\bm{x}\in X$, $\bm{v},\bm{u}\in\T_{\bm{x}}\mathcal{M}^B$, with $\norm{\bm{v}}_{\bm{x}}=1$ and $\norm{\bm{u}}_{\bm{x}}=1$, we have
	\[
	\norm{\grad J^B(\bm{x})}_{\mathcal{E}}\le \beta_g,
	\]
	and 
	\[
	\abs{\left(\hess J^B_{\bm{x}}(\bm{0}_{\mathcal{E}})[\bm{v}],\bm{u}\right)_{\mathcal{E}}}\le \beta_H.
	\]
\end{theorem}
\begin{proof}
	Since $\grad J$ is a bounded mapping and $X$ is bounded by \eqref{eqn:boundedness of X}, there exists a constant $\beta_g>0$ such that
    \begin{equation}\label{eqn:upper bound of DJ}
    \|\grad J(\bm{x})\|_{\mE} \le \beta_g,\qquad \forall\,\bm{x}\in X.
    \end{equation}
	Then, for any $\bm{x}\in X$ and any $\bm{v}\in \mathrm{T}_{\bm{x}}\mathcal{M}^B$ with $\norm{\bm{v}}_{\bm{x}}=1$, \eqref{eqn:first-order derivative of JB and JxB} in Theorem~\ref{thm: taylor} yields
    \begin{multline*}
            \bigl|(\grad J^B(\bm{x}),\bm{v})_{\mathcal{E}}\bigr|
    =\bigl|(\grad J^B(\bm{x}),\bm{v})_{\bm{x}}\bigr|
    =\bigl|(\grad J(\bm{x}),\bm{v})_{\mathcal{E}}\bigr| \\
    \le \|\grad J(\bm{x})\|_{\mathcal{E}}\,\|\bm{v}\|_{\mathcal{E}}
    =\|\grad J(\bm{x})\|_{\mathcal{E}}\,\|\bm{v}\|_{\bm{x}}
    \le \beta_g .
    \end{multline*}
    Taking $\bm{v}=\grad J^B(\bm{x})/\|\grad J^B(\bm{x})\|_{\mathcal{E}}$ whenever
    $\grad J^B(\bm{x})\neq 0$ (and noting that the claim is trivial otherwise) implies
    \begin{equation}
        \|\grad J^B(\bm{x})\|_{\mathcal{E}}\le \beta_g.
    \end{equation}
	Next, since $\hess J$ and $\ell$ are bounded mappings, there exist $C_2,C_3>0$ such that for any $\bm{x}\in X$, we have
    \begin{equation}\label{eqn:upper bound of D2J and ellx}
    \norm{\hess J(\bm{x})}\le C_2,\quad  \norm{\ell_{\bm{x}}}\le C_3.
    \end{equation}
	Then, for any $\bm{x}\in X$ and any $\bm{v},\bm{u}\in \mathrm{T}_{\bm{x}}\mathcal{M}^B$ with $\|\bm{v}\|_{\bm{x}}=1$ and $\|\bm{u}\|_{\bm{x}}=1$, it follows from
    \eqref{eqn:second-order direvative of JxB} in Theorem~\ref{thm: taylor},
    together with \eqref{Lambda}, \eqref{eqn:Px bounded above X}, \eqref{eqn:Lx-1 bounded above X}, \eqref{eqn:upper bound of DJ}, \eqref{eqn:upper bound of D2J and ellx}, and the Cauchy--Schwarz inequality, that
    \[
    \begin{aligned}
    \left|\bigl(\hess J^B_{\bm{x}}(\bm{0}_{\mathcal{E}})[\bm{v}],\bm{u}\bigr)_{\mathcal{E}}\right|
    &=
    \left|\Bigl(\bigl[\hess J(\bm{x})+\bm{\pi}_{\bm{x}}\bigr]\bm{v},\bm{u}\Bigr)_{\mathcal{E}}\right| \\
    &\le \|\hess J(\bm{x})\|\,\|\bm{v}\|_{\mathcal{E}}\,\|\bm{u}\|_{\mathcal{E}}
         + \|\grad J(\bm{x})\|_{\mathcal{E}}\,\|\Lambda_{\bm{x}}\|\,\|\ell_{\bm{x}}(\bm{v},\bm{u})\|_{\mathcal{E}} \\
    &\le \Bigl(\|\hess J(\bm{x})\| + \|\grad J(\bm{x})\|_{\mathcal{E}}\,\|\Lambda_{\bm{x}}\|\,\|\ell_{\bm{x}}\|\Bigr)
          \|\bm{v}\|_{\bm{x}}\,\|\bm{u}\|_{\bm{x}} \\
    &=\|\hess J(\bm{x})\| + \|\grad J(\bm{x})\|_{\mathcal{E}}\,\|\mathrm{Id}_{\mathcal{E}} - \bm{P_x}B^*L_{\bm{x}}^{-1}B\|\,\|\ell_{\bm{x}}\|\\
    &\le C_2 + \beta_g\Bigl(\|\mathrm{Id}_{\mathcal{E}}\| + \|\bm{P}_{\bm{x}}\|\,\|B^*\|\,\|L_{\bm{x}}^{-1}\|\,\|B\|\Bigr)C_3 \\
    &\le C_2 + \beta_g\bigl(1+K_LK_P\|B\|^2\bigr)C_3
    =: \beta_H .
    \end{aligned}
    \]
    This completes the proof.
\end{proof}

\section{Projection-Induced Riemannian Metric and Gradient}\label{sec:projection-induced gradient}


\noindent The projection \(\bm{P}_{\bm{x}}^B\) constructed in Proposition~\ref{prop:projection to TxMb} is an explicit map from the ambient Hilbert space \(\mathcal{E}\) onto \(\T_{\bm{x}}\mathcal{M}^B\), and therefore provides a natural way to construct tangent directions. However, since \(\bm{P}_{\bm{x}}^B\) is not assumed to be orthogonal with respect to the ambient inner product, directly projecting the ambient gradient does not, in general, yield the Riemannian gradient associated with the ambient-induced metric. Indeed, by Theorem~\ref{thm: taylor}, \(\grad J^B(\bm{x})\) is characterized by
\begin{equation}\label{eqn:equation for riemannian gradient}
    \left(\grad J^B(\bm{x}),\bm{v}\right)_{\bm{x}} = \left(\grad J(\bm{x}),\bm{v}\right)_{\mathcal{E}}, \qquad \forall\,\bm{v}\in\T_{\bm{x}}\mathcal{M}^B.
\end{equation}
Thus, computing \(\grad J^B(\bm{x})\) requires solving the tangent-space Riesz representation problem \eqref{eqn:equation for riemannian gradient}, which can be computationally expensive in practice. To resolve this, we introduce a projection-induced Riemannian metric \(\widetilde{g}_{\bm{x}}\) on \(\T_{\bm{x}}\mathcal{M}^B\), under which \(\bm{P}_{\bm{x}}^B(\bm{P}_{\bm{x}}^B)^*\grad J(\bm{x})\) is the exact Riemannian gradient. The induced norm of \(\widetilde{g}_{\bm{x}}\) is uniformly equivalent to the ambient norm on bounded subsets of \(\mathcal{M}^B\), providing the analytic basis for the subsequent convergence analysis.

\noindent Throughout this section, we use the same bounded set \(X\) as defined in \eqref{eqn:X} in Section~\ref{sec:properties}. We first record the boundedness of the projected operator \(\bm{P}_{\bm{x}}^B\): pointwise on \(\mathcal{M}^B\), and uniformly on \(X\).

\begin{lemma}\label{lem:PxB upper bound}
    For every \(\bm{x}\in\mathcal{M}^B\), the projection \(\bm{P}_{\bm{x}}^B\) defined in \eqref{eqn:PxB def} is a bounded linear operator on \(\mathcal{E}\); in particular, its adjoint \((\bm{P}_{\bm{x}}^B)^*\) is well defined. Moreover, there exists a constant \(K_B>0\), depending only on the bounded set \(X\), such that
    \begin{equation}\label{eqn:PxB upper bound}
        \|\bm{P}_{\bm{x}}^B\|\le K_B, \qquad \forall\,\bm{x}\in X.
    \end{equation}
\end{lemma}

\begin{proof}
By the definition \eqref{eqn:PxB def}, $\bm{P}_{\bm{x}}^B=\bm{P}_{\bm{x}}-\bm{P}_{\bm{x}}B^*L_{\bm{x}}^{-1}B\bm{P}_{\bm{x}}$. For every \(\bm{x}\in\mathcal{M}^B\), the operators \(\bm{P}_{\bm{x}}\) and \(B\) are bounded, and Assumption~\ref{asmp:Lx invertible}, together with the Bounded Inverse Theorem, implies that \(L_{\bm{x}}^{-1}\) is bounded. Since \(\|B^*\|=\|B\|\), it follows that
\[
    \|\bm{P}_{\bm{x}}^B\| \le \|\bm{P}_{\bm{x}}\| + \|\bm{P}_{\bm{x}}\|^2\|B\|^2\|L_{\bm{x}}^{-1}\| < \infty,
\]
which proves the first assertion. For \(\bm{x}\in X\), the uniform bounds \eqref{eqn:Px bounded above X} and \eqref{eqn:Lx-1 bounded above X} give
\[
    \|\bm{P}_{\bm{x}}^B\| \le K_P+K_P^2K_L\|B\|^2,
\]
and thus \eqref{eqn:PxB upper bound} holds with $K_B:=K_P+K_P^2K_L\|B\|^2$.
\end{proof}

\noindent We next record a useful identity relating the ambient gradient, the ambient-induced Riemannian gradient, and the adjoint of the projection \(\bm{P}_{\bm{x}}^B\).

\begin{lemma}\label{lem:PxBstar gradient identity}
For any \(\bm{x}\in\mathcal{M}^B\), one has
\begin{equation}\label{eqn:PxBstar gradient identity}
    (\bm{P}_{\bm{x}}^B)^*\grad J(\bm{x})
    =
    (\bm{P}_{\bm{x}}^B)^*\grad J^B(\bm{x}).
\end{equation}
\end{lemma}

\begin{proof}
For any \(\bm{u}\in\mathcal{E}\), Proposition~\ref{prop:projection to TxMb} gives $\bm{P}_{\bm{x}}^B\bm{u}\in\T_{\bm{x}}\mathcal{M}^B$. Hence, by \eqref{eqn:equation for riemannian gradient}, we have
\[
\left((\bm{P}_{\bm{x}}^B)^*\grad J(\bm{x}),\bm{u}\right)_{\mathcal{E}}
=
\left(\grad J(\bm{x}),\bm{P}_{\bm{x}}^B\bm{u}\right)_{\mathcal{E}}
=
\left(\grad J^B(\bm{x}),\bm{P}_{\bm{x}}^B\bm{u}\right)_{\mathcal{E}}
=
\left((\bm{P}_{\bm{x}}^B)^*\grad J^B(\bm{x}),\bm{u}\right)_{\mathcal{E}}.
\]
Since \(\bm{u}\in\mathcal{E}\) is arbitrary, \eqref{eqn:PxBstar gradient identity} follows immediately.
\end{proof}

\begin{proposition}[Projection-Induced Riemannian Metric and Exact Gradient]\label{prop:projection induced metric gradient}
For $\bm{x}\in \mathcal{M}^B$, let $\T_{\bm{x}}:=\T_{\bm{x}}\mathcal{M}^B$ and define the operator
\begin{equation}\label{eqn:Ax def}
    A_{\bm{x}} := \bm{P}_{\bm{x}}^B(\bm{P}_{\bm{x}}^B)^*\big|_{\T_{\bm{x}}} : \T_{\bm{x}}\to \T_{\bm{x}} .
\end{equation}
Then, \(A_{\bm{x}}\) is self-adjoint, positive definite, and boundedly invertible on \(\T_{\bm{x}}\). Let
\begin{equation}\label{eqn:Kx def}
    K_{\bm{x}}:=A_{\bm{x}}^{-1}.
\end{equation}
The bilinear form
\begin{equation}\label{eqn:projection induced metric}
    \widetilde{g}_{\bm{x}}(\bm{u},\bm{v}) := (K_{\bm{x}}\bm{u},\bm{v})_{\mathcal{E}}, \qquad \forall\,\bm{u},\bm{v}\in \T_{\bm{x}},
\end{equation}
defines a Hilbert inner product on \(\T_{\bm{x}}\). We denote the induced norm by
\begin{equation}\label{eqn:projection induced norm}
    \|\bm{v}\|_{\widetilde{g}_{\bm{x}}} := \sqrt{\widetilde{g}_{\bm{x}}(\bm{v},\bm{v})} = \sqrt{(K_{\bm{x}}\bm{v},\bm{v})_{\mathcal{E}}}, \qquad \forall\,\bm{v}\in \T_{\bm{x}}.
\end{equation}
Under Assumption~\ref{asmp:boundedness of P}, the section \(\bm{x}\mapsto\widetilde g_{\bm{x}}\) is continuous, so that \(\widetilde g:=\{\widetilde g_{\bm{x}}\}_{\bm{x}\in\mathcal{M}^B}\) is a continuous strong Riemannian metric on the Hilbert manifold \(\mathcal{M}^B\) \cite[Ch.~VII]{lang2012fundamentals}; we refer to it as the projection-induced Riemannian metric. Moreover, the vector
\begin{equation}\label{eqn:projection induced gradient def}
    \bm{G}_{\bm{x}} := \bm{P}_{\bm{x}}^B(\bm{P}_{\bm{x}}^B)^*\grad J(\bm{x}) \in \T_{\bm{x}}\mathcal{M}^B
\end{equation}
is the exact Riemannian gradient of \(J^B\) with respect to \(\widetilde{g}\), namely
\begin{equation}\label{eqn:projection induced exact gradient}
    \widetilde{g}_{\bm{x}}(\bm{G}_{\bm{x}},\bm{v}) = DJ^B(\bm{x})[\bm{v}], \qquad \forall\,\bm{v}\in \T_{\bm{x}}.
\end{equation}
Finally, on \(X\), the norm induced by the projection-induced metric is uniformly equivalent to the ambient norm:
\begin{equation}\label{eqn:metric equivalence compact}
    \frac{1}{K_B}\|\bm{v}\|_{\mathcal{E}} \le \|\bm{v}\|_{\widetilde{g}_{\bm{x}}} \le \|\bm{v}\|_{\mathcal{E}}, \qquad \forall\,\bm{x}\in X,\ \bm{v}\in \T_{\bm{x}}.
\end{equation}
\end{proposition}

\begin{proof}
Fix $\bm{x}\in \mathcal{M}^B$. For any \(\bm{u},\bm{v}\in \T_{\bm{x}}\), the definition of \(A_{\bm{x}}\) gives
\begin{equation}\label{eqn:Ax=PBx*}
    (A_{\bm{x}}\bm{u},\bm{v})_{\mathcal{E}} = \left((\bm{P}_{\bm{x}}^B)^*\bm{u}, (\bm{P}_{\bm{x}}^B)^*\bm{v}\right)_{\mathcal{E}}.
\end{equation}
Thus \(A_{\bm{x}}\) is self-adjoint and positive semidefinite. Since \(\bm{P}_{\bm{x}}^B\bm{v}=\bm{v}\) for every \(\bm{v}\in \T_{\bm{x}}\), we have
\[
    \|\bm{v}\|_{\mathcal{E}}^2 = (\bm{P}_{\bm{x}}^B\bm{v},\bm{v})_{\mathcal{E}} = \left(\bm{v}, (\bm{P}_{\bm{x}}^B)^*\bm{v}\right)_{\mathcal{E}} \le \|\bm{v}\|_{\mathcal{E}} \|(\bm{P}_{\bm{x}}^B)^*\bm{v}\|_{\mathcal{E}}.
\]
Hence $\|(\bm{P}_{\bm{x}}^B)^*\bm{v}\|_{\mathcal{E}}\ge\|\bm{v}\|_{\mathcal{E}}$, and therefore
\[
    (A_{\bm{x}}\bm{v},\bm{v})_{\mathcal{E}} = \|(\bm{P}_{\bm{x}}^B)^*\bm{v}\|_{\mathcal{E}}^2 \ge \|\bm{v}\|_{\mathcal{E}}^2,
\]
so the bilinear form \eqref{eqn:Ax=PBx*} is coercive on \(\T_{\bm{x}}\) with coercivity constant $1$. Meanwhile, by the Cauchy--Schwarz inequality, \eqref{eqn:Ax=PBx*}, and the boundedness of $\bm{P}_{\bm{x}}^B$ recorded in Lemma~\ref{lem:PxB upper bound}, we have
\[
    (A_{\bm{x}}\bm{u},\bm{v})_{\mathcal{E}} \le \|(\bm{P}_{\bm{x}}^B)^*\bm{u}\|_{\mathcal{E}} \|(\bm{P}_{\bm{x}}^B)^*\bm{v}\|_{\mathcal{E}} \le \|\bm{P}_{\bm{x}}^B\|^2 \|\bm{u}\|_{\mathcal{E}}\|\bm{v}\|_{\mathcal{E}},
\]
and so the form is also bounded. By the Lax--Milgram theorem, \(A_{\bm{x}}\) is boundedly invertible on \(\T_{\bm{x}}\), and the coercivity constant $1$ yields $\|K_{\bm{x}}\|\le 1$. In the sense of quadratic forms on \(\T_{\bm{x}}\), the two estimates above read
\begin{equation}\label{eqn:Ax quadratic form bounds}
    \mathrm{Id}_{\T_{\bm{x}}} \le A_{\bm{x}} \le \|\bm{P}_{\bm{x}}^B\|^2\,\mathrm{Id}_{\T_{\bm{x}}}, \qquad \forall\,\bm{x}\in\mathcal{M}^B,
\end{equation}
and hence
\begin{equation}\label{eqn:Kx quadratic form bounds}
    \frac{1}{\|\bm{P}_{\bm{x}}^B\|^2}\,\mathrm{Id}_{\T_{\bm{x}}} \le K_{\bm{x}} \le \mathrm{Id}_{\T_{\bm{x}}}, \qquad \forall\,\bm{x}\in\mathcal{M}^B.
\end{equation}
In particular, \(K_{\bm{x}}\) is self-adjoint and positive definite, so \(\widetilde g_{\bm{x}}\) is a symmetric positive-definite bilinear form on \(\T_{\bm{x}}\), and therefore an inner product, whose induced norm is equivalent to the ambient norm by \eqref{eqn:Kx quadratic form bounds}. Since \(\T_{\bm{x}}\) is complete under the ambient Hilbert norm $\norm{\cdot}_{\mathcal{E}}$, it is also complete under \(\|\cdot\|_{\widetilde g_{\bm{x}}}\). For \(\bm{x}\in X\), Lemma~\ref{lem:PxB upper bound} gives $\|\bm{P}_{\bm{x}}^B\|\le K_B$, so \eqref{eqn:Kx quadratic form bounds} strengthens to
\[
    \frac{1}{K_B^2}\|\bm{v}\|_{\mathcal{E}}^2 \le (K_{\bm{x}}\bm{v},\bm{v})_{\mathcal{E}} \le \|\bm{v}\|_{\mathcal{E}}^2, \qquad \forall\,\bm{x}\in X,\ \bm{v}\in \T_{\bm{x}};
\]
taking square roots proves \eqref{eqn:metric equivalence compact}.

\noindent We next prove the continuity of the section \(\bm{x}\mapsto\widetilde g_{\bm{x}}\). Since \(K_{\bm{x}}\) and \(K_{\bm{y}}\) act on different fibers, we compare their extensions to \(\mathcal{E}\). Define $\widehat A_{\bm{x}}:=A_{\bm{x}}\bm{P}_{\bm{x}}^B+(I-\bm{P}_{\bm{x}}^B)\in\mathcal{L}(\mathcal{E})$, which is invertible with \(\widehat A_{\bm{x}}^{-1}=K_{\bm{x}}\bm{P}_{\bm{x}}^B+(I-\bm{P}_{\bm{x}}^B)=:\widehat K_{\bm{x}}\). By Assumption~\ref{asmp:boundedness of P}, \(\bm{x}\mapsto \bm{P}_{\bm{x}}\) is continuous in operator norm, and hence so are \(\bm{x}\mapsto \bm{P}_{\bm{x}}^B\) and \(\bm{x}\mapsto \widehat A_{\bm{x}}\). Moreover, \(\|K_{\bm{x}}\|\le 1\) and the uniform bound \(\|\bm{P}_{\bm{x}}^B\|\le K_{P^B}\) give \(\|\widehat K_{\bm{x}}\|\le 1+2K_{P^B}=:C\). The resolvent identity \(\widehat K_{\bm{x}}-\widehat K_{\bm{y}}=\widehat K_{\bm{x}}(\widehat A_{\bm{y}}-\widehat A_{\bm{x}})\widehat K_{\bm{y}}\), valid in \(\mathcal{L}(\mathcal{E})\), then yields
\[
\|\widehat K_{\bm{x}}-\widehat K_{\bm{y}}\|\le C^2\|\widehat A_{\bm{x}}-\widehat A_{\bm{y}}\|,
\]
so \(\bm{x}\mapsto \widehat K_{\bm{x}}\) is continuous in operator norm. Since \(\widetilde g_{\bm{x}}(\bm{u},\bm{v})=(\widehat K_{\bm{x}}\bm{u},\bm{v})_{\mathcal{E}}\) for \(\bm{u},\bm{v}\in\T_{\bm{x}}\mathcal{M}^B\), the map \((\bm{x},\bm{u},\bm{v})\mapsto\widetilde g_{\bm{x}}(\bm{u},\bm{v})\) is jointly continuous on \(\T\mathcal{M}^B\oplus\T\mathcal{M}^B\); equivalently, \(\widetilde g\) is continuous in every local trivialization of the bundle of symmetric bilinear forms on \(\T\mathcal{M}^B\). Since each \(\widetilde g_{\bm{x}}\) induces a norm equivalent to the ambient fiber norm, \(\widetilde g\) is a continuous strong Riemannian metric on \(\mathcal{M}^B\) in the sense of \cite[Ch.~VII]{lang2012fundamentals}.


\noindent We now proceed to prove \eqref{eqn:projection induced exact gradient}. By \eqref{eqn:PxBstar gradient identity} and \eqref{eqn:Ax def}, we have
\[
    \bm{G}_{\bm{x}} = \bm{P}_{\bm{x}}^B(\bm{P}_{\bm{x}}^B)^*\grad J(\bm{x})=A_{\bm{x}}\grad J^B(\bm{x}).
\]
For any \(\bm{v}\in \T_{\bm{x}}\), it follows that
\[
    \widetilde{g}_{\bm{x}}(\bm{G}_{\bm{x}},\bm{v}) = (K_{\bm{x}}\bm{G}_{\bm{x}},\bm{v})_{\mathcal{E}} = (K_{\bm{x}}A_{\bm{x}}\grad J^B(\bm{x}),\bm{v})_{\mathcal{E}} = (\grad J^B(\bm{x}),\bm{v})_{\mathcal{E}} = DJ^B(\bm{x})[\bm{v}],
\]
which proves \eqref{eqn:projection induced exact gradient}.
\end{proof}

\begin{corollary}[Uniform Norm Equivalences]\label{cor:uniform norm equivalences}
Under the same assumptions as Proposition~\ref{prop:projection induced metric gradient}, for every \(\bm{x}\in X\), the following estimates hold:
\begin{equation}\label{eqn:Pstar gradient equivalence}
    \|\grad J^B(\bm{x})\|_{\mathcal{E}}
    \le
    \|\bm{G}_{\bm{x}}\|_{\widetilde{g}_{\bm{x}}}
    =\|(\bm{P}_{\bm{x}}^B)^*\grad J(\bm{x})\|_{\mathcal{E}}
    \le
    K_B\|\grad J^B(\bm{x})\|_{\mathcal{E}},
\end{equation}
\begin{equation}\label{eqn:Gx gradient equivalence}
    \|\grad J^B(\bm{x})\|_{\mathcal{E}}
    \le
    \|\bm{G}_{\bm{x}}\|_{\mathcal{E}}
    \le
    K_B^2\|\grad J^B(\bm{x})\|_{\mathcal{E}},
\end{equation}
and
\begin{equation}\label{eqn:Gx angle estimate}
    \left(
        \bm{G}_{\bm{x}},
        \grad J^B(\bm{x})
    \right)_{\mathcal{E}}
    \ge
    \frac{1}{K_B^2}
    \|\bm{G}_{\bm{x}}\|_{\mathcal{E}}
    \|\grad J^B(\bm{x})\|_{\mathcal{E}} .
\end{equation}
\end{corollary}

\begin{proof}
We first establish \eqref{eqn:Pstar gradient equivalence}. Using
\[
\|\bm{G}_{\bm{x}}\|_{\widetilde{g}_{\bm{x}}}
= \sqrt{(K_{\bm{x}}\bm{G}_{\bm{x}},\bm{G}_{\bm{x}})_{\mathcal{E}}}
= \sqrt{(\grad J(\bm{x}),\bm{P}_{\bm{x}}^B(\bm{P}_{\bm{x}}^B)^*\grad J(\bm{x}))_{\mathcal{E}}}
= \|(\bm{P}_{\bm{x}}^B)^*\grad J^B(\bm{x})\|_{\mathcal{E}},
\]
together with \eqref{eqn:metric equivalence compact}, we obtain \eqref{eqn:Pstar gradient equivalence}.
For \eqref{eqn:Gx gradient equivalence}, the upper bound follows from \(\bm{G}_{\bm{x}}=A_{\bm{x}}\grad J^B(\bm{x})\) and \(\|A_{\bm{x}}\|\le K_B^2\). The lower bound follows from
\[
    \|\bm{G}_{\bm{x}}\|_{\mathcal{E}}
    \|\grad J^B(\bm{x})\|_{\mathcal{E}}
    \ge
    (\bm{G}_{\bm{x}},\grad J^B(\bm{x}))_{\mathcal{E}}
    =
    (A_{\bm{x}}\grad J^B(\bm{x}),\grad J^B(\bm{x}))_{\mathcal{E}}
    \ge
    \|\grad J^B(\bm{x})\|_{\mathcal{E}}^2.
\]
Finally, for \eqref{eqn:Gx angle estimate}, combining $(\bm{G}_{\bm{x}},\grad J^B(\bm{x}))_{\mathcal{E}}\ge\|\grad J^B(\bm{x})\|_{\mathcal{E}}^2$ with \(\|\bm{G}_{\bm{x}}\|_{\mathcal{E}} \le K_B^2\|\grad J^B(\bm{x})\|_{\mathcal{E}}\) gives
\[
    (\bm{G}_{\bm{x}},\grad J^B(\bm{x}))_{\mathcal{E}}
    \ge
    \|\grad J^B(\bm{x})\|_{\mathcal{E}}^2
    \ge
    \frac{1}{K_B^2}\|\bm{G}_{\bm{x}}\|_{\mathcal{E}}\|\grad J^B(\bm{x})\|_{\mathcal{E}}.
\]
This completes the proof.
\end{proof}

\begin{remark}[Variable Riemannian Metric Interpretation]
\label{rem:variable metric interpretation}

\noindent Proposition~\ref{prop:projection induced metric gradient} reveals the geometric mechanism underlying the projection-induced gradient construction. For any $\bm{x}\in \mathcal{M}^B$, the explicitly computed vector $\bm G_{\bm{x}}=\bm P_{\bm{x}}^B(\bm P_{\bm{x}}^B)^*\operatorname{grad}J(\bm{x})$ is not merely a heuristic surrogate for $\grad J^B(\bm{x})$. Rather, under the projection-induced metric $\widetilde g_{\bm{x}}$, it serves as the exact Riemannian gradient of the objective $J^B$.

\noindent This shift in perspective is central to our computational framework which is shown in our companion work \cite{song2026optimization2}. As illustrated in Figure~\ref{fig:projection_metric_framework}, adopting $\widetilde g_{\bm{x}}$ allows us to bypass the computationally implicit tangent-space Riesz solve required by the ambient metric. Instead, it unlocks a direct, operator-based evaluation of the exact Riemannian gradient.

\begin{figure}[H]
\centering
\begin{tikzpicture}[
node distance=4cm,
>=latex,
every node/.style={align=center}
]

\node (metric1) at (0,1.5)
{$(\cdot,\cdot)_{\mathcal E}$};

\node (metric2) at (8,1.5)
{$ \widetilde g_{\bm{x}}(\bm{u},\bm{v}) = (K_{\bm{x}}\bm{u},\bm{v})_{\mathcal E} $};

\draw[->,thick]
(metric1) -- node[above]
{$K_{\bm{x}}=\left(\bm P_{\bm{x}}^B(\bm P_{\bm{x}}^B)^*\big|_{\T_{\bm{x}}\mathcal{M}^B}\right)^{-1}$}
(metric2);

\node (grad1) at (0,0)
{$ \operatorname{grad}J^B(\bm{x}) $};

\node (grad2) at (8,0)
{$ \bm G_{\bm{x}} = \bm P_{\bm{x}}^B(\bm P_{\bm{x}}^B)^* \operatorname{grad}J(\bm{x}) $};

\draw[<->,thick]
(grad1) -- (grad2);

\node (comp1) at (0,-1.5)
{Implicit \\Riesz representative computation};

\node (comp2) at (8,-1.5)
{Explicit \\Direct operator evaluation};

\draw[<->,thick]
(comp1) -- (comp2);

\draw[dashed]
(metric1) -- (grad1);

\draw[dashed]
(metric2) -- (grad2);

\draw[dashed]
(grad1) -- (comp1);

\draw[dashed]
(grad2) -- (comp2);

\end{tikzpicture}

\caption{Projection-induced Riemannian metric and gradient.}
\label{fig:projection_metric_framework}
\end{figure}
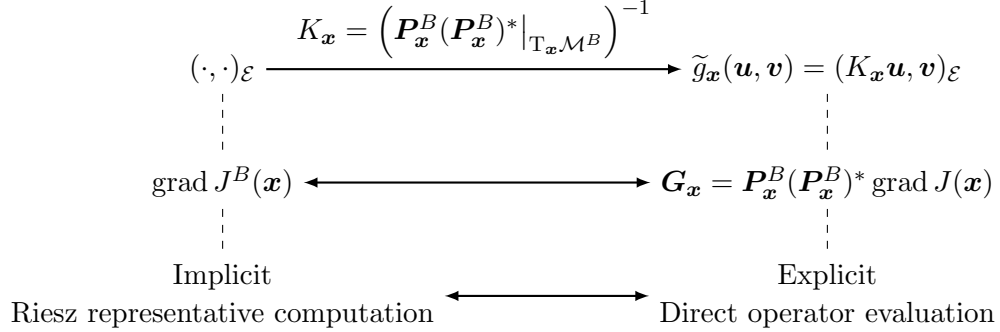

\noindent The role of the metric $\widetilde g_{\bm{x}}$ is primarily analytical. Although the gradient characterization is formulated with respect to $\widetilde g_{\bm{x}}$, the identity $\widetilde g_{\bm{x}}(\bm G_{\bm{x}},\bm v)=(\operatorname{grad}J(\bm{x}),\bm v)_{\mathcal E}$ ensures the algorithm can be implemented without explicitly assembling or applying the inverse operator $K_{\bm{x}}$. Furthermore, the uniform norm equivalence \eqref{eqn:metric equivalence compact} inherently links the projection-induced metric with the ambient-induced metric, thereby enabling the convergence analysis to be carried out within the standard Riemannian framework.

\end{remark}

\section{Foundations of Algorithmic Design I: Riemannian Line-Search Methods}\label{sec:algorithmLS}

\noindent The geometric constructions and analytic estimates developed in Sections~\ref{sec:geometry of MB}--\ref{sec:projection-induced gradient}, including the retraction, the projection operators, and the projection-induced Riemannian gradient, provide the basis for both the design and convergence analysis of Riemannian algorithms for \eqref{J^B}, from first-order gradient methods to second-order approaches.

\noindent To substantiate the utility of these foundational derivations, we select two canonical algorithms as illustrative examples. In this section, we focus on the Riemannian line-search method as a showcase to demonstrate algorithmic design grounded in the theoretical foundations derived in Sections \ref{sec:geometry of MB}--\ref{sec:properties}. In the following subsections, we present the detailed algorithm and provide a rigorous convergence analysis. Subsequently, in Section \ref{sec:algorithmTR}, we extend to the Riemannian trust-region method, further exemplifying how these foundational tools support more complex, second-order optimization strategies.

\subsection{Algorithm Design}
\noindent We propose a Riemannian line-search framework on the manifold \(\mathcal{M}^B\). Starting from an initial point \(\bm{x}^0 \in \mathcal{M}^B\), the method generates a sequence \(\{\bm{x}^k\}\) by moving along retraction curves induced by \(\bm{R}_{\bm{x}}^B\). At each iteration, the algorithm identifies a descent direction in the tangent space and computes a step size that satisfies the Armijo sufficient decrease condition.
\begin{algorithm}[H]
	\caption{A Riemannian Line-Search Method for (\ref{J^B})}
	\label{alg:line-search}
	\begin{algorithmic}[1]
		\State Given $\bm{x}^0\in \mathcal{M}^B$, initial step size $\bar{\alpha}>0$, backtracking factor $\rho\in (0, 1)$, and Armijo parameter $c\in (0, 1)$.
		\While {not converged}
		\State Select a unit-norm vector $\bm{d}^k \in \T_{\bm{x}^k}\mathcal{M}^B$ such that $-\bm{d}^k$ is a descent direction;
        \State Initialize step size $\bar{\alpha}_k = \min\{\bar{\alpha},\Delta_{RL}\}$;
        \State \textbf{Backtracking line-search:} Find the smallest $j \ge 0$ such that $\alpha_k = \bar{\alpha}_k\rho^j$ satisfies the Armijo condition:
		\begin{equation}\label{eqn:armijo}
			J^B\left(\bm{R}^B_{\bm{x}^k}\left(-\alpha_k \bm{d}^k\right)\right) \le J^B(\bm{x}^k) - c\alpha_k ( \grad J^B(\bm{x}^k), \bm{d}^k )_{\bm{x}^k}.
		\end{equation}
		\State Update iterate via retraction:
		\begin{equation}\label{eqn:update_ls}
			\bm{x}^{k+1} = \bm{R}^B_{\bm{x}^k}(-\alpha_k \bm{d}^k).
		\end{equation}
		\State Update the iteration index: \( k = k + 1 \);
		\EndWhile
	\end{algorithmic}
\end{algorithm}

\vspace{-0.15cm}
\noindent For the practical realization of Algorithm \ref{alg:line-search}, we now specify the search direction, the step-size rule, the stopping criterion, and the evaluation of the retraction.

\medskip
\noindent\textbf{Search direction.}
A typical choice for $\bm{d}^k$ is \(\grad J^B(\bm{x}^k)/\norm{\grad J^B(\bm{x}^k)}_{\bm{x}^k}\), where the evaluation of \(\grad J^B(\bm{x}^k)\) is expensive. By Theorem \ref{thm: taylor}, \(\grad J^B(\bm{x}^k)\) is characterized through
\begin{equation}\label{eqn:gradJBxk}
    \bigl(\grad J^B(\bm{x}^k), \bm{v}\bigr)_{\bm{x}^k}
    =
    \bigl(\grad J(\bm{x}^k), \bm{v}\bigr)_{\mathcal{E}},
    \qquad
    \forall \bm{v} \in \T_{\bm{x}^k}\mathcal{M}^B.
\end{equation}

\noindent Explicitly recovering \(\grad J^B(\bm{x}^k)\) from \eqref{eqn:gradJBxk} requires solving a tangent-space Riesz representation problem, which is computationally costly. To avoid this, we employ the projection-induced gradient $\bm{G}_{\bm{x}^k}=\bm{P}_{\bm{x}^k}^B(\bm{P}_{\bm{x}^k}^B)^*\grad J(\bm{x}^k)$, defined in \eqref{eqn:projection induced gradient def}. By Proposition~\ref{prop:projection induced metric gradient}, $\bm{G}_{\bm{x}^k}$ is the exact Riemannian gradient with respect to the projection-induced Riemannian metric and is uniformly aligned with $\grad J^B(\bm{x}^k)$. We therefore define the normalized search direction by
\begin{equation}\label{eqn:line-search direction imp}
 \bm{d}^k
 =
 \frac{\bm{G}_{\bm{x}^k}}{\|\bm{G}_{\bm{x}^k}\|_{\bm{x}^k}}.
\end{equation}
As justified later in Remark~\ref{rem:choose of dk}, this choice ensures that \(-\bm{d}^k\) is a valid descent direction whenever \(\grad J^B(\bm{x}^k)\neq \bm{0}_{\mathcal E}\).

\medskip
\noindent\textbf{Step-size selection.} To determine a step size that satisfies the Armijo condition \eqref{eqn:armijo}, we employ the backtracking line-search procedure outlined in Algorithm \ref{alg:backtracking}.

\begin{algorithm}[H]
	\caption{A Backtracking Line-Search Strategy on $\mathcal{M}^B$}
	\label{alg:backtracking}
	\begin{algorithmic}[1]
		\Require Current iterate $\bm{x}^k \in \mathcal{M}^B$, descent direction $\bm{d}^k \in \T_{\bm{x}^k}\mathcal{M}^B$, Riemannian gradient $\grad J^B(\bm{x}^k)$, initial step $\bar{\alpha}_k= \min\{\bar{\alpha},\Delta_{RL}\} > 0$, contraction factor $\rho \in (0, 1)$, and sufficient decrease parameter $c \in (0, 1)$.
		
		\State Set \(\alpha\leftarrow \bar\alpha_k\).
		
		\State \textbf{Repeat} until the Armijo condition is satisfied:
		\While{ $J^B\left(\bm{R}^B_{\bm{x}^k}\left(-\alpha \bm{d}^k\right)\right) > J^B(\bm{x}^k) - c\alpha ( \grad J^B(\bm{x}^k), \bm{d}^k )_{\bm{x}^k}$}
			\State Reduce step size:
			\begin{equation}
				\alpha \leftarrow \rho \alpha
			\end{equation}
		\EndWhile
		
		\State \Return step size $\alpha_k = \alpha$.
	\end{algorithmic}
\end{algorithm}

\noindent Given the direction \(\bm{d}^k\), the step size \(\alpha_k\) is determined by the Armijo backtracking procedure in Algorithm \ref{alg:line-search}. Since \(\bm{d}^k \in \T_{\bm{x}^k}\mathcal{M}^B\), the defining relation of the Riemannian gradient implies
\begin{equation*}
    \bigl(\grad J^B(\bm{x}^k), \bm{d}^k\bigr)_{\bm{x}^k}
    =
    \bigl(\grad J(\bm{x}^k), \bm{d}^k\bigr)_{\mathcal{E}}.
\end{equation*}
Therefore, in implementation, the Armijo condition \eqref{eqn:armijo} can be checked without explicitly computing \(\grad J^B(\bm{x}^k)\). In the next section, Lemma \ref{lem:armijo sufficient decrease} shows that this backtracking procedure always produces a positive step size and that the iterates \(\{\bm{x}^k\}_{k\ge 0}\) remain in the bounded set \(X\) defined in \eqref{eqn:level set line-search}.

\medskip
\noindent\textbf{Termination criterion.} A natural stopping condition is $\norm{\grad J^B(\bm{x}^k)}_{\bm{x}^k}<\mathrm{tol}$. However, evaluating \(\grad J^B(\bm{x}^k)\) requires solving a tangent-space Riesz representation problem and is therefore computationally expensive. To avoid this additional cost, we instead employ the criterion
\[
\norm{\bm{G}_{\bm{x}^k}}_{\widetilde{g}_{\bm{x}^k}}=\norm{(\bm{P}_{\bm{x}^k}^B)^*\grad J(\bm{x}^k)}_{\mathcal E}<\mathrm{tol}.
\]
The validity of this substitution follows from Corollary~\ref{cor:uniform norm equivalences}, which establishes the norm equivalence
\[
\norm{\grad J^B(\bm{x}^k)}_{\bm{x}^k}\le\norm{(\bm{P}_{\bm{x}^k}^B)^*\grad J(\bm{x}^k)}_{\mathcal E}\le K_B\norm{\grad J^B(\bm{x}^k)}_{\bm{x}^k}.
\]
Hence, the above stopping criterion provides a computationally efficient equivalent form of the stationarity measure \(\norm{\grad J^B(\bm{x}^k)}_{\bm{x}^k}\), with equivalence constants controlled uniformly by \(K_B\). In particular, it vanishes if and only if the first-order stationarity condition \(\grad J^B(\bm{x}^k)=\bm 0_{\mathcal E}\) holds.

\medskip
\noindent\textbf{Evaluation of the retraction.}
To compute the update \(\bm{R}^B_{\bm{x}^k}(-\alpha_k \bm{d}^k)\), we must solve the nonlinear equation
\begin{equation}\label{retraction newton step line-search}
	B\bm{R}_{\bm{x}^k}\bigl(\bm{p}^k-\bm{P}_{\bm{x}^k}B^* \delta \bigr)=c,
\end{equation}
where \(\bm{p}^k := -\alpha_k \bm{d}^k\). We solve \eqref{retraction newton step line-search} by Newton's method. Using \eqref{eqn:Fx}, the derivative with respect to \(\delta\) is given by
\begin{equation*}
	\frac{\partial F_{\bm{x}^k}}{\partial \delta}\bigg|_{(\bm{p}^k, \delta)} 
	=
	-B \circ D\bm{R}_{\bm{x}^k}\bigl(\bm{p}^k-\bm{P}_{\bm{x}^k}B^*\delta\bigr)
	\circ \bm{P}_{\bm{x}^k}B^*.
\end{equation*}
This provides the practical realization of the retraction step in Algorithm \ref{alg:line-search}.


\subsection{Convergence Analysis for Algorithm \ref{alg:line-search}}
Let $\bm{x}^0\in \mathcal{M}^B$ denote the initial point. We define the sublevel set $X$ as
\begin{equation}\label{eqn:level set line-search}
    X := \left\{\bm{x} \in \mathcal{M}^B \mid J^B(\bm{x}) \le J^B(\bm{x}^0)\right\}\subset \mathcal{M}^B\subset\mathcal{M}.
\end{equation}
By the coercivity of \(J\), $X$ is bounded. Since the estimates in Section~\ref{sec:properties} were established for an arbitrary bounded set, we henceforth specialize the set \(X\) in \eqref{eqn:X} to the one defined in \eqref{eqn:level set line-search}, and keep the same notation for the corresponding constants. Moreover, \(J^B\) is bounded from below on \(\mathcal{M}^B\). Building on the properties established in Section \ref{sec:properties}, we proceed with the following convergence analysis.

\begin{assumption}\label{asmp:dk DJBxk angle}
    We assume that the sequence of descent directions $\{\bm{d}^k\}_{k\ge 0}$ generated by Algorithm \ref{alg:line-search} satisfies the following condition: there exists a constant $K_a>0$ such that
    \begin{equation}\label{eqn:uniform lower bound for angle dk DJBxk}
         \left(\bm{d}^k,\grad J^B(\bm{x}^k)\right)_{\bm{x}^k}\ge K_a\norm{\bm{d}^k}_{\bm{x}^k}\norm{\grad J^B(\bm{x}^k)}_{\bm{x}^k}=K_a\norm{\grad J^B(\bm{x}^k)}_{\bm{x}^k}.
    \end{equation}
\end{assumption}

\noindent It is important to verify that Assumption \ref{asmp:dk DJBxk angle} is practical and compatible with our proposed computational strategies in the previous section. The following remark confirms that our specific choice of the normalized projection-induced Riemannian gradient \eqref{eqn:line-search direction imp} satisfies Assumption \ref{asmp:dk DJBxk angle}.

\begin{remark}\label{rem:choose of dk}
    By Corollary~\ref{cor:uniform norm equivalences}, whenever
    \(\|\grad J^B(\bm{x}^k)\|_{\bm{x}^k}>0\), the projection-induced gradient
    \(
        \bm{G}_{\bm{x}^k}
        =
        \bm{P}_{\bm{x}^k}^B
        (\bm{P}_{\bm{x}^k}^B)^*
        \grad J(\bm{x}^k)
    \)
    is nonzero and satisfies
    \[
    \left(
        \bm{G}_{\bm{x}^k},
        \grad J^B(\bm{x}^k)
    \right)_{\mathcal E}
    \ge
    \frac{1}{K_B^2}
    \|\bm{G}_{\bm{x}^k}\|_{\mathcal E}
    \|\grad J^B(\bm{x}^k)\|_{\mathcal E}.
    \]
    Hence the normalized direction \(\bm{d}^k=\bm{G}_{\bm{x}^k}/\|\bm{G}_{\bm{x}^k}\|_{\bm{x}^k}\) satisfies Assumption~\ref{asmp:dk DJBxk angle} with \(K_a=1/K_B^2\).
\end{remark}

\noindent In the remainder of this subsection, we work under Assumption~\ref{asmp:dk DJBxk angle}. Having established the validity of the search direction, we next consider the step-size selection. The following lemma shows that the backtracking line search terminates with a step size satisfying a sufficient decrease condition.

\begin{lemma}\label{lem:armijo sufficient decrease}
    At iteration $k$, the backtracking line-search procedure (Algorithm \ref{alg:backtracking}) with parameters $\rho, c \in (0,1)$ yields a step size $\alpha_k > 0$ satisfying
    \begin{multline}\label{eqn:armijo sufficient decrease}
        J^B(\bm{x}^k)-J^B\left(\bm{R}^B_{\bm{x}^k}\left(-\alpha_k \bm{d}^k\right)\right) \ge\\
        cK_a\min\left\{\bar{\alpha}_k,\ \rho\Delta_{RL},\ \frac{2\rho(1-c)}{\beta_{RL}}K_a\norm{\grad J^B(\bm{x}^k)}_{\bm{x}^k}\right\}\norm{\grad J^B(\bm{x}^k)}_{\bm{x}^k}.
    \end{multline}
    Furthermore, given the update $\bm{x}^{k+1} = \bm{R}^B_{\bm{x}^k}(-\alpha_k \bm{d}^k)$, the monotonicity condition $J^B(\bm{x}^k)\ge J^B(\bm{x}^{k+1})$ holds, implying that the entire sequence $\{\bm{x}^k\}_{k\ge 0}$ generated by Algorithm \ref{alg:line-search} remains within the set $X$.
\end{lemma}

\begin{proof}
    Since $\bm{x}^0 \in X$, we proceed by induction, assuming that $\bm{x}^k \in X$. Invoking Proposition \ref{prop:LC1 second order property}, for any step size $\alpha \in [0, \Delta_{RL}]$, and observing that $\norm{-\alpha \bm{d}^k}_{\bm{x}^k} = \alpha \le \Delta_{RL}$, we obtain the following estimate:
    \begin{equation}\label{eqn:armijo second order decrease}
        J^B(\bm{x}^k)-J^B(\bm{R}^B_{\bm{x}^k}(-\alpha\bm{d}^k))\ge \alpha\left(\grad J^B(\bm{x}^k),\bm{d}^k\right)_{\bm{x}^k}-\frac{\beta_{RL}}{2}\alpha^2.
    \end{equation}
    Conversely, suppose the algorithm rejects a candidate step size $\alpha$ (i.e., the Armijo condition is violated). Then, we have
    \begin{equation}\label{eqn:armijo fail}
        J^B(\bm{x}^k)-J^B(\bm{R}^B_{\bm{x}^k}(-\alpha\bm{d}^k)) < c\alpha\left( \grad J^B(\bm{x}^k), \bm{d}^k \right)_{\bm{x}^k}.
    \end{equation}
    Combining \eqref{eqn:armijo second order decrease} and \eqref{eqn:armijo fail} implies that any rejected step size $\alpha$ must satisfy
    \begin{equation}\label{eqn:armijo fail step lower bound}
        \alpha > \min\left\{\Delta_{RL},\ \frac{2(1-c)}{\beta_{RL}}\left( \grad J^B(\bm{x}^k), \bm{d}^k \right)_{\bm{x}^k}\right\}.
    \end{equation}
    It follows from \eqref{eqn:armijo second order decrease} that the Armijo sufficient decrease condition
    \begin{equation}\label{eqn:armijo decrease}
        J^B(\bm{x}^k)-J^B(\bm{R}^B_{\bm{x}^k}(-\alpha\bm{d}^k)) \ge c\alpha\left( \grad J^B(\bm{x}^k), \bm{d}^k \right)_{\bm{x}^k}
    \end{equation}
    is necessarily satisfied whenever
    \begin{equation*}
        0 \le \alpha \le \min\left\{\Delta_{RL},\ \frac{2(1-c)}{\beta_{RL}}\left( \grad J^B(\bm{x}^k), \bm{d}^k \right)_{\bm{x}^k}\right\}.
    \end{equation*}
    Consequently, if a candidate $\alpha$ falls below the bound derived in \eqref{eqn:armijo fail step lower bound}, Algorithm \ref{alg:backtracking} must terminate (or it would have terminated at a larger step size). Taking into account the backtracking reduction factor $\rho$, the accepted step size $\alpha_k$ is bounded from below by
    \begin{equation}\label{eqn:armijo alphak lower bound}
        \alpha_k \ge \min\left\{\bar{\alpha}_k,\ \rho\Delta_{RL},\ \frac{2\rho(1-c)}{\beta_{RL}}\left( \grad J^B(\bm{x}^k), \bm{d}^k \right)_{\bm{x}^k}\right\}.
    \end{equation}
    Since $\alpha_k$ satisfies the Armijo condition, we combine Assumption \ref{asmp:dk DJBxk angle} (specifically \eqref{eqn:uniform lower bound for angle dk DJBxk}) with \eqref{eqn:armijo alphak lower bound} to deduce:
    \begin{equation*}
        \begin{aligned}
            &J^B(\bm{x}^k)-J^B(\bm{R}^B_{\bm{x}^k}(-\alpha_k\bm{d}^k)) \\
            \ge& c\alpha_k( \grad J^B(\bm{x}^k), \bm{d}^k )_{\bm{x}^k}\\
            \ge& c\min\left\{\bar{\alpha}_k,\ \rho\Delta_{RL},\ \frac{2\rho(1-c)}{\beta_{RL}}\left( \grad J^B(\bm{x}^k), \bm{d}^k \right)_{\bm{x}^k}\right\}\left( \grad J^B(\bm{x}^k), \bm{d}^k \right)_{\bm{x}^k}\\
            \ge& cK_a\min\left\{\bar{\alpha}_k,\ \rho\Delta_{RL},\ \frac{2\rho(1-c)}{\beta_{RL}}K_a\norm{\grad J^B(\bm{x}^k)}_{\bm{x}^k}\right\}\norm{\grad J^B(\bm{x}^k)}_{\bm{x}^k}.
        \end{aligned}
    \end{equation*}
    This completes the proof of \eqref{eqn:armijo sufficient decrease}. Finally, the update $\bm{x}^{k+1} = \bm{R}^B_{\bm{x}^k}(-\alpha_k \bm{d}^k)$ ensures $J^B(\bm{x}^k)\ge J^B(\bm{x}^{k+1})$, which inductively confirms that the sequence $\{\bm{x}^k\}_{k\ge 0}$ generated by Algorithm \ref{alg:line-search} remains in $X$.
\end{proof}

\noindent Having established the sufficient decrease property for a single iteration, we are now positioned to prove the convergence of Algorithm \ref{alg:line-search}. The main result establishes that the gradient norm of the iterates, $\norm{\grad J^B(\bm{x}^k)}_{\mathcal{E}}$, converges to zero.

\begin{theorem}
    Let $\{\bm{x}^k\}_{k\ge 0}$ denote the sequence of iterates generated by Algorithm \ref{alg:line-search}, utilizing the backtracking line-search described in Algorithm \ref{alg:backtracking} with fixed parameters $c, \rho \in (0,1)$ and with initial step sizes
$\bar{\alpha}_k=\min\{\bar\alpha,\Delta_{RL}\}$ for all $k\ge0$. Then, it holds that
    \begin{equation*}
        \lim_{k\to+\infty}\norm{\grad J^B(\bm{x}^k)}_{\mathcal{E}}=0.
    \end{equation*}
\end{theorem}

\begin{proof}
By Lemma~\ref{lem:armijo sufficient decrease} and the lower boundedness of
\(J^B\), the nonnegative series
\[
    \sum_{k=0}^{\infty}
    \min\left\{
        \bar{\alpha}_k,\rho\Delta_{RL},
        \frac{2\rho(1-c)}{\beta_{RL}}K_a
        \|\grad J^B(\bm{x}^k)\|_{\bm{x}^k}
    \right\}
    \|\grad J^B(\bm{x}^k)\|_{\bm{x}^k}
\]
is convergent, and hence its general term tends to zero. Since
\(\bar{\alpha}_k=\min\{\bar{\alpha},\Delta_{RL}\}>0\), the minimum factor is
uniformly bounded away from zero on every subsequence where
\(\|\grad J^B(\bm{x}^k)\|_{\bm{x}^k}\ge\epsilon>0\); therefore no such subsequence exists, which proves that
\(\|\grad J^B(\bm{x}^k)\|_{\bm{x}^k}\to0\). Since
\(\|\grad J^B(\bm{x}^k)\|_{\bm{x}^k}
=
\|\grad J^B(\bm{x}^k)\|_{\mathcal E}\), the asserted convergence follows.
\end{proof}

\section{Foundations of Algorithmic Design II: Riemannian Trust-Region Methods}\label{sec:algorithmTR}

\noindent We now focus on the Riemannian trust-region method to exemplify the algorithmic design grounded in the theoretical foundations established in Sections \ref{sec:geometry of MB}--\ref{sec:projection-induced gradient}. We present the detailed algorithm followed by a rigorous convergence analysis.

\subsection{Algorithm Design}\label{subsec:algorithmTRdesign}

\noindent At each iterate \(\bm{x}^k\), the trust-region method constructs, on the tangent space
\(\T_{\bm{x}^k}\mathcal M^B\), a local quadratic model $m_k:\T_{\bm{x}^k}\mathcal M^B\to\mathbb R$ defined by
\begin{equation}\label{eqn:qudratic model mk} 
m_k(\bm{p})
:= J^B_{\bm{x}^k}(\bm{0}_{\mathcal{E}})
+ \left(\grad J^B_{\bm{x}^k}(\bm{0}_{\mathcal{E}}),\bm{p}\right)_{\bm{x}^k}
+ \frac{1}{2}\left(\hess J^B_{\bm{x}^k}(\bm{0}_{\mathcal{E}})[\bm{p}],\bm{p}\right)_{\bm{x}^k}\qquad
\bm p\in\T_{\bm{x}^k}\mathcal M^B .
\end{equation}
The model \eqref{eqn:qudratic model mk} approximates $J^B_{\bm{x}^k}$ to second-order, with error $o(\|\bm{p}\|_{\bm{x}^k}^2)$ as $\bm{p}\to \bm{0}_{\mathcal{E}}$. Moreover, by the definition of $J^B_{\bm{x}^k}$ in \eqref{eqn:JxB} and the equation \eqref{eqn:first-order derivative of JB and JxB} in Theorem~\ref{thm: taylor}, \eqref{eqn:qudratic model mk} can be written equivalently as

$$
m_k(\bm{p})
= J^B_{\bm{x}^k}(\bm{0}_{\mathcal{E}})
+ \left(\grad J^B(\bm{x}^k),\bm{p}\right)_{\bm{x}^k}
+ \frac{1}{2}\left(\hess J^B_{\bm{x}^k}(\bm{0}_{\mathcal{E}})[\bm{p}],\bm{p}\right)_{\bm{x}^k}.
$$
On a nonterminated iteration, the step \( \Tp^k \) is then computed by solving the trust-region subproblem
\begin{equation}\label{TR}
	\begin{aligned}
		\min_{\Tp\in \T_{{\bm{x}}^k}\mathcal{M}^B}\quad &m_k(\bm{p})=J^B_{\bm{x}^k}(\bm{0}_{\mathcal{E}})+\left( \grad J^B(\bm{x}^k),\Tp\right)_{\bm{x}^k}+\frac{1}{2} \left(\hess J^B_{\bm{x}^k}(\bm{0}_{\mathcal{E}}) [\Tp],\Tp \right)_{\bm{x}^k}\\
		\text{s.t.}\quad&\|\Tp\|_{\bm{x}^k} \le \Delta^k.\\
	\end{aligned}
\end{equation}
The solution \( \Tp^k \) is retracted back to $\mathcal{M}^B$ to obtain a new candidate point, ensuring feasibility while maintaining the algorithm's convergence properties.

\noindent This results in the Riemannian trust-region method for solving problem \eqref{J^B}, as formalized in Algorithm \ref{alg:trust-region}.

\begin{algorithm}[H]
	\caption{A Riemannian Trust-Region Method for  (\ref{J^B})}
	\label{alg:trust-region}
	\begin{algorithmic}[1]
		\State Given $\bar{\Delta}\in(0,\Delta_{RL}]$, $\bm{x}^0\in \mathcal{M}^B$, $\Delta^0\in (0, \bar{\Delta}]$, and $\tilde{r}\in [0,\frac{1}{4}).$
		\While {not converged}
		\State Obtain $\Tp^k$ by approximately solving \eqref{TR};
		\State Compute tentative next iterate $\bm{x}_{+}^k=\bm{R}^B_{\bm{x}^k}(\Tp^k)$ by finding $\delta\in\mathcal{Y}$ such that 
		\begin{equation}\label{retraction newton step}
			B\bm{R}_{\bm{x}^k}\left({\bm{p}^k-\bm{P}_{\bm{x}^k}B^* \delta }\right)=c.
		\end{equation}
		\State Evaluate 
        \begin{equation}\label{eqn:rhok}
            \rho_k:=\frac{J^B_{\bm{x}^k}(\bm{0}_{\mathcal{E}})-J^B_{\bm{x}^k}(\bm{p}^k)}{m_k(\bm{0}_{\mathcal{E}})-m_k(\bm{p}^k)}=\frac{J^B(\bm{x}^k)-J^B(\bm{x}^k_+)}{m_k(\bm{0}_{\mathcal{E}})-m_k(\bm{p}^k)}.
        \end{equation}
		\If{$\rho_k<\frac{1}{4}$}
		\State $\Delta^{k+1}=\frac{1}{4}\Delta^k$;
		\Else
		\If{$\rho_k>\frac{3}{4}$ and $\|\Tp^k\|_{\bm{x}^k}=\Delta^k$}
		\State $\Delta^{k+1}=\min(2\Delta^k,\bar{\Delta} )$;
		\Else
		\State $\Delta^{k+1}=\Delta^{k}$;
		\EndIf
		\EndIf
		\If{$\rho_k>\tilde{r}$}
		\State $\bm{x}^{k+1}=\bm{x}_{+}^k$;
		\Else 
		\State $\bm{x}^{k+1}=\bm{x}^k$;
		\EndIf
		\State Update the iteration index: \( k = k + 1 \);
		\EndWhile
	\end{algorithmic}
\end{algorithm}

\medskip
\noindent\textbf{Termination criterion.} As in the line-search method, we use the same computable stopping criterion based on the projection-induced gradient norm
\begin{equation}\label{eqn:termination tr}
\eta_k
:=
\|\bm G_{\bm{x}^k}\|_{\widetilde g_{\bm{x}^k}}=
\|(\bm P_{\bm{x}^k}^B)^*\grad J(\bm{x}^k)\|_{\mathcal E}<\mathrm{tol}_1.
\end{equation}
By Corollary~\ref{cor:uniform norm equivalences}, $\eta_k$ is uniformly
equivalent, on bounded sets, to \(\|\grad J^B(\bm{x}^k)\|_{\bm{x}^k}\).

\medskip
\noindent\textbf{CG--Steihaug method for subproblem \eqref{TR}.} The trust-region subproblem \eqref{TR} is solved approximately by a CG--Steihaug method \cite{steihaug1983conjugate}. All inner products and norms in the iteration are taken with respect to \((\cdot,\cdot)_{\bm{x}^k}\), and the Hessian action is the one appearing in the quadratic model \eqref{TR}. The resulting method is summarized in Algorithm~\ref{alg:CG--Steihaug}.

\begin{algorithm}[H]
	\caption{The CG--Steihaug Method for \eqref{TR}}
	\label{alg:CG--Steihaug}
	\begin{algorithmic}[1]
			\State \textbf{Input:} Tolerance $\mathrm{tol}_1,\mathrm{tol}_2 > 0$, initial point $\bm{p}_0 = \bm{0}_{\mathcal{E}}$.
			\State Set $\bm{r}_0=\grad J^B(\bm{x}^k)$ and $\bm{d}_0 = -\bm{r}_0$.
            \If{$\|\bm{r}_0\|_{\bm{x}^k} < \mathrm{tol}_1$}
					\State Return $\bm{p} = \bm{p}_0$
				\EndIf
			\While {not converged}
					\If{$\left(\hess J^B_{\bm{x}^k}(\bm{0}_{\mathcal{E}})[\bm{d}_j], \bm{d}_j\right)_{\bm{x}^k} \le 0$}
						\State Find $\tau$ such that $\bm{p} = \bm{p}_j + \tau \bm{d}_j$ minimizes \eqref{TR} and satisfies $\|\bm{p}\|_{\bm{x}^k} = \Delta^k$
						\State Return $\bm{p}$
					\Else
						\State Compute $\alpha_j = \frac{(\bm{r}_j, \bm{r}_j)_{\bm{x}^k}}{\left(\hess J^B_{\bm{x}^k}(\bm{0}_{\mathcal{E}})[\bm{d}_j], \bm{d}_j\right)_{\bm{x}^k}}$
						\State Update $\bm{p}_{j+1} = \bm{p}_j + \alpha_j \bm{d}_j$
						\If{$\|\bm{p}_{j+1}\|_{\bm{x}^k} \ge \Delta^k$}
							\State Find $\tau \ge 0$ such that $\bm{p} = \bm{p}_j + \tau \bm{d}_j$ satisfies $\|\bm{p}\|_{\bm{x}^k} = \Delta^k$
							\State Return $\bm{p}$
						\EndIf
						\State Update $\bm{r}_{j+1}=\bm{r}_{j}+\alpha_j\,\hess J^B_{\bm{x}^k}(\bm{0}_{\mathcal{E}})[\bm{d}_j]$ 
						\If{$\|\bm{r}_{j+1}\|_{\bm{x}^k} < \mathrm{tol}_2 \cdot\|\bm{r}_0\|_{\bm{x}^k}$}
							\State Return $\bm{p} = \bm{p}_{j+1}$
						\EndIf
						\State Compute $\beta_{j+1} = \frac{(\bm{r}_{j+1},\bm{r}_{j+1})_{\bm{x}^k}}{(\bm{r}_j, \bm{r}_j)_{\bm{x}^k}}$
						\State Update $\bm{d}_{j+1} = -\bm{r}_{j+1} + \beta_{j+1} \bm{d}_j$
					\EndIf
					\State Update the iteration index: $ j = j + 1 $
			\EndWhile
		\end{algorithmic}
\end{algorithm}

\begin{remark}
    Algorithm~\ref{alg:CG--Steihaug} is formulated intrinsically on the tangent space \(\mathrm{T}_{\bm{x}^k}\mathcal{M}^B\). All residuals \(\bm{r}_j\) lie in \(\mathrm{T}_{\bm{x}^k}\mathcal{M}^B\). Moreover, the
    search directions are defined recursively by \(\bm{d}_{j+1}=-\bm{r}_{j+1}+\beta_{j+1}\bm{d}_j\), with \(\bm{d}_0 \in \mathrm{T}_{\bm{x}^k}\mathcal{M}^B\). Since \(\mathrm{T}_{\bm{x}^k}\mathcal{M}^B\) is a vector space, it is closed under linear combinations. Hence, by induction, all search directions \(\bm{d}_j\) remain in \(\mathrm{T}_{\bm{x}^k}\mathcal{M}^B\). Consequently, the iterates \(\bm{p}_j\) and the returned vector \(\bm{p}\), which are constructed along these directions, also belong to \(\mathrm{T}_{\bm{x}^k}\mathcal{M}^B\).

 \noindent  Compared with standard conjugate gradient methods
    \cite{nocedal2006numerical}, the CG--Steihaug method incorporates the trust-region constraint through the initial stopping check and two early termination mechanisms. The
    iteration stops if the current search direction \(\bm{d}_j\) has non-positive curvature with respect to \(\hess J^B_{\bm{x}^k}(\bm{0}_{\mathcal{E}})\), or if the next iterate \(\bm{p}_{j+1}\) would leave the trust region. In either case, the final step is obtained by intersecting the current search path with the trust-region boundary.
\end{remark}

\begin{remark}\label{rem:cauchy decrease of tcg}
    It is established that the Cauchy point, as defined in
    \cite[(7.8) on p.~142]{absil2008optimization}, fulfills the Cauchy decrease condition
    \begin{equation}\label{Cauchy decrease inequality}
        {m}_{k}(\bm{0}_{\mathcal{E}})-{m}_{k}(\bm{p}^k)\ge c_1\norm{\grad J^B(\bm{x}^k)}_{\mE}\min\left(\Delta^k,\frac{\norm{\grad J^B(\bm{x}^k)}_{\mathcal{E}}}{\norm{\hess J^B_{\bm{x}^k}(\bm{0}_{\mathcal{E}})}}\right),
    \end{equation}
    with \(c_1=\frac{1}{2}\). Provided the initial residual test is not triggered, the tangent vector \(\bm{p}^k\) generated by the CG--Steihaug method satisfies this inequality. This follows from the fact that the CG--Steihaug method first reaches the Cauchy point and then proceeds, when possible, to further improve the model decrease \cite[Proposition 7.3.2]{absil2008optimization}.
\end{remark}

\begin{remark}
    Although Algorithm~\ref{alg:CG--Steihaug} has the same algebraic structure as the classical CG--Steihaug method, its implementation is not straightforward in the present setting. Although the projection \(\bm P_{\bm{x}^k}^B\) onto \(\T_{\bm{x}^k}\mathcal M^B\) is available, it is generally not orthogonal with respect to \((\cdot,\cdot)_{\bm{x}^k}\). Hence, it does not directly provide the Riesz representatives of the tangent residuals $\bm{r}_j$ appearing in the intrinsic CG iteration. These residuals therefore have to be computed through weak characterizations. More precisely, by applying \eqref{eqn:first-order derivative of JB and JxB} in Theorem~\ref{thm: taylor}, the initialization \(\bm r_0=\grad J^B(\bm{x}^k)\) amounts to finding \(\bm r_0\in\T_{\bm{x}^k}\mathcal M^B\) such that 
    \begin{equation}\label{eqn:linear system for gradient} 
    \left(\bm r_0,\bm P_{\bm{x}^k}^B\bm v\right)_{\mathcal E} = \left(\grad J(\bm{x}^k),\bm P_{\bm{x}^k}^B\bm v\right)_{\mathcal E}, \qquad \forall\,\bm v\in\mathcal E . 
    \end{equation} 
    Similarly, using \eqref{eqn:second-order directional derivative of JxB} in Theorem~\ref{thm: taylor}, the residual update \[ \bm r_{j+1} = \bm r_j + \alpha_j \hess J_{\bm{x}^k}^B(\bm 0_{\mathcal E})[\bm d_j] \] requires finding \(\bm r_{j+1}\in\T_{\bm{x}^k}\mathcal M^B\) satisfying
    \begin{equation}\label{eqn:linear system for hessian} 
    \left(\bm r_{j+1},\bm P_{\bm{x}^k}^B\bm v\right)_{\mathcal E} = \left(\bm r_j,\bm P_{\bm{x}^k}^B\bm v\right)_{\mathcal E} + \alpha_j \left( \left(\hess J(\bm{x}^k)+\bm\pi_{\bm{x}^k}\right)[\bm d_j], \bm P_{\bm{x}^k}^B\bm v \right)_{\mathcal E}, \qquad \forall\,\bm v\in\mathcal E . 
    \end{equation} 
    Thus, the realization of Algorithm~\ref{alg:CG--Steihaug} requires solving a tangent-space Riesz representation problem at initialization and after every Hessian action. This repeated solution may become a computational bottleneck, especially when the inner CG iteration takes many steps. This motivates the projection-induced strategy introduced below, which avoids these residual solves while retaining the essential CG--Steihaug structure.
\end{remark}

\medskip
\noindent\textbf{A projection-induced realization of Algorithm~\ref{alg:CG--Steihaug}.}

\noindent As discussed above, the realization of Algorithm~\ref{alg:CG--Steihaug} requires a tangent-space Riesz representation solve at every inner iteration. We therefore develop a projection-induced CG--Steihaug method for the same quadratic model $m_k$, in which the Riesz representatives are available in closed form through the projection-induced metric \(\widetilde g_{\bm{x}^k}\) of Proposition~\ref{prop:projection induced metric gradient}. Since the inner iteration is carried out in a different metric, the trial steps and search directions generated below need not coincide with those of Algorithm~\ref{alg:CG--Steihaug}; only the quadratic model and the trust-region constraint are retained.

\noindent For a fixed outer iterate \(\bm{x}^k\), set
\begin{equation}\label{eqn:H^k}
    H^k:=\hess J(\bm{x}^k)+\bm\pi_{\bm{x}^k}.
\end{equation}
By Theorem~\ref{thm: taylor}, the quadratic part of the model $m_k$ is \(\left(H^k\bm p,\bm q\right)_{\mathcal E}\) for \(\bm p,\bm q\in\T_{\bm{x}^k}\mathcal M^B\). Recall from Proposition~\ref{prop:projection induced metric gradient} that \(A_{\bm{x}^k} = \bm P_{\bm{x}^k}^B \left(\bm P_{\bm{x}^k}^B\right)^* \big|_{\T_{\bm{x}^k}\mathcal M^B}\) and \(K_{\bm{x}^k}=A_{\bm{x}^k}^{-1}\) are self-adjoint and positive definite on \(\T_{\bm{x}^k}\mathcal M^B\), and that the projection-induced metric of \eqref{eqn:projection induced metric} is \(\widetilde g_{\bm{x}^k}(\bm u,\bm v)=\left(K_{\bm{x}^k}\bm u,\bm v\right)_{\mathcal E}\). The entire construction rests on the following identity: for every \(\bm w\in\mathcal E\) and every \(\bm z\in\T_{\bm{x}^k}\mathcal M^B\),
\begin{equation}\label{eqn:riesz identity}
\widetilde g_{\bm{x}^k}\!\left(\bm P_{\bm{x}^k}^B(\bm P_{\bm{x}^k}^B)^*\bm w,\bm z\right)
=\left(\bm P_{\bm{x}^k}^B(\bm P_{\bm{x}^k}^B)^*\bm w,K_{\bm{x}^k}\bm z\right)_{\mathcal E}
=\left(\bm w,A_{\bm{x}^k}K_{\bm{x}^k}\bm z\right)_{\mathcal E}
=\left(\bm w,\bm z\right)_{\mathcal E},
\end{equation}
where we used, in turn, the self-adjointness of \(K_{\bm{x}^k}\) on \(\T_{\bm{x}^k}\mathcal M^B\), the self-adjointness of \(\bm P_{\bm{x}^k}^B(\bm P_{\bm{x}^k}^B)^*\) on \(\mathcal E\) together with \(K_{\bm{x}^k}\bm z\in\T_{\bm{x}^k}\mathcal M^B\), and \(A_{\bm{x}^k}K_{\bm{x}^k}=\mathrm{Id}\) on \(\T_{\bm{x}^k}\mathcal M^B\). In words, \(\bm P_{\bm{x}^k}^B(\bm P_{\bm{x}^k}^B)^*\bm w\) is the Riesz representative, with respect to \(\widetilde g_{\bm{x}^k}\), of the functional \(\bm z\mapsto(\bm w,\bm z)_{\mathcal E}\) on \(\T_{\bm{x}^k}\mathcal M^B\), and it is obtained by a single application of \(\bm P_{\bm{x}^k}^B(\bm P_{\bm{x}^k}^B)^*\) rather than by a linear solve. Applying \eqref{eqn:riesz identity} with \(\bm w=\grad J(\bm{x}^k)\) and \(\bm w=H^k\bm p\) gives, for \(\bm p\in\T_{\bm{x}^k}\mathcal M^B\), we have
\begin{equation}\label{eqn:model ambient representation}
\begin{aligned}
    m_k(\bm p) &= J^B_{\bm{x}^k}(\bm{0}_{\mathcal{E}}) + \left(\grad J(\bm{x}^k),\bm p\right)_{\mathcal E} + \frac12\left(H^k\bm p,\bm p\right)_{\mathcal E}\\
    &= J^B_{\bm{x}^k}(\bm{0}_{\mathcal{E}})+\widetilde g_{\bm{x}^k}\!\left(\bm P_{\bm{x}^k}^B(\bm P_{\bm{x}^k}^B)^*\grad J(\bm{x}^k),\bm p\right)+\frac12\,\widetilde g_{\bm{x}^k}\!\left(\bm P_{\bm{x}^k}^B(\bm P_{\bm{x}^k}^B)^*H^k\bm p,\bm p\right).
\end{aligned}
\end{equation}

\noindent We now derive a projection-induced CG--Steihaug method. The CG recursion is associated with the gradient and Hessian of \(m_k\) with respect to the Hilbert-space metric \(\widetilde g_{\bm{x}^k}\), while the original trust-region constraint in \eqref{TR} is retained. Throughout the remainder of this subsection, \(\bm p_j\) and \(\bm d_j\) denote its trial steps and search directions, rather than those of Algorithm~\ref{alg:CG--Steihaug}. By
\eqref{eqn:model ambient representation}, the gradient and Hessian of \(m_k\) with
respect to \(\widetilde g_{\bm{x}^k}\) are respectively given by
\[
\bm G_{\bm{x}^k}=\bm P_{\bm{x}^k}^B\left(\bm P_{\bm{x}^k}^B\right)^*\grad J(\bm{x}^k)
\quad\hbox{and} \quad
\mathcal H_k:=\bm P_{\bm{x}^k}^B\left(\bm P_{\bm{x}^k}^B\right)^*H^k\big|_{\T_{\bm{x}^k}\mathcal M^B}.
\]
Since \(H^k\) is self-adjoint with respect to \(\left(\cdot,\cdot\right)_{\mathcal E}\),
the identity \eqref{eqn:riesz identity} gives
\[
\widetilde g_{\bm{x}^k}(\mathcal H_k\bm u,\bm v)=(H^k\bm u,\bm v)_{\mathcal E}=(\bm u,H^k\bm v)_{\mathcal E}
=\widetilde g_{\bm{x}^k}(\bm u,\mathcal H_k\bm v), \quad \mbox{ for } \bm u,\bm v\in\T_{\bm{x}^k}\mathcal M^B,
\] 
and so \(\mathcal H_k\) is self-adjoint on \((\T_{\bm{x}^k}\mathcal M^B,\widetilde g_{\bm{x}^k})\). The CG--Steihaug recursion in this space, initialized with \(\bm p_0=\bm 0_{\mathcal E}\), \(\bm y_0=\bm G_{\bm{x}^k}\), and \(\bm d_0=-\bm y_0\), reads
\[
\bm p_{j+1}=\bm p_j+\alpha_j\bm d_j,
\qquad
\bm y_{j+1}=\bm y_j+\alpha_j\,\bm P_{\bm{x}^k}^B\left(\bm P_{\bm{x}^k}^B\right)^*H^k\bm d_j,
\qquad
\bm d_{j+1}=-\bm y_{j+1}+\beta_{j+1}\bm d_j,
\]
with the usual coefficients
\begin{equation}\label{eqn:alpha beta metric form}
\alpha_j
=\frac{\|\bm y_j\|_{\widetilde g_{\bm{x}^k}}^2}
{\widetilde g_{\bm{x}^k}\!\left(\bm P_{\bm{x}^k}^B(\bm P_{\bm{x}^k}^B)^*H^k\bm d_j,\bm d_j\right)},
\qquad
\beta_{j+1}
=\frac{\|\bm y_{j+1}\|_{\widetilde g_{\bm{x}^k}}^2}{\|\bm y_j\|_{\widetilde g_{\bm{x}^k}}^2}.
\end{equation}
Here, \(\bm y_j\) is the residual of the iteration: an induction on the recursion gives
\(\bm y_j=\bm G_{\bm{x}^k}+\mathcal H_k\bm p_j\), and therefore, by
\eqref{eqn:riesz identity}, we have
\[
\widetilde g_{\bm{x}^k}(\bm y_j,\bm z)
=\left(\grad J(\bm{x}^k)+H^k\bm p_j,\bm z\right)_{\mathcal E}
=Dm_k(\bm p_j)[\bm z]
\qquad\text{for all }\bm z\in\T_{\bm{x}^k}\mathcal M^B.
\]
That is, \(\bm y_j\) is the Riesz representative of \(Dm_k(\bm p_j)\) with respect to \(\widetilde g_{\bm{x}^k}\), and it is updated by a single application of \(\bm P_{\bm{x}^k}^B(\bm P_{\bm{x}^k}^B)^*\) per iteration. Since \(\mathcal H_k\) is self-adjoint in this metric, the iteration inherits the usual conjugacy and descent properties of CG. Equivalently, it can be interpreted as a preconditioned CG realization with inverse preconditioner \(A_{\bm{x}^k}\), or, with preconditioner \(K_{\bm{x}^k}=A_{\bm{x}^k}^{-1}\).

\noindent It remains to evaluate the two inner products in \eqref{eqn:alpha beta metric form} without applying \(K_{\bm{x}^k}=A_{\bm{x}^k}^{-1}\), which would amount to a tangent-space linear solve. For the curvature term, identity \eqref{eqn:riesz identity} with \(\bm w=H^k\bm d_j\) yields
\begin{equation}\label{eqn:projection induced curvature}
\widetilde g_{\bm{x}^k}\!\left(\bm P_{\bm{x}^k}^B(\bm P_{\bm{x}^k}^B)^*H^k\bm d_j,\bm d_j\right)
=(H^k\bm d_j,\bm d_j)_{\mathcal E},
\end{equation}
so the curvature test coincides with its ambient counterpart. To compute the residual norm $\|\bm y_j\|_{\widetilde g_{\bm{x}^k}}$, we introduce
\[
\bm n_0:=\grad J(\bm{x}^k),
\qquad
\bm n_{j+1}:=\bm n_j+\alpha_jH^k\bm d_j.
\]
An induction gives 
\begin{equation}\label{eqn:nj direct}
    \bm n_j=\grad J(\bm{x}^k)+H^k\bm p_j,
\end{equation}
i.e., \(\bm n_j\) is the ambient representation of the model derivative, \(\left(\bm n_j,\bm z\right)_{\mathcal E}=Dm_k(\bm p_j)[\bm z]\) for \(\bm z\in\T_{\bm{x}^k}\mathcal M^B\), and consequently
\[
\bm y_j
=\bm G_{\bm{x}^k}+\bm P_{\bm{x}^k}^B\left(\bm P_{\bm{x}^k}^B\right)^*H^k\bm p_j
=\bm P_{\bm{x}^k}^B\left(\bm P_{\bm{x}^k}^B\right)^*\bm n_j .
\]
Applying \eqref{eqn:riesz identity} with \(\bm w=\bm n_j\) and \(\bm z=\bm y_j\) then
gives
\begin{equation}\label{eqn:projection induced residual norm}
    \|\bm y_j\|_{\widetilde g_{\bm{x}^k}}^2
    =\widetilde g_{\bm{x}^k}\!\left(\bm P_{\bm{x}^k}^B(\bm P_{\bm{x}^k}^B)^*\bm n_j,\bm y_j\right)
    =\left(\bm n_j,\bm y_j\right)_{\mathcal E}.
\end{equation}
Substituting \eqref{eqn:projection induced curvature} and \eqref{eqn:projection induced residual norm} into \eqref{eqn:alpha beta metric form}, the
coefficients take the computable form
\[
\alpha_j=\frac{(\bm n_j,\bm y_j)_{\mathcal E}}{(H^k\bm d_j,\bm d_j)_{\mathcal E}},
\qquad
\beta_{j+1}=\frac{(\bm n_{j+1},\bm y_{j+1})_{\mathcal E}}{(\bm n_j,\bm y_j)_{\mathcal E}},
\]
and, by \eqref{eqn:projection induced residual norm}, the quantity
\(\sqrt{(\bm n_j,\bm y_j)_{\mathcal E}}\) used in the stopping tests below is exactly the
residual norm \(\|\bm y_j\|_{\widetilde g_{\bm{x}^k}}\). Every step of the iteration thus
requires only applications of \(\bm P_{\bm{x}^k}^B\),
\(\left(\bm P_{\bm{x}^k}^B\right)^*\), and \(H^k\). These identities give the realization
summarized next.

\noindent For clarity, the following table compares the algebraic structures of the two realizations. The quantities in its two columns belong to two independently generated sequences; the comparison does not assert equality of their trial steps, residuals, or search directions at the same index.
\begin{center}
\small
\setlength{\tabcolsep}{3pt}
\renewcommand{\arraystretch}{2.2}
\begin{tabularx}{\textwidth}{
>{\centering\arraybackslash}p{0.16\textwidth}|
>{\centering\arraybackslash}X|
>{\centering\arraybackslash}X
}
\hline
Quantity
&
Ambient-metric realization
&
Projection-induced realization
\\
\hline

\makecell{Residual\\ representation}
&
\makecell[c]{
With \(\bm n_j^{\rm am}:=\grad J(\bm x^k)+H^k\bm p_j^{\rm am}\),\\[2pt]
solve for \(\bm r_j^{\rm am}\in\T_{\bm{x}^k}\mathcal M^B\):\\[2pt]
\(\left(\bm r_j^{\rm am},\bm z\right)_{\mathcal E}
=Dm_k(\bm p_j^{\rm am})[\bm z]
=\left(\bm n_j^{\rm am},\bm z\right)_{\mathcal E}\)
}
&
\makecell[c]{
Compute\\[2pt]
\(\begin{aligned}
\bm y_j
&=
\bm P_{\bm{x}^k}^B
\left(\bm P_{\bm{x}^k}^B\right)^*
\bm n_j
\end{aligned}\)
}
\\
\hline

Residual norm
&
\(\|\bm r_j^{\rm am}\|_{\mathcal E}^2\)
&
\makecell[c]{
\(\begin{aligned}
\|\bm y_j\|_{\widetilde g_{\bm{x}^k}}^2
&=
(\bm n_j,\bm y_j)_{\mathcal E}
\end{aligned}\)
}
\\
\hline

Initial residual
&
\(\bm r_0^{\rm am}=\grad J^B(\bm{x}^k)\)
&
\makecell[c]{
\(\begin{aligned}
\bm n_0&=\grad J(\bm{x}^k),\\
\bm y_0
&=
\bm P_{\bm{x}^k}^B
\left(\bm P_{\bm{x}^k}^B\right)^*
\bm n_0
\end{aligned}\)
}
\\
\hline

Search direction
&
\(\bm d_j^{\rm am}=-\bm r_j^{\rm am}+\beta_j^{\rm am}\bm d_{j-1}^{\rm am}\)
&
\(\bm d_j=-\bm y_j+\beta_j\bm d_{j-1}\)
\\
\hline

Step size
&
\makecell[c]{
\(\begin{aligned}
\alpha_j^{\rm am}
&=
\frac{(\bm r_j^{\rm am},\bm r_j^{\rm am})_{\mathcal E}}
{(H^k\bm d_j^{\rm am},\bm d_j^{\rm am})_{\mathcal E}}
\end{aligned}\)
}
&
\makecell[c]{
\(\begin{aligned}
\alpha_j
&=
\frac{(\bm n_j,\bm y_j)_{\mathcal E}}
{(H^k\bm d_j,\bm d_j)_{\mathcal E}}
\end{aligned}\)
}
\\
\hline

CG coefficient
&
\makecell[c]{
\(\begin{aligned}
\beta_{j+1}^{\rm am}
&=
\frac{\|\bm r_{j+1}^{\rm am}\|_{\mathcal E}^2}
{\|\bm r_j^{\rm am}\|_{\mathcal E}^2}
\end{aligned}\)
}
&
\makecell[c]{
\(\begin{aligned}
\beta_{j+1}
&=
\frac{
(\bm n_{j+1},\bm y_{j+1})_{\mathcal E}
}{
(\bm n_j,\bm y_j)_{\mathcal E}
}
\end{aligned}\)
}
\\
\hline
\end{tabularx}
\end{center}

\noindent The trust-region constraint is kept in the norm \(\norm{\cdot}_{\bm{x}^k}\) of \eqref{TR}, unchanged from Algorithm~\ref{alg:CG--Steihaug}. This choice is deliberate. First, the subproblem \eqref{TR}---including the model \(m_k\), the radius \(\Delta^k\), and the constraint norm---is unchanged from that used by the outer trust-region framework. Hence, the acceptance test and the radius-update rule remain unchanged, and the existing outer convergence analysis applies once the Cauchy-decrease condition established below is invoked. Second, measuring the constraint in \(\|\cdot\|_{\widetilde g_{\bm{x}^k}}\) would require applications of \(K_{\bm{x}^k}=A_{\bm{x}^k}^{-1}\), that is, tangent-space linear solves, which is precisely the cost the projection-induced realization is designed to avoid. Finally, by Proposition~\ref{prop:projection induced metric gradient} the two norms are uniformly equivalent on bounded sets, so this choice is compatible with the Cauchy-decrease analysis given later. The resulting method is summarized in Algorithm~\ref{alg:projection-induced-CG--Steihaug}.

\begin{algorithm}[H]
	\caption{A Projection-Induced CG--Steihaug Algorithm for \eqref{TR}}
	\label{alg:projection-induced-CG--Steihaug}
	\begin{algorithmic}[1]
		\State \textbf{Input:} Tolerance $\mathrm{tol}_1,\mathrm{tol}_2 > 0$, initial point $\bm{p}_0 = \bm{0}_{\mathcal{E}}$.
		\State Set $j=0$ and $\bm{n}_0=\grad J(\bm{x}^k)$. Compute $\bm{y}_0=\bm{P}_{\bm{x}^k}^B\left(\bm{P}_{\bm{x}^k}^B\right)^*\bm{n}_0$, and set $\bm{d}_0=-\bm{y}_0$.
        \If{$\sqrt{(\bm n_0,\bm y_0)_{\mathcal E}} < \mathrm{tol}_1$}
		    \State Return $\bm{p} = \bm{p}_0$
		\EndIf
		\While {not converged}
		\If{$(H^k \bm{d}_j, \bm{d}_j)_{\mathcal{E}} \le 0$}
		    \State Find $\tau$ such that $\bm{p} = \bm{p}_j + \tau\bm{d}_j$ minimizes \eqref{TR} and satisfies $\norm{\bm{p}}_{\bm{x}^k} = \Delta^k$
		    \State Return $\bm{p}$
		\Else
		    \State Compute $\alpha_j = \frac{(\bm{n}_j,\bm{y}_j)_{\mathcal{E}}}{(H^k \bm{d}_j, \bm{d}_j)_{\mathcal{E}}}$
		    \State Update $\bm{p}_{j+1} = \bm{p}_j + \alpha_j \bm{d}_j$
		    \If{$\norm{\bm{p}_{j+1}}_{\bm{x}^k} \ge \Delta^k$}
		        \State Find $\tau \ge 0$ such that $\bm{p} = \bm{p}_j + \tau \bm{d}_j$ satisfies $\norm{\bm{p}}_{\bm{x}^k} = \Delta^k$
		        \State Return $\bm{p}$
		    \EndIf
		    \State Compute $\bm{n}_{j+1}=\bm{n}_j+\alpha_jH^k\bm{d}_j$
		    \State Compute $\bm{y}_{j+1}=\bm{y}_j+\alpha_j\bm{P}_{\bm{x}^k}^B\left(\bm{P}_{\bm{x}^k}^B\right)^*H^k\bm{d}_j$
		    \If{$\sqrt{(\bm n_{j+1},\bm y_{j+1})_{\mathcal E}}\le\mathrm{tol}_2\cdot\sqrt{(\bm n_0,\bm y_0)_{\mathcal E}}$}
		        \State Return $\bm{p} = \bm{p}_{j+1}$
		    \EndIf
		    \State Compute $\beta_{j+1} = \frac{( \bm{n}_{j+1}, \bm{y}_{j+1})_{\mathcal{E}}}{( \bm{n}_{j}, \bm{y}_{j})_{\mathcal{E}}}$
		    \State Update $\bm{d}_{j+1} = -\bm{y}_{j+1} + \beta_{j+1} \bm{d}_j$
		\EndIf
		\State Set $ j = j + 1 $
		\EndWhile
	\end{algorithmic}
\end{algorithm}

\noindent The initial residual test in Algorithm~\ref{alg:projection-induced-CG--Steihaug} coincides with the stopping check above \eqref{eqn:termination tr}, since
\[
\sqrt{(\bm n_0,\bm y_0)_{\mathcal E}}
=
\|(\bm P_{\bm{x}^k}^B)^*\grad J(\bm{x}^k)\|_{\mathcal E}
=
\eta_k .
\]
If this test is triggered, the outer trust-region algorithm terminates. Otherwise, Algorithm~\ref{alg:projection-induced-CG--Steihaug} proceeds with a nonzero initial residual; this is the case considered in the Cauchy-decrease analysis below.

\noindent Two features of Algorithm~\ref{alg:projection-induced-CG--Steihaug} deserve emphasis. First, all iterates and search directions remain in \(\T_{\bm{x}^k}\mathcal M^B\): each \(\bm y_j\) lies in the range of \(\bm P_{\bm{x}^k}^B\), and hence so does every \(\bm d_j\) and \(\bm p_j\). Second, no tangent-space Riesz representation solve is required: each iteration involves only applications of \(\bm P_{\bm{x}^k}^B\), \(\left(\bm P_{\bm{x}^k}^B\right)^*\), and \(H^k\), while the residual norm and the CG coefficients retain their exact variational meaning through the projection-induced metric \(\widetilde g_{\bm{x}^k}\).

\begin{remark}
    The projection-induced realization does not alter the trust-region subproblem~\eqref{TR}: the model \(m_k\), the radius \(\Delta^k\), the constraint norm, and the retraction are those of the original Riemannian trust-region framework. Only the inner solver changes: it generates its own sequence from the projection-induced residual \(\bm y_j=\bm P_{\bm{x}^k}^B\left(\bm P_{\bm{x}^k}^B\right)^*\bm n_j\), whose norm is measured in \(\widetilde g_{\bm{x}^k}\). In general, the trial steps \(\bm p_j\) produced here therefore differ from those of Algorithm~\ref{alg:CG--Steihaug}, and no identification of the two inner sequences is used in the analysis below. Moreover, if the outer iterates remain in a bounded subset \(X\subset\mathcal M^B\), Proposition~\ref{prop:projection induced metric gradient} shows that \(\widetilde g_{\bm{x}^k}\) is uniformly equivalent to the ambient-induced Riemannian metric on \(X\) in the sense of induced norms; this uniform equivalence is the only property of the inner solver invoked in the convergence analysis. Consequently, Algorithm~\ref{alg:projection-induced-CG--Steihaug} requires no modification of the outer trust-region framework.
\end{remark}

\noindent We assume the existence of a bounded set \(X\subset\mathcal{M}^B\subset\mathcal{E}\) containing the sequence \(\{\bm{x}^k\}_{k\ge0}\) generated by Algorithm~\ref{alg:trust-region}. This assumption is verified in Section~\ref{subsec:tr-convergence} by the monotonicity of \(J^B(\bm{x}^k)\) and the coercivity of \(J\).

\begin{proposition}\label{prop:Cauchy decrease}
At iteration \(k\) of Algorithm~\ref{alg:trust-region}, let $\bm{G}_k=\bm{P}_{\bm{x}^k}^B\left(\bm{P}_{\bm{x}^k}^B\right)^*\grad J(\bm{x}^k)$ be the projection-induced gradient at \(\bm{x}^k\). Assume that the stopping test $\eta_k<\mathrm{tol}_1$ in Algorithm~\ref{alg:trust-region} is not triggered.

\noindent The first search direction of Algorithm~\ref{alg:projection-induced-CG--Steihaug} generates the Cauchy step $\bm{p}=-\tau^k\bm{G}_k$, where
\begin{equation}\label{eqn:tauk}
    \tau^k=\left\{\begin{aligned}
        &\frac{\Delta^k}{\norm{\bm{G}_k}_{\mathcal{E}},}&\text{if }\left(H^k[\bm{G}_k],\bm{G}_k\right)_{\mathcal{E}}\le 0,\\
        &\min\left\{\frac{\Delta^k}{\norm{\bm{G}_k}_{\mathcal{E}}},\frac{\norm{\left(\bm{P}_{\bm{x}^k}^B\right)^*\grad J(\bm{x}^k)}_{\mathcal{E}}^2}{\left(H^k\left[\bm{G}_k\right],\bm{G}_k\right)_{\mathcal{E}}}\right\},&\text{otherwise}.
    \end{aligned}
    \right.
\end{equation}
This step satisfies the Cauchy point decrease condition
\begin{align*}
    {m}_{k}(\bm{0}_{\mathcal{E}})-{m}_{k}(\bm{p})&\ge \frac{1}{2K_B^2}\norm{\left(\bm{P}_{\bm{x}^k}^B\right)^*\grad J(\bm{x}^k)}_{\mE}\min\left(\Delta^k,\frac{\norm{\left(\bm{P}_{\bm{x}^k}^B\right)^*\grad J(\bm{x}^k)}_{\mathcal{E}}}{\norm{\hess J^B_{\bm{x}^k}(\bm{0}_{\mathcal{E}})}}\right)\\
    &\ge\frac{1}{2K_B^2}\norm{\grad J^B(\bm{x}^k)}_{\mE}\min\left(\Delta^k,\frac{\norm{\grad J^B(\bm{x}^k)}_{\mathcal{E}}}{\norm{\hess J^B_{\bm{x}^k}(\bm{0}_{\mathcal{E}})}}\right).
\end{align*}
\end{proposition}

\begin{proof}
    In the first step, the search direction is $-\bm{y}_0=-\bm{G}_k$. The step size $\tau^k$ in \eqref{eqn:tauk} is readily verified. Using Lemma \ref{lem:PxBstar gradient identity}, the model difference is
    \begin{align*}
         m_k(\bm{p})-m_k(\bm{0}_{\mathcal{E}})&=\left(\grad J^B(\bm{x}^k),-\tau^k\bm{G}_k\right)_{\bm{x}^k} + \frac{(\tau^k)^2}{2}\left(H^k[\bm{G}_k],\bm{G}_k\right)_{\mathcal{E}}\\
         &=-\tau^k\norm{\left(\bm{P}_{\bm{x}^k}^B\right)^*\grad J(\bm{x}^k)}^2_{\mathcal{E}} + \frac{(\tau^k)^2}{2}\left(H^k[\bm{G}_k],\bm{G}_k\right)_{\mathcal{E}}.
    \end{align*}
    We consider the following three cases:

    \noindent\textbf{Case I:} $\left(H^k[\bm{G}_k],\bm{G}_k\right)_{\mathcal{E}}\le 0$. By Lemma \ref{lem:PxB upper bound}, we have
    \begin{equation*}
        m_k(\bm{p})-m_k(\bm{0}_{\mathcal{E}}) \le -\frac{\Delta^k\norm{\left(\bm{P}_{\bm{x}^k}^B\right)^*\grad J(\bm{x}^k)}^2_{\mathcal{E}}}{\norm{\bm{G}_k}_{\mathcal{E}}} \le -\frac{1}{K_B}\norm{\left(\bm{P}_{\bm{x}^k}^B\right)^*\grad J(\bm{x}^k)}_{\mathcal{E}}\Delta^k.
    \end{equation*}

    \noindent\textbf{Case II:} $\left(H^k[\bm{G}_k],\bm{G}_k\right)_{\mathcal{E}}> 0$ and $\frac{\Delta^k}{\norm{\bm{G}_k}_{\mathcal{E}}}\ge\frac{\norm{\left(\bm{P}_{\bm{x}^k}^B\right)^*\grad J(\bm{x}^k)}_{\mathcal{E}}^2}{\left(H^k\left[\bm{G}_k\right],\bm{G}_k\right)_{\mathcal{E}}}$. Substituting the minimizer and applying Theorem \ref{thm: taylor}, Lemma \ref{lem:PxB upper bound} yields
    \begin{multline*}
        m_k(\bm{p})-m_k(\bm{0}_{\mathcal{E}})
        =-\frac{1}{2} \frac{\norm{\left(\bm{P}_{\bm{x}^k}^B\right)^*\grad J(\bm{x}^k)}^4_{\mathcal{E}}}{\left(H^k\left[\bm{G}_k\right],\bm{G}_k\right)_{\mathcal{E}}}=-\frac{1}{2} \frac{\norm{\left(\bm{P}_{\bm{x}^k}^B\right)^*\grad J(\bm{x}^k)}^4_{\mathcal{E}}}{\left(\hess J^B_{\bm{x}^k}(\bm{0}_{\mathcal{E}})\left[\bm{G}_k\right],\bm{G}_k\right)_{\mathcal{E}}}\\
        \le -\frac{1}{2}\frac{\norm{\left(\bm{P}_{\bm{x}^k}^B\right)^*\grad J(\bm{x}^k)}^4_{\mathcal{E}}}{\norm{\hess J^B_{\bm{x}^k}(\bm{0}_{\mathcal{E}})}\norm{\bm{G}_k}_{\mathcal{E}}^2} \le -\frac{1}{2K_B^2}\frac{\norm{\left(\bm{P}_{\bm{x}^k}^B\right)^*\grad J(\bm{x}^k)}^2_{\mathcal{E}}}{\norm{\hess J^B_{\bm{x}^k}(\bm{0}_{\mathcal{E}})}}.
    \end{multline*}

    \noindent\textbf{Case III:} $\left(H^k[\bm{G}_k],\bm{G}_k\right)_{\mathcal{E}}> 0$ and $\frac{\Delta^k}{\norm{\bm{G}_k}_{\mathcal{E}}}<\frac{\norm{\left(\bm{P}_{\bm{x}^k}^B\right)^*\grad J(\bm{x}^k)}_{\mathcal{E}}^2}{\left(H^k\left[\bm{G}_k\right],\bm{G}_k\right)_{\mathcal{E}}}$. The step is limited by $\Delta^k$. Using Lemma \ref{lem:PxB upper bound}, we obtain
    \begin{align*}
        m_k(\bm{p})-m_k(\bm{0}_{\mathcal{E}})
        &\le-\frac{\Delta^k\norm{\left(\bm{P}_{\bm{x}^k}^B\right)^*\grad J(\bm{x}^k)}^2_{\mathcal{E}}}{\norm{\bm{G}_k}_{\mathcal{E}}}+\frac{1}{2}\frac{\Delta^k\norm{\left(\bm{P}_{\bm{x}^k}^B\right)^*\grad J(\bm{x}^k)}_{\mathcal{E}}^2}{\norm{\bm{G}_k}_{\mathcal{E}}}\\
        &=-\frac{1}{2}\frac{\Delta^k\norm{\left(\bm{P}_{\bm{x}^k}^B\right)^*\grad J(\bm{x}^k)}^2_{\mathcal{E}}}{\norm{\bm{G}_k}_{\mathcal{E}}} \le -\frac{1}{2K_B}\norm{\left(\bm{P}_{\bm{x}^k}^B\right)^*\grad J(\bm{x}^k)}_{\mathcal{E}}\Delta^k.
    \end{align*}
    
   \noindent Collecting these results and recalling that $K_B\ge 1$, we conclude that
    \begin{align*}
        &m_k(\bm{0}_{\mathcal{E}})-m_k(\bm{p})\\
        \ge& \min \left\{\frac{\Delta^k}{K_B}\norm{\left(\bm{P}_{\bm{x}^k}^B\right)^*\grad J(\bm{x}^k)}_{\mathcal{E}},\ \frac{\norm{\left(\bm{P}_{\bm{x}^k}^B\right)^*\grad J(\bm{x}^k)}^2_{\mathcal{E}}}{2K_B^2\norm{\hess J^B_{\bm{x}^k}(\bm{0}_{\mathcal{E}})}},\ \frac{\Delta^k}{2K_B}\norm{\left(\bm{P}_{\bm{x}^k}^B\right)^*\grad J(\bm{x}^k)}_{\mathcal{E}}\right\}\\
        \ge &\frac{1}{2K_B^2}\norm{\left(\bm{P}_{\bm{x}^k}^B\right)^*\grad J(\bm{x}^k)}_{\mE}\min\left(\Delta^k,\frac{\norm{\left(\bm{P}_{\bm{x}^k}^B\right)^*\grad J(\bm{x}^k)}_{\mathcal{E}}}{\norm{\hess J^B_{\bm{x}^k}(\bm{0}_{\mathcal{E}})}}\right).
    \end{align*}
    This completes the proof.
\end{proof}

\noindent Having established the Cauchy decrease condition for the initial step, we now analyze the behavior of the model during the subsequent iterations of the projection-induced CG--Steihaug Algorithm. To this end, we first provide an exact expansion of the local quadratic model.

\begin{lemma}\label{lem:model_expansion}
    For the sequence generated by Algorithm~\ref{alg:projection-induced-CG--Steihaug}, evaluating the local quadratic model \(m_k\) at \(\bm{p}_{j+1} = \bm{p}_j + \alpha_j \bm{d}_j\) yields the exact expansion
    \[
        m_k(\bm{p}_{j+1}) = m_k(\bm{p}_j) + \alpha_j (\bm{n}_j, \bm{d}_j)_{\mathcal{E}} + \frac{1}{2} \alpha_j^2 (H^k \bm{d}_j, \bm{d}_j)_{\mathcal{E}}.
    \]
\end{lemma}

\begin{proof}
    By \eqref{eqn:nj direct}, $\bm n_j=\grad J(\bm{x}^k)+H^k\bm p_j$. Substituting \(\bm p_{j+1}=\bm p_j+\alpha_j\bm d_j\) into \eqref{eqn:model ambient representation} and using the symmetry of \(H^k\), we obtain
    \[
    \begin{aligned}
        m_k(\bm{p}_{j+1})
        &= m_k(\bm{p}_j)
        + \alpha_j \left( \grad J(\bm{x}^k) + H^k \bm{p}_j, \bm{d}_j \right)_{\mathcal{E}}
        + \frac{1}{2} \alpha_j^2 \left( H^k \bm{d}_j, \bm{d}_j \right)_{\mathcal{E}}\\
        &=m_k(\bm{p}_j)+\alpha_j(\bm n_j,\bm d_j)_{\mathcal E}
        +\frac12\alpha_j^2(H^k\bm d_j,\bm d_j)_{\mathcal E}.
    \end{aligned}
    \]
    This is the claimed expansion.
\end{proof}

\noindent To evaluate the change in the model at each iteration, we also need to establish a fundamental orthogonality property of the sequence generated by the algorithm.

\begin{lemma}\label{lem:cg_orthogonality}
    For the sequence generated by the Algorithm \ref{alg:projection-induced-CG--Steihaug}, the current residual \(\bm{n}_j\) is orthogonal to the previous search direction \(\bm{d}_{j-1}\), i.e., \((\bm{n}_j, \bm{d}_{j-1})_{\mathcal{E}} = 0\) for all \(j \ge 1\).
\end{lemma}

\begin{proof}
    For \(j=1\), using \(\bm n_1=\bm n_0+\alpha_0H^k\bm d_0\) and \(\bm d_0=-\bm y_0\), we obtain
    \[
    (\bm n_1,\bm d_0)_{\mathcal E}=(\bm n_0,\bm d_0)_{\mathcal E}+\alpha_0(H^k\bm d_0,\bm d_0)_{\mathcal E}=-(\bm n_0,\bm y_0)_{\mathcal E}+(\bm n_0,\bm y_0)_{\mathcal E}=0.
    \]
    The induction step then follows as below. For $j\ge 2$, taking the inner product of \(\bm{n}_j = \bm{n}_{j-1} + \alpha_{j-1} H^k \bm{d}_{j-1}\) with \(\bm{d}_{j-1}\) gives
    \begin{equation}\label{eqn:ortho_expansion}
        (\bm{n}_j, \bm{d}_{j-1})_{\mathcal{E}} = (\bm{n}_{j-1}, \bm{d}_{j-1})_{\mathcal{E}} + \alpha_{j-1} (H^k \bm{d}_{j-1}, \bm{d}_{j-1})_{\mathcal{E}}.
    \end{equation}
    Using the search direction update \(\bm{d}_{j-1} = -\bm{y}_{j-1} + \beta_{j-1} \bm{d}_{j-2}\) and the inductive hypothesis \((\bm{n}_{j-1}, \bm{d}_{j-2})_{\mathcal{E}} = 0\), we have \((\bm{n}_{j-1}, \bm{d}_{j-1})_{\mathcal{E}} = -(\bm{n}_{j-1}, \bm{y}_{j-1})_{\mathcal{E}}\). Substituting this and the step size definition \(\alpha_{j-1} = \frac{(\bm{n}_{j-1}, \bm{y}_{j-1})_{\mathcal{E}}}{(H^k \bm{d}_{j-1}, \bm{d}_{j-1})_{\mathcal{E}}}\) into \eqref{eqn:ortho_expansion} yields
    \[
        (\bm{n}_j, \bm{d}_{j-1})_{\mathcal{E}} = -(\bm{n}_{j-1}, \bm{y}_{j-1})_{\mathcal{E}} + (\bm{n}_{j-1}, \bm{y}_{j-1})_{\mathcal{E}} = 0.
    \]
    This completes the proof.
\end{proof}

\noindent With the exact model expansion and the orthogonality property in hand, we are now positioned to prove that the local quadratic model strictly decreases at each iteration before termination.

\begin{proposition}\label{prop:cg_strict_decrease}
    Let \(\{\bm{p}_j\}\) be the sequence of iterates generated by Algorithm \ref{alg:projection-induced-CG--Steihaug}. For any iteration \(j\) where the algorithm does not terminate or truncate (that is, \((H^k \bm{d}_j, \bm{d}_j)_{\mathcal{E}} > 0\), the trial step satisfies \(\|\bm p_{j+1}\|_{\bm{x}^k}<\Delta^k\), and the residual stopping test has not terminated the inner iteration), the local quadratic model \(m_k\) strictly decreases, {\it i.e.}, we have
    \begin{equation}\label{eqn:picg-decrease}
        m_k(\bm{p}_{j+1}) < m_k(\bm{p}_j).
    \end{equation}
\end{proposition}

\begin{proof}
    By Lemma \ref{lem:model_expansion}, the change in the quadratic model is given by
    \begin{equation}\label{eqn:quadratic model difference}
         m_k(\bm{p}_{j+1}) - m_k(\bm{p}_j) = \alpha_j (\bm{n}_j, \bm{d}_j)_{\mathcal{E}} + \frac{1}{2} \alpha_j^2 (H^k \bm{d}_j, \bm{d}_j)_{\mathcal{E}}.       
    \end{equation}
    For \(j=0\), \(\bm d_0=-\bm y_0\), while for \(j\ge1\), \(\bm d_j=-\bm y_j+\beta_j\bm d_{j-1}\) and Lemma~\ref{lem:cg_orthogonality} give
    \begin{equation}\label{eqn:nj=yj}
        (\bm n_j,\bm d_j)_{\mathcal E}
        =
        -(\bm n_j,\bm y_j)_{\mathcal E}.
    \end{equation}
    Moreover, it holds that
    \begin{equation*}
        (\bm n_j,\bm y_j)_{\mathcal E}
        =
        \left\|\left(\bm P_{\bm{x}^k}^B\right)^*\bm n_j\right\|_{\mathcal E}^2
        \ge0.
    \end{equation*}
    If this quantity were zero, the corresponding residual norm in Algorithm~\ref{alg:projection-induced-CG--Steihaug} would be zero and the algorithm would have stopped either at the initial test (\(j=0\)) or at the residual test in the preceding inner iteration (\(j\ge1\)). Hence, for every nonterminated iteration under consideration,
    \((\bm n_j,\bm y_j)_{\mathcal E}>0\). Substituting the identity \eqref{eqn:nj=yj} and $\alpha_j=\frac{(\bm n_j,\bm y_j)_{\mathcal E}}{(H^k\bm d_j,\bm d_j)_{\mathcal E}}$ into the model expansion \eqref{eqn:quadratic model difference} yields
    \[
        m_k(\bm p_{j+1})-m_k(\bm p_j)
        =
        -\frac12\alpha_j(\bm n_j,\bm y_j)_{\mathcal E}<0,
    \]
    because \((H^k\bm d_j,\bm d_j)_{\mathcal E}>0\) implies \(\alpha_j>0\).
\end{proof}

\begin{proposition}\label{prop:cg_truncation_decrease}
    Suppose that the initial residual test in Algorithm~\ref{alg:projection-induced-CG--Steihaug} is not triggered. If the algorithm terminates at an inner iteration by either the nonpositive-curvature condition or the trust-region boundary condition, then the returned truncated step \(\bm p=\bm p_j+\tau\bm d_j\) satisfies
    \[
    m_k(\bm{p}) < m_k(\bm{p}_j).
    \]
\end{proposition}

\begin{proof}
    For the current search direction, define the univariate function
    \[
    q(t) := m_k(\bm p_j + t \bm d_j), \quad t \ge 0.
    \] 
    Using Lemma~\ref{lem:model_expansion} with \(t\) in place of \(\alpha_j\), we have
        \[
        q(t)
        =
        m_k(\bm p_j)
        +
        t(\bm n_j,\bm d_j)_{\mathcal E}
        +
        \frac12t^2(H^k\bm d_j,\bm d_j)_{\mathcal E}.
    \]
    By \eqref{eqn:nj=yj}, we have
    \[
        q'(0)=(\bm n_j,\bm d_j)_{\mathcal E}=-(\bm n_j,\bm y_j)_{\mathcal E}<0 .
    \]
    If \((H^k\bm d_j,\bm d_j)_{\mathcal E}\le0\), then \(q'(t)\le q'(0)<0\) for all \(t\ge0\), so \(q\) is strictly decreasing on \([0,\infty)\). Since \(\bm p_j\) is still inside the trust region before termination, the boundary parameter satisfies \(\tau>0\), and hence \(m_k(\bm p)=q(\tau)<q(0)=m_k(\bm p_j)\).

    \noindent If the algorithm terminates because \(\|\bm p_{j+1}\|_{\bm{x}^k}\ge\Delta^k\), then \((H^k\bm d_j,\bm d_j)_{\mathcal E}>0\) and $\alpha_j=-\frac{q'(0)}{(H^k\bm d_j,\bm d_j)_{\mathcal E}}$ is the minimizer of \(q\) on the search line. The boundary intersection satisfies \(0<\tau\le\alpha_j\), because \(\bm p_j\) is feasible and \(\bm p_{j+1}=\bm p_j+\alpha_j\bm d_j\) leaves the trust region. Therefore \(q'(t)<0\) on \([0,\tau)\), and again \(m_k(\bm p)=q(\tau)<q(0)=m_k(\bm p_j)\).
\end{proof}

\begin{remark}
    At any outer iteration for which the stopping criterion is not satisfied, Algorithm~\ref{alg:projection-induced-CG--Steihaug} is called with nonzero initial projection-induced residual. Its first search direction generates the Cauchy point \(\bm{p} = -\tau^k \bm{G}_k\). Moreover, by Proposition~\ref{prop:cg_strict_decrease}, each subsequent nonterminated inner iteration strictly decreases the model. Hence the trust-region step \(\bm p^k\) returned by Algorithm~\ref{alg:projection-induced-CG--Steihaug} satisfies \(m_k(\bm p^k)\le m_k(\bm p)\). Consequently, it holds that
    \begin{multline*}
        m_k(\bm{0}_{\mathcal{E}})-m_k(\bm{p}^k)\ge m_k(\bm{0}_{\mathcal{E}})-m_k(\bm{p})\\
        \ge\frac{1}{2K_B^2}\norm{\left(\bm{P}_{\bm{x}^k}^B\right)^*\grad J(\bm{x}^k)}_{\mE}\min\left(\Delta^k,\frac{\norm{\left(\bm{P}_{\bm{x}^k}^B\right)^*\grad J(\bm{x}^k)}_{\mathcal{E}}}{\norm{\hess J^B_{\bm{x}^k}(\bm{0}_{\mathcal{E}})}}\right).
    \end{multline*}
    Thus, Algorithm~\ref{alg:projection-induced-CG--Steihaug} returns a step satisfying the Cauchy decrease condition with \(c_1=\frac{1}{2K_B^2}\).
\end{remark}

\subsection{Convergence Analysis for Algorithm \ref{alg:trust-region}}
\label{subsec:tr-convergence}

\noindent In this section, we establish the convergence of Algorithm \ref{alg:trust-region}. Let $\bm{x}^0\in \mathcal{M}^B$ denote the initial point. We define the sublevel set $X$ as
\begin{equation*}
	X := \left\{\bm{x} \in \mathcal{M}^B \mid J^B(\bm{x}) \le J^B(\bm{x}^0)\right\}\subset \mathcal{M}^B\subset\mathcal{M}.
\end{equation*}
Given the coercivity of $J$, it follows that $X$ is bounded. Consequently, without loss of generality, we identify this set $X$ with the same one discussed in Section \ref{sec:properties}. Furthermore, we note that $J^B$ is bounded from below on $\mathcal{M}^B$.

\noindent By construction, the sequence $\{\bm{x}^k\}_{k\ge 0}$ generated by Algorithm \ref{alg:trust-region} satisfies the monotonicity property $J^B(\bm{x}^{k+1}) \le J^B(\bm{x}^k)$. Consequently, we conclude that $\{\bm{x}^k\}_{k\ge 0} \subset X$. Based on the properties established in Section \ref{sec:properties}, we derive the following convergence result:

\begin{theorem}
	Let \(\{\bm{x}^k\}_{k\ge 0}\) be a sequence of iterates generated by Algorithm \ref{alg:trust-region} with \(\tilde{r} \in \Big[0, \frac{1}{4}\Big)\). Suppose that all $\bm{p}^k$'s obtained in Algorithm \ref{alg:trust-region} satisfy the Cauchy decrease inequality 
	\begin{equation*}
	    {m}_{k}(\bm{0}_{\mathcal{E}})-{m}_{k}(\bm{p}^k)\ge c_1\norm{\grad J^B(\bm{x}^k)}_{\mE}\min\left(\Delta^k,\frac{\norm{\grad J^B(\bm{x}^k)}_{\mathcal{E}}}{\norm{\hess J^B_{\bm{x}^k}(\bm{0}_{\mathcal{E}})}}\right)
	\end{equation*}
	for some positive $c_1$. We then have
	\[
	\liminf_{k\to\infty}\norm{\grad J^B(\bm{x}^k)}_{\mathcal{E}}=0.
	\]
\end{theorem}
\begin{proof}
By construction of Algorithm~\ref{alg:trust-region}, the sequence \(\{J^B(\bm{x}^k)\}\) is nonincreasing. Hence, all iterates remain in the bounded sublevel set \(X\), and the uniform estimates established in Section~\ref{sec:properties} apply.

\noindent The Cauchy decrease condition is assumed in the statement. By Theorem~\ref{thm:radially lipschitz}, \(J_{\bm{x}}^B\) is radially Lipschitz-\(C^1\) on \(X\) with constants \(\beta_{RL}\) and \(\Delta_{RL}\). Moreover, Theorem~\ref{thm:boundedness of DJ and D^2J} implies that, for some \(\beta_H>0\), we have
\[
    \left|
    \left(
        \hess J_{\bm{x}}^B(\bm{0}_{\mathcal E})[\bm v],\bm v
    \right)_{\mathcal E}
    \right|
    \le
    \beta_H\|\bm v\|_{\mathcal E}^2,
    \qquad
    \forall\,\bm{x}\in X .
\]
Therefore, setting \(\varrho:=\max\{\beta_H,\beta_{RL}\}\), for every \(\|\bm p^k\|_{\bm{x}^k}\le\Delta^k\le\bar\Delta\le\Delta_{RL}\), we obtain
\begin{equation}\label{eqn:mk-Jk}
    \left|
        m_k(\bm p^k)-J_{\bm{x}^k}^B(\bm p^k)
    \right|
    \le
    \frac12(\beta_H+\beta_{RL})\|\bm p^k\|_{\bm{x}^k}^2
    \le
    \varrho\|\bm p^k\|_{\bm{x}^k}^2 .
\end{equation}
Thus, the model-accuracy estimate required in the standard trust-region analysis holds uniformly.

\noindent Consequently, all hypotheses used in the proof of \cite[Theorem~7.4.2]{absil2008optimization} are satisfied, with \(f=J^B\), \(R=\bm R^B\), and \(m_k\) as in \eqref{eqn:qudratic model mk}. Since the remaining argument relies only on the estimates above and the trust-region radius update rules, it applies verbatim in the present setting. The remaining radius argument is unchanged and therefore gives the asserted liminf convergence. This completes the proof.
\end{proof}

\section{Conclusions}

\noindent In this paper, we established theoretical foundations for the optimization problem on an affine-transversal submanifold in a Hilbert space. The central viewpoint was to treat the coupled feasible set as a single geometric object rather than handling the nonlinear manifold constraint and the affine equality constraint separately. We established the Hilbert-submanifold structure of the feasible set, characterized its tangent spaces, and constructed the associated projection, retraction, and derivative tools. These results provide the theoretical basis for feasibility-preserving algorithmic design directly in function spaces and enable the systematic use of classical finite-dimensional optimization techniques for optimization problems on affine-transversal submanifolds in Hilbert spaces.

\noindent A noteworthy aspect of our work is the variable-metric interpretation of the projection-induced gradient, which enables the use of the computationally convenient non-orthogonal projection in the algorithms, without requiring any special structure or complicated procedure to construct an orthogonal projection. Under the projection-induced metric introduced in Section~\ref{sec:projection-induced gradient}, the vector \(\bm G_{\bm x}=\bm{P}_{\bm{x}}^B(\bm{P}_{\bm{x}}^B)^*\grad J(\bm{x})\) is the exact Riemannian gradient of \(J^B\). The uniform equivalence of the induced norms ensures that \(\bm G_{\bm x}\) is uniformly comparable in norm to the ambient-induced Riemannian gradient \(\grad J^B(\bm x)\), and their ambient inner product admits a uniform positive lower bound in terms of the product of their norms. The projection-induced CG--Steihaug step can also be implemented using explicit projection-induced gradient and residual formulas, without an additional tangent-space Riesz solve, while still satisfying the Cauchy decrease property needed for convergence.

\bibliographystyle{siamplainmc}
\bibliography{reference.bib}

\end{document}